\pdfoutput=1
\documentclass[english]{shinybook}
\newcommand{\Rbar}{\overline{\mathbb{R}}}

\newcommand{\1}{\mathbb{1}}
\newcommand{\N}{\mathbb{N}}

\newcommand{\R}{\mathbb{R}}

\usepackage{pgffor}
\foreach \x in {A,...,Z}{%
\expandafter\xdef\csname cal\x\endcsname{\noexpand\ensuremath{\noexpand\mathcal{\x}}}
}

\renewcommand{\implies}{\Rightarrow}
\newcommand{\setto}{\rightrightarrows}
\newcommand{\equivalent}{\Leftrightarrow}

\newcommand{\eps}{\varepsilon}
\renewcommand{\phi}{\varphi}

\newcommand{\dom}{\operatorname{\mathrm{dom}}}
\renewcommand{\ker}{\operatorname{\mathrm{ker}}}
\newcommand{\rg}{\operatorname{\mathrm{ran}}}
\newcommand{\graph}{\operatorname{\mathrm{graph}}}
\newcommand{\epi}{\operatorname{\mathrm{epi}}}
\newcommand{\sign}{\operatorname{\mathrm{sign}}}
\newcommand{\Id}{\mathrm{Id}}
\newcommand{\prox}{\mathrm{prox}}
\newcommand{\proj}{\mathrm{proj}}
\DeclareMathOperator{\co}{\mathrm{co}}
\DeclareMathOperator*{\argmin}{\mathrm{arg\,min}}

\newtheorem{theorem}{Theorem}[chapter]
\newtheorem{lemma}[theorem]{Lemma}
\newtheorem{cor}[theorem]{Corollary}

\theoremstyle{definition}

\newtheorem{example}[theorem]{Example}
\newtheorem*{remark}{Remark}

\newcommand{\norm}[1]{\|#1\|}
\newcommand{\inner}[1]{\left(#1\right)}
\newcommand{\dual}[1]{\langle #1 \rangle}
\newcommand{\setof}[2]{\left\{#1:#2\right\}}
\newcommand{\weakto}{\rightharpoonup}

\usepackage{enumerate}

\title{Nonsmooth Analysis and Optimization}
\subtitle{Lecture notes}
\author{Christian Clason}
\publishers{\small\email{c.clason@uni-graz.at}\\\url{https://homepage.uni-graz.at/c.clason}}
\date{%
    \today
    \\
    {\small\sffamily\textsc{arxiv}:\,\href{https://arxiv.org/abs/1708.04180v3}{\nolinkurl{1708.04180v3}}}
}
\hypersetup{
    pdftitle={Nonsmooth Analysis and Optimization},
    pdfauthor={Christian Clason},
    pdfkeywords={Convex analysis, proximal point methods, Clarke subdifferential, semi-smooth Newton methods}
}

\begin{document}
\maketitle

\frontmatter

\tableofcontents

\mainmatter

\chapter*{Introduction}
\markboth{Introduction}{Introduction}

Optimization is concerned with finding solutions to problems of the form
\begin{equation*}
    \min_{x\in U} F(x)
\end{equation*}
for a function $F:X\to\R$ and a set $U\subset X$. Specifically, one considers the following questions:
\begin{enumerate}
    \item Does this problem admit a solution, i.e., is there an $\bar x\in U$ such that
        \begin{equation*}
            F(\bar x)\leq F(x)  \qquad\text{for all }x\in U?
        \end{equation*}
    \item Is there an intrinsic characterization of $\bar x$, i.e., one not requiring comparison with all other $x\in U$?
    \item How can this $\bar x$ be computed (efficiently)?
\end{enumerate}
For $U\subset \R^n$, these questions can be answered roughly as follows.
\begin{enumerate}
    \item If $U$ is compact and $F$ is continuous, the Weierstraß Theorem yields that $F$ attains its minimum at a point $\bar x \in U$ (as well as its maximum).
    \item If $F$ is differentiable, the \emph{Fermat principle}
        \begin{equation*}
            0 =  F'(\bar x)
        \end{equation*}
        holds.
    \item If $F$ is continuously differentiable and $U$ is open, one can apply the \emph{steepest descent} or gradient method to compute an $\bar x$ satisfying the Fermat principle:
        Choosing a starting point $x^0$ and setting
        \begin{equation*}
            x^{k+1} = x^k - t_k F'(x^k),\qquad k=0,\dots,
        \end{equation*}
        for suitable step sizes $t_k$, we have that $x^k \to \bar x$ for $k\to\infty$.

        If $F$ is even twice continuously differentiable, one can apply Newton's method to the Fermat principle: Choosing a suitable starting point $x^0$ and setting
        \begin{equation*}
            x^{k+1} = x^k - F''(x^k)^{-1} F'(x^k), \qquad k=0,\dots,
        \end{equation*}
        we have that $x^k\to \bar x$ for $k\to\infty$.
\end{enumerate}

However, there are many practically relevant functions that are \emph{not} differentiable, such as the absolute value or maximum function. The aim of nonsmooth analysis is therefore to find generalized derivative concepts that on the one hand allow the above sketched approach for such functions and on the other hand admit a sufficiently rich calculus to give \emph{explicit} derivatives for a sufficiently large class of functions. Here we concentrate on the two classes of
\begin{enumerate}[i)]
    \item convex functions, 
    \item locally Lipschitz continuous functions,
\end{enumerate}
which together cover a wide spectrum of applications. In particular, the first class will lead us to generalized gradient methods, while the second class are the basis for generalized Newton methods. To fix ideas, we aim at treating problems of the form
\begin{equation*}
    \min_{x\in C} \frac1p\norm{F(x)-z}_Y^p + \frac{\alpha}{q}\norm{x}_X^q
\end{equation*}
for a convex set $C\subset X$, a (possibly nonlinear but differentiable) operator $F:X\to Y$, $\alpha\geq 0$ and $p,q\in[1,\infty)$ (in particular, $p=1$ and/or $q=1$). Such problems are ubiquitous in inverse problems, imaging, and optimal control of differential equations.
Hence, we consider optimization in \emph{infinite-dimensional} function spaces; i.e., we are looking for functions as minimizers. The main benefit (beyond the frequently cleaner notation) is that the developed algorithms become \emph{discretization independent}: they can be applied to any (reasonable) finite-dimensional approximation, and the details -- in particular, the fineness -- of the approximation do not influence the convergence behavior of the algorithm.

Since we deal with infinite-dimensional spaces, some knowledge of functional analysis is assumed, but the necessary background will be summarized in \cref{chap:functan}. The results on pointwise operators on Lebesgue spaces also require elementary (Lebesgue) measure and integration theory. Basic familiarity with classical nonlinear optimization is helpful but not necessary.

\bigskip

These notes are based on graduate lectures given 2014 (in slightly different form), 2016--2020 at the University of Duisburg-Essen, and 2021 at the University of Graz. As such, no claim is made of originality (beyond possibly the selection -- and, more importantly, omission -- of material). Rather, like a magpie, I have collected the shiniest results and proofs I could find, mainly from \cite{Brokate,Schirotzek:2007,Attouch,Bauschke,Clarke:2013,Ulbrich:2002a,Schiela:2008a}. All mistakes, of course, are entirely my own.

\part{Background}

\chapter{Functional analysis}\label{chap:functan}

In this chapter we collect the basic concepts and results (and, more importantly, fix notations) from linear functional analysis that will be used in the following. For details and proofs, the reader is referred to the standard literature, e.g., \cite{Alt:2016,Brezis:2010a} and to \cite{Clason}.

\section{Normed vector spaces}

In the following, $X$ will denote a vector space over the field $\mathbb{K}$, where we restrict ourselves for the sake of simplicity to the case $\mathbb{K}=\R$.
A mapping $\norm{\cdot}:X\to \R^+:=[0,\infty)$ is called a \emph{norm} (on $X$), if for all $x\in X$ there holds
\begin{enumerate}[(i)]
    \item $\norm{\lambda x} = |\lambda| \norm{x}$ for all $\lambda\in\mathbb{K}$,
    \item $\norm{x+y} \leq \norm{x} + \norm{y}$ for all $y\in X$,
    \item $\norm{x} = 0$ if and only if $x = 0\in X$.
\end{enumerate}
\begin{example}\label{ex:functan:norm}
    \begin{enumerate}[(i)]
        \item The following mappings define norms on $X = \R^N$:
            \begin{equation*}
                \begin{aligned}
                    \norm{x}_p &= \left(\sum_{i=1}^N |x_i|^p\right)^{1/p}\qquad1\leq p<\infty,\\
                    \norm{x}_\infty &= \max_{i=1,\dots,N} |x_i|.
                \end{aligned}
            \end{equation*}
        \item The following mappings define norms on  $X = \ell^p$ (the space of real-valued sequences for which these terms are finite):
            \begin{equation*}
                \begin{aligned}
                    \norm{x}_p &= \left(\sum_{i=1}^\infty |x_i|^p\right)^{1/p}\qquad 1\leq p<\infty,\\
                    \norm{x}_\infty &= \sup_{i=1,\dots,\infty} |x_i|.
                \end{aligned}
            \end{equation*}
        \item The following mappings define norms on $X = L^p(\Omega)$ (the space of real-valued measurable functions on the domain $\Omega\subset \R^n$ for which these terms are finite):
            \begin{equation*}
                \begin{aligned}
                    \norm{u}_{L^p} &= \left(\int_\Omega |u(x)|^p\right)^{1/p}\qquad 1\leq p<\infty,\\
                    \norm{u}_{L^\infty} &= \mathop\mathrm{ess\,\sup}_{x\in \Omega} |u(x)|.
                \end{aligned}
            \end{equation*}
        \item The following mapping defines a norm on $X = C(\overline\Omega)$ (the space of continuous functions on $\overline\Omega$):
            \begin{equation*}
                \norm{u}_C = \sup_{x\in \overline\Omega} |u(x)|.
            \end{equation*}
            An analogous norm is defined on $X=C_0(\Omega)$ (the space of continuous functions on $\Omega$ with compact support), if the supremum is taken only over the space of continuous functions on $\Omega$ with compact support), if the supremum is taken only over $x\in\Omega$.
    \end{enumerate}
\end{example}
If $\norm{\cdot}$ is a norm on $X$, the tuple $(X,\norm{\cdot})$ is called a \emph{normed vector space}, and one frequently denotes this by writing $\norm{\cdot}_X$. If the norm is canonical (as in \cref{ex:functan:norm}\,(ii)--(iv)), it is often omitted and one speaks simply of ``the normed vector space $X$''.

Two norms $\norm{\cdot}_1$, $\norm{\cdot}_2$ are called \emph{equivalent} on $X$, if there are constants $c_1,c_2 >0$ such that
\begin{equation*}
    c_1 \norm{x}_2 \leq \norm{x}_1 \leq c_2 \norm{x}_2 \qquad\text{for all } x\in X.
\end{equation*}
If $X$ is finite-dimensional, all norms on $X$ are equivalent. However, the corresponding constants $c_1$ and $c_2$ may depend on the dimension $N$ of $X$; avoiding such dimension-dependent constants is one of the main reasons to consider optimization in infinite-dimensional spaces.

If $(X,\norm{\cdot}_X)$ and $(Y,\norm{\cdot}_Y)$ are normed vector spaces with $X\subset Y$, we call $X$ \emph{continuously embedded} in $Y$, denoted by $X\hookrightarrow Y$, if there exists a $C>0$ with
\begin{equation*}
    \norm{x}_Y \leq C\norm{x}_X \qquad\text{for all } x\in X.
\end{equation*}

\bigskip

We now consider mappings between normed vector spaces. In the following, let $(X,\norm{\cdot}_X)$ and $(Y,\norm{\cdot}_Y)$ be normed vector spaces, $U\subset X$, and $F: U\to Y$ be a mapping. We denote by
\begin{itemize}
    \item $\dom F  := U$ the \emph{domain of definition} of $F$;
    \item $\ker F  := \setof{x\in U}{F(x) = 0}$ \emph{kernel} or \emph{null space} of $F$;
    \item $\rg F  := \setof{F(x)\in Y}{x\in U}$ the \emph{range} of $F$;
    \item $\graph F  := \setof{(x,y)\in X\times Y}{y=F(x)}$ the \emph{graph} of $F$.
\end{itemize}
We call $F:U\to Y$
\begin{itemize}
    \item \emph{continuous} in $x\in U$, if for all $\eps>0$ there exists a $\delta >0$ with
        \begin{equation*}
            \norm{F(x)-F(z)}_Y \leq \eps\qquad \text{for all } z\in U \text{ with }  \norm{x-z}_X \leq \delta;
        \end{equation*}
    \item \emph{Lipschitz continuous}, if there exists an $L>0$ (called \emph{Lipschitz constant}) with
        \begin{equation*}
            \norm{F(x_1)-F(x_2)}_Y \leq L \norm{x_1-x_2}_X \qquad\text{for all } x_1,x_2 \in U.
        \end{equation*}
    \item \emph{locally Lipschitz continuous} in $x\in U$, if there exists a $\delta>0$ and a $L=L(x,\delta)>0$ with
        \begin{equation*}
            \norm{F(x)-F(z)}_Y \leq L \norm{x-z}_X \qquad\text{for all } z\in U \text{ with }  \norm{x-z}_X \leq \delta.
        \end{equation*}
\end{itemize}

If $T:X\to Y$ is linear, continuity is equivalent to the existence of a constant $C>0$ with
\begin{equation*}
    \norm{Tx}_Y \leq C\norm{x}_X \qquad\text{for all }x\in X.
\end{equation*}
For this reason, continuous linear mappings are called \emph{bounded}; one speaks of a bounded linear \emph{operator}.
The space $L(X,Y)$ of bounded linear operators is itself a normed vector space if endowed with the \emph{operator norm}
\begin{equation*}
    \norm{T}_{L(X,Y)}     = \sup_{x\in X\setminus\{0\}}\frac{\norm{Tx}_Y}{\norm{x}_X}
    = \sup_{\norm{x}_X= 1} \norm{Tx}_Y =   \sup_{\norm{x}_X\leq 1} \norm{Tx}_Y
\end{equation*}
(which is equal to the smallest possible constant $C$ in the definition of continuity). If $T\in L(X,Y)$ is bijective, the inverse $T^{-1}:Y\to X$ is continuous if and only if there exists a $c>0$ with
\begin{equation*}
    c\norm{x}_X \leq \norm{Tx}_Y \qquad\text{for all }x\in X.
\end{equation*}
In this case, $\norm{T^{-1}}_{L(Y,X)} = c^{-1}$ for the largest possible choice of $c$.

\section{Strong and weak convergence}

A norm directly induces a notion of convergence, the so-called \emph{strong convergence}:
A sequence $\{x_n\}_{n\in\N}\subset X$ converges (strongly in $X$) to a $x\in X$, denoted by $x_n\to x$, if
\begin{equation*}
    \lim_{n\to\infty} \norm{x_n-x}_X = 0.
\end{equation*}

A subset $U\subset X$ is called
\begin{itemize}
    \item \emph{closed}, if for every convergent sequence $\{x_n\}_{n\in\N}\subset U$ the limit $x\in U$ as well;
    \item \emph{compact}, if every sequence $\{x_n\}_{n\in\N}\subset U$ contains a convergent subsequence $\{x_{n_k}\}_{k\in\N}$ with limit $x\in U$.
\end{itemize}
A mapping $F:X\to Y$ is continuous if and only if $x_n\to x$ implies $F(x_n)\to F(x)$, and \emph{closed}, if $x_n\to x$ and $F(x_n)\to y$ imply $F(x) = y$ (i.e., $\graph F\subset X\times Y$ is a closed set).

Further we define for later use for $x\in X$ and $r>0$
\begin{itemize}
    \item the \emph{open ball} $O_r(x) := \setof{z\in X}{\norm{x-z}_X< r}$ and
    \item the \emph{closed ball} $K_r(x) := \setof{z\in X}{\norm{x-z}_X\leq r}$.
\end{itemize}
The closed ball around $0\in X$ with radius $1$ is also referred to a the \emph{unit ball} $B_X$.
A set $U\subset X$ is called
\begin{itemize}
    \item \emph{open}, if for all $x\in U$ there exists an $r>0$ with $O_r(x)\subset U$ (i.e., all $x\in U$ are \emph{interior points} of $U$, which together form the \emph{interior} $U^o$);
    \item \emph{bounded}, if it is contained in $K_r(0)$ for a $r>0$;
    \item \emph{convex}, if for any $x,y\in U$ and $\lambda\in[0,1]$ also $\lambda x + (1-\lambda)y\in U$.
\end{itemize}
In normed vector spaces it always holds that the complement of an open set is closed and vice versa (i.e., the closed sets in the sense of topology are exactly the (sequentially) closed set as defined above). The definition of a norm directly implies that both open and closed balls are convex.

A normed vector space $X$ is called \emph{complete} if every Cauchy sequence in $X$ is convergent; in this case, $X$ is called a \emph{Banach space}.
All spaces in \cref{ex:functan:norm} are Banach spaces. If $Y$ is a Banach space, so is $L(X,Y)$ if endowed with the operator norm.
Convex subsets of Banach spaces have the following useful property which derives from the Baire Theorem.
\begin{lemma}\label{lem:functan:coreint}
    Let $X$ be a Banach space and $U\subset X$ be closed and convex. Then
    \begin{equation*}
        U^o = \setof{x\in U}{\text{for all } h\in X \text{ there is a } \delta>0      \text{ with } x+ t h \in U \text{ for all } t\in [0,\delta]}.
    \end{equation*}
\end{lemma}
The set on the right-hand side is called \emph{algebraic interior} or \emph{core}. For this reason, \cref{lem:functan:coreint} is sometimes referred to as the \enquote{core-int Lemma}. Note that the inclusion \enquote{$\subset$} always holds in normed vector spaces due to the definition of interior points via open balls.

Of particular importance to us is the special case $L(X,Y)$ for $Y=\R$, the space of \emph{bounded linear functionals} on $X$. In this case, $X^*:=L(X,\R)$ is called the \emph{dual space} (or just \emph{dual} of $X$. For $x^*\in X^*$ and $x\in X$, we set
\begin{equation*}
    \dual{x^*,x}_X := x^*(x) \in\R.
\end{equation*}
This \emph{duality pairing} indicates that we can also interpret it as $x$ acting on $x^*$, which will become important later. The definition of the operator norm immediately implies that
\begin{equation}\label{eq:functan:cs_banach}
    \dual{x^*,x}_X \leq \norm{x^*}_{X^*}\norm{x}_X\qquad\text{for all }x\in X,x^*\in X^*.
\end{equation}

In many cases, the dual of a Banach space can be identified with another known Banach space.
\begin{example}\label{ex:functan:dual}
    \begin{enumerate}[(i)]
        \item $(\R^N,\norm{\cdot}_p)^* \cong (\R^N,\norm{\cdot}_q)$ with $p^{-1}+q^{-1} = 1$, where we set $0^{-1}=\infty$ and $\infty^{-1} = 0$. The duality pairing is given by
            \begin{equation*}
                \dual{x^*,x}_{p} = \sum_{i=1}^N x^*_i x_i.
            \end{equation*}
        \item $(\ell^p)^* \cong (\ell^q)$ for $1< p < \infty$. The duality pairing is given by
            \begin{equation*}
                \dual{x^*,x}_{p} = \sum_{i=1}^\infty x^*_i x_i.
            \end{equation*}
            Furthermore, $(\ell^1)^*=\ell^\infty$, but $(\ell^\infty)^*$ is not a sequence space.
        \item Analogously, $L^p(\Omega)^* \cong L^q(\Omega)$ for $1< p < \infty$. The duality pairing is given by
            \begin{equation*}
                \dual{u^*,u}_{p} = \int_\Omega u^*(x)u(x)\,dx.
            \end{equation*}
            Furthermore, $L^1(\Omega)^*\cong L^\infty(\Omega)$, but $L^\infty(\Omega)^*$ is not a function space.
        \item $C_0(\Omega)^*\cong \mathcal{M}(\Omega)$, the space of \emph{Radon measure}; it contains among others the Lebesgue measure as well as Dirac measures $\delta_x$ for $x\in\Omega$, defined via $\delta_x(u) = u(x)$ for $u\in C_0(\Omega)$. The duality pairing is given by
            \begin{equation*}
                \dual{u^*,u}_{C} = \int_\Omega u(x)\,du^*.
            \end{equation*}
    \end{enumerate}
\end{example}

A central result on dual spaces is the Hahn--Banach Theorem, which comes in both an algebraic and a geometric version.
\begin{theorem}[Hahn--Banach, algebraic]\label{thm:hb_extension}
    Let $X$ be a normed vector space. For any $x\in X$ there exists a $x^*\in X^*$ with
    \begin{equation*}
        \norm{x^*}_{X^*} = 1 \qquad\text{and}\qquad \dual{x^*,x}_X = \norm{x}_X.
    \end{equation*}
\end{theorem}
\begin{theorem}[Hahn--Banach, geometric]\label{thm:hb_separation}
    Let $X$ be a normed vector space and $A,B\subset X$ be convex, nonempty, and disjoint.
    \begin{enumerate}[(i)]
        \item If $A$ is open, there exists an $x^*\in X^*$ and a $\lambda\in \R$ with
            \begin{equation*}
                \dual{x^*,x_1}_X < \lambda \leq \dual{x^*,x_2}_X \qquad\text{for all }x_1\in A, x_2\in B.
            \end{equation*}
        \item If $A$ is closed and $B$ is compact, there exists an $x^*\in X^*$ and a $\lambda\in \R$ with
            \begin{equation*}
                \dual{x^*,x_1}_X \leq \lambda < \dual{x^*,x_2}_X \qquad\text{for all }x_1\in A, x_2\in B.
            \end{equation*}
    \end{enumerate}
\end{theorem}
Particularly the geometric version -- also referred to as \emph{separation theorems} -- is of crucial importance in convex analysis. We will also require their following variant, which is known as \emph{Eidelheit's Theorem}.
\begin{cor}\label{lem:convex:eidelheit}
    Let $X$ be a normed vector space and $A,B\subset X$ be convex and nonempty. If the interior $A^o$ of $A$ is nonempty and disjoint with $B$, there exists an $x^*\in X^*\setminus\{0\}$ and a $\lambda\in \R$ with
    \begin{equation*}
        \dual{x^*,x_1}_X \leq \lambda \leq \dual{x^*,x_2}_X \qquad\text{for all }x_1\in A, x_2\in B.
    \end{equation*}
\end{cor}
\begin{proof}
    \Cref{thm:hb_separation}\,(i) yields the existence of $x^*$ and $\lambda$ satisfying the claim for all $x \in A^o$ (even with strict inequality, which also implies $x^*\neq 0$).
    It thus remains to show that $\dual{x^*,x}_X\leq \lambda$ also for $x\in A\setminus A^o$.
    Since $A^o$ is nonempty, there exists an $x_0\in A^o$, i.e., there is an $r>0$ with $O_r(x_0)\subset A$. The convexity of $A$ then implies that $t\tilde x + (1-t)x\in A$ for all $\tilde x \in O_r(x_0)$ and $t\in [0,1]$. Hence,
    \begin{equation*}
        tO_r(x_0) + (1-t) x = O_{tr}(tx_0 + (1-t) x) \subset A,
    \end{equation*}
    and in particular $x(t) := t x_0 + (1-t) x \in A^o$ for all $t\in (0,1)$.

    We can thus find a sequence $\{x_n\}_{n\in\N}\subset A^o$ (e.g., $x_n = x(n^{-1})$) with $x_n\to x$. Due to the continuity of $x^*\in X=L(X,\R)$ we can thus pass to the limit $n\to\infty$ and obtain
    \begin{equation*}
        \begin{split}
            \dual{x^*,x}_X = \lim_{n\to \infty} \dual{x^*,x_n}_X \leq \lambda.
            \qedhere
        \end{split}
    \end{equation*}
\end{proof}

In a certain way, a normed vector space is thus characterized by its dual.
A direct consequence of \cref{thm:hb_extension} is that the norm on a Banach space can be expressed in the manner of an operator norm.
\begin{cor}\label{cor:functan:norm_dual}
    Let $X$ be a Banach space. Then for all $x\in X$,
    \begin{equation*}
        \norm{x}_X = \sup_{\norm{x^*}_{X^*}\leq 1} |\dual{x^*,x}_X|,
    \end{equation*}
    and the supremum is attained.
\end{cor}
A vector $x\in X$ can therefore be considered as a linear and, by \eqref{eq:functan:cs_banach}, bounded functional on $X^*$, i.e., as an element of the \emph{bidual} $X^{**}:=(X^*)^*$.
The embedding $X\subset X^{**}$ is realized by the \emph{canonical injection}
\begin{equation*}
    J:X\to X^{**},\qquad \dual{Jx,x^*}_{X^{*}} := \dual{x^*,x}_X\quad\text{for all }x^*\in X^*.
\end{equation*}
Clearly, $J$ is linear; \cref{thm:hb_extension} furthermore implies that $\norm{Jx}_{X^{**}} = \norm{x}_X$.
If the canonical injection is surjective and we can thus identify $X^{**}$ with $X$, the space $X$ is called \emph{reflexive}. All finite-dimensional spaces are reflexive, as are \cref{ex:functan:norm}\,(ii) and (iii) for $1<p<\infty$ but not $\ell^1,\ell^\infty$ as well as $L^1(\Omega),L^\infty(\Omega)$ and $C(\overline\Omega)$.

\bigskip

The duality pairing induces further notions of convergence: the \emph{weak convergence} on $X$ as well as the \emph{weak-$*$ convergence} on $X^*$.
\begin{enumerate}[(i)]
    \item A sequence $\{x_n\}_{n\in\N}\subset X$ converges weakly (in $X$) to $x\in X$, denoted by  $x_n\weakto x$, if
        \begin{equation*}
            \dual{x^*,x_n}_X\to  \dual{x^*,x}_X \qquad\text{for all }x^*\in X^*.
        \end{equation*}
    \item A sequence $\{x^*_n\}_{n\in\N}\subset X^*$ converges weakly-$*$ (in $X^*$) to $x^*\in X^*$, denoted by $x^*_n\weakto^* x^*$, if
        \begin{equation*}
            \dual{x_n^*,x}_X\to  \dual{x^*,x}_X \qquad\text{for all }x\in X.
        \end{equation*}
\end{enumerate}
Weak convergence generalizes the concept of componentwise convergence in $\R^n$, which -- as can be seen from the proof of the Heine--Borel Theorem -- is the appropriate concept in the context of compactness.
Strong convergence implies weak convergence by continuity of the duality pairing; in the same way, convergence with respect to the operator norm (also called \emph{pointwise convergence}) implies weak-$*$ convergence. If $X$ is reflexive, weak and weak-$*$ convergence (both in $X=X^{**}$!) coincide. In finite-dimensional spaces, all convergence notions coincide.

If $x_n\to x$ and $x_n^*\weakto^* x^*$ or $x_n\weakto x$ and $x_n^*\to x^*$, then  $\dual{x_n^*,x_n}_X\to \dual{x^*,x}_X$. However, the duality pairing of weak(-$*$) convergent sequences does not converge in general.

As for strong convergence, one defines weak(-$*$) continuity and closedness of mappings as well as weak(-$*$) closedness and compactness of sets. The last property is of fundamental importance in optimization; its characterization is therefore a central result of this chapter.
\begin{theorem}[Eberlein--\u{S}mulyan]\label{thm:ebsmul}
    If $X$ is a normed vector space, $B_X$ is weakly compact if and only if $X$ is reflexive.
\end{theorem}
Hence in a reflexive space, all bounded and weakly closed sets are weakly compact. Note that weak closedness is a \emph{stronger} claim than closedness, since the property has to hold for more sequences. For convex sets, however, both concepts coincide.
\begin{lemma}\label{lem:convex_closed}
    A convex set $U\subset X$ is closed if and only if it is weakly closed.
\end{lemma}
\begin{proof}
    Weakly closed sets are always closed since a convergent sequence is also weakly convergent.
    Let now $U\subset X$ be convex closed and nonempty (otherwise nothing has to be shown) and consider a sequence $\{x_n\}_{n\in\N}\subset U$ with $x_n\weakto x\in X$. Assume that $x\in X\setminus U$. Then, the sets $U$ and $\{x\}$ satisfy the premise of \cref{thm:hb_separation}\,(ii); we thus find an $x^*\in X^*$ and a $\lambda\in\R$ with
    \begin{equation*}
        \dual{x^*,x_n}_X \leq \lambda < \dual{x^*,x}_X\quad \text{for all } n\in\N.
    \end{equation*}
    Passing to the limit $n\to\infty$ in the first inequality yields the contradiction
    \begin{equation*}
        \begin{split}
            \dual{x^*,x}_X < \dual{x^*,x}_X.
            \qedhere
        \end{split}
    \end{equation*}
\end{proof}

If $X$ is not reflexive (e.g., $X=L^\infty(\Omega)$), we have to turn to weak-$*$ convergence.
\begin{theorem}[Banach--Alaoglu]\label{thm:banachal}
    If $X$ is a separable normed vector space (i.e., contains a countable dense subset), $B_{X^*}$ is weakly-$*$ compact.
\end{theorem}
By the Weierstraß Approximation Theorem, both $C(\overline\Omega)$ and $L^p(\Omega)$ for $1\leq p<\infty$ are separable; also, $\ell^p$ is separable for $1\leq p <\infty$.
Hence, bounded and weakly-$*$ closed balls in  $\ell^\infty$, $L^\infty(\Omega)$, and $\mathcal{M}(\Omega)$ are weakly-$*$ compact. However, these spaces themselves are not separable.

Finally, we will also need the following \enquote{weak-$*$} separation theorem, whose proof is analogous to the proof of \cref{thm:hb_separation} (using the fact that the linear weakly-$*$ continuous functionals are exactly those of the form $x^*\mapsto \dual{x^*,x}_X$ for some $x\in X$); see also \cite[Theorem~3.4(b)]{Rudin:1991}.
\begin{theorem}\label{thm:clarke:hb}
    Let $A\subset X^*$ be a non-empty, convex, and weakly-$*$ closed subset and $x^*\in X^*\setminus A$. Then there exist an $x\in X$ and a $\lambda\in\R$ with
    \begin{equation*}
        \dual{z^*,x}_X \leq \lambda < \dual{x^*,x}_X \qquad\text{for all }z^*\in A.
    \end{equation*}
\end{theorem}
Note, however, that closed convex sets in non-reflexive spaces do \emph{not} have to be weakly-$*$ closed.

\bigskip

Since a normed vector space is characterized by its dual, this is also the case for linear operators acting on this space.
For any $T\in L(X,Y)$, the \emph{adjoint operator} $T^*\in L(Y^*,X^*)$ is defined via
\begin{equation*}
    \dual{T^*y^*,x}_X = \dual{y^*,Tx}_Y\qquad\text{for all } x\in X, y^*\in Y^*.
\end{equation*}
It always holds that $\norm{T^*}_{L(Y^*,X^*)} = \norm{T}_{L(X,Y)}$.
Furthermore, the continuity of $T$ implies that $T^*$ is weakly-$*$ continuous (and $T$ weakly continuous).

\section{Hilbert spaces}

Especially strong duality properties hold in Hilbert spaces.
A mapping $\inner{\cdot,\cdot}:X\times X\to \R$ on a vector space $X$ over $\R$ is called \emph{inner product}, if
\begin{enumerate}[(i)]
    \item $\inner{\alpha x+\beta y,z} = \alpha\inner{x,z} + \beta \inner{y,z}$ for all $x,y,z\in X$ and $\alpha,\beta\in \R$;
    \item $\inner{x,y} = \inner{y,x}$ for all $x,y\in X$;
    \item $\inner{x,x} \geq 0$ for all $x\in X$ with equality if and only if $x=0$.
\end{enumerate}
A Banach space together with an inner product $(X,\inner{\cdot,\cdot}_X)$ is called a \emph{Hilbert space}; if the inner product is canonical, it is frequently omitted, and the Hilbert space is simply denoted by $X$. An inner product induces a norm
\begin{equation*}
    \norm{x}_X := \sqrt{\inner{x,x}_X},
\end{equation*}
which satisfies the \emph{Cauchy--Schwarz inequality}:
\begin{equation*}
    \inner{x,y}_X \leq \norm{x}_X\norm{y}_X.
\end{equation*}
The spaces in \cref{ex:functan:dual}\,(i--iii) for $p=2(=q)$ are all Hilbert spaces, where the inner product coincides with the duality pairing and induces the canonical norm.

\bigskip

The relevant point in our context is that the dual of a Hilbert space $X$ can be identified with $X$ itself.
\begin{theorem}[Fréchet--Riesz]\label{thm:frechetriesz}
    Let $X$ be a Hilbert space. Then for each $x^*\in X^*$ there exists a unique $z_{x^*}\in X$ with $\norm{x^*}_{X^*} = \norm{z_{x^*}}_X$ and
    \begin{equation*}
        \dual{x^*,x}_X = \inner{x,z_{x^*}}_X \qquad\text{for all } x\in X.
    \end{equation*}
\end{theorem}
The element $z_{x^*}$ is called \emph{Riesz representation} of $x^*$.
The (linear) mapping $J_X:X^*\to X$, $x^*\mapsto z_{x^*}$, is called \emph{Riesz isomorphism}, and can be used to show that every Hilbert space is reflexive.

\Cref{thm:frechetriesz} allows to use the inner product instead of the duality pairing in Hilbert spaces. For example, a sequence $\{x_n\}_{n\in\N}\subset X$ converges weakly to $x\in X$ if and only if
\begin{equation*}
    \inner{x_n,z}_X \to \inner{x,z}_X\qquad\text{for all } z\in X.
\end{equation*}

Similar statements hold for linear operators on Hilbert spaces. For a linear operator $T\in L(X,Y)$ between Hilbert spaces $X$ and $Y$, the \emph{Hilbert space adjoint operator} $T^\star\in L(Y,X)$ is defined via
\begin{equation*}
    \inner{T^\star y,x}_X = \inner{Tx,y}_Y \qquad\text{for all } x\in X, y\in Y.
\end{equation*}
If $T^\star = T$, the operator $T$ is called \emph{self-adjoint}.
Both definitions of adjoints are related via $T^\star = J_X T^* J_Y^{-1}$.
If the context is obvious, we will not distinguish the two in notation.

\chapter{Calculus of variations}

We first consider the question about the existence of minimizers of a (nonlinear) functional $F:U\to\R$ for a subset $U$ of a Banach space $X$. Answering such questions is one of the goals of the \emph{calculus of variations}.

It is helpful to include the constraint $x\in U$ into the functional by extending $F$ to all of $X$ with the value $\infty$. We thus consider
\begin{equation*}
    \overline F:X\to\Rbar:=\R\cup \{\infty\},\qquad
    \overline F(x) =
    \begin{cases}
        F(x) & \text{if }x\in U,\\
        \infty & \text{if }x\in X\setminus U.
    \end{cases}
\end{equation*}
We extend the usual arithmetic on $\R$ to $\Rbar$ by letting $t<\infty$ and $t+\infty = \infty$ for all $t\in\R$;
subtraction and multiplication of negative numbers with $\infty$ and in particular $F(x) = -\infty$ is not allowed, however. Thus if there is any $x\in U$ at all, a minimizer $\bar x$ of $\overline F$ necessarily must lie in $U$ and coincide with a minimizer of $F$ over $U$.

We thus consider from now on functionals $F:X\to\Rbar$.
The set on which $F$ is finite is called the \emph{effective domain}
\begin{equation*}
    \dom F := \setof{x\in X}{F(x)<\infty}.
\end{equation*}
If $\dom F \neq \emptyset$, the functional $F$ is called \emph{proper}.

\section{The direct method}

We now generalize the Weierstraß Theorem (every real-valued continuous function on a compact set attains its minimum and maximum) to Banach spaces and in particular to functions of the form $\overline F$. Since we are only interested in minimizers, we only require a \enquote{one-sided} continuity: We call $F$ \emph{lower semicontinuous} in $x\in X$ if
\begin{equation*}
    F(x) \leq \liminf_{n\to\infty} F(x_n)\qquad\text{for every } \{x_n\}_{n\in\N}\subset X \text{ with  }x_n\to x.
\end{equation*}
Analogously, we define \emph{weakly(-$*$) lower semicontinuous} functionals via weakly(-$*$) convergent sequences.
Finally, $F$ is called \emph{coercive} if for every sequence $\{x_n\}_{n\in\N}\subset X$ with $\norm{x_n}_X\to\infty$ we also have $F(x_n)\to \infty$.

We now have all concepts at hand for proving the central existence result in the calculus of variations. The strategy for its proof is known as the \emph{direct method}.\footnote{This strategy is applied so often in the literature that one usually just writes \enquote{Existence of a minimizer follows from the direct method} or even just \enquote{Existence follows from standard arguments}.
The basic idea goes back to Hilbert; the version based on lower semicontinuity which we use here is due to \href{http://www-history.mcs.st-and.ac.uk/Biographies/Tonelli.html}{Leonida Tonelli} (1885--1946), who had a lasting influence on the modern calculus of variations through it.}
\begin{theorem}[direct method]\label{thm:variation:existence}
    Let $X$ be a reflexive Banach space and $F:X\to\Rbar$ be proper, coercive, and weakly lower semicontinuous. Then the minimization problem
    \begin{equation*}
        \min_{x\in X} F(x)
    \end{equation*}
    has a solution $\bar x\in \dom F$.
\end{theorem}
\begin{proof}
    The proof can be separated into three steps.
    \begin{enumerate}[(i)]
        \item \emph{Pick a minimizing sequence.}

            Since $F$ is proper, there exists an $M:=\inf_{x\in X} F(x)<\infty$ (although $M=-\infty$ is not excluded so far). Thus, by the definition of the infimum, there exists a sequence $\{y_n\}_{n\in\N}\subset \rg F\setminus\{\infty\}\subset \R$ with $y_n\to M$, i.e., there exists a sequence $\{x_n\}_{n\in\N}\subset X$ with
            \begin{equation*}
                F(x_n) \to M = \inf_{x\in X} F(x).
            \end{equation*}
            Such a sequence is called \emph{minimizing sequence}. Note that from the convergence of $\{F(x_n)\}_{n\in\N}$ we cannot conclude the convergence of $\{x_n\}_{n\in\N}$ (yet).

        \item \emph{Show that the minimizing sequence contains a weakly convergent subsequence.}

            Assume to the contrary that $\{x_n\}_{n\in\N}$ is unbounded, i.e., that $\norm{x_n}_X\to\infty$ for $n\to \infty$. The coercivity of $F$ then implies that $F(x_n)\to \infty$ as well, in contradiction to $F(x_n)\to M<\infty$ by definition of the minimizing sequence.
            Hence, the sequence is bounded, i.e., there is an $M>0$ with $\norm{x_n}_X\leq M$ for all $n\in\N$.
            In particular, $\{x_n\}_{n\in\N} \subset K_M(0)$. The \nameref{thm:ebsmul} \cref{thm:ebsmul} therefore implies the existence of a weakly converging subsequence $\{x_{n_k}\}_{k\in\N}$ with limit $\bar x\in X$. (This limit is a candidate for the minimizer.)

        \item \emph{Show that this limit is a minimizer.}

            From the definition of the minimizing sequence, we also have $F(x_{n_k})\to M$ for $k\to\infty$. Together with the weak lower semicontinuity of $F$ and the definition of the infimum we thus obtain
            \begin{equation*}
                \inf_{x\in X} F(x) \leq F(\bar x) \leq \liminf_{k\to\infty} F(x_{n_k}) = M = \inf_{x\in X} F(x)<\infty.
            \end{equation*}
            This implies that $\bar x\in \dom F$ and that $\inf_{x\in X} F(x)=F(\bar x)> -\infty$. Hence, the infimum is attained in $\bar x$ which is therefore the desired minimizer.
            \qedhere
    \end{enumerate}
\end{proof}
If $X$ is not reflexive but the dual of a separable Banach space, we can argue analogously using the \nameref{thm:banachal} \cref{thm:banachal}

Note how the topology on $X$ used in the proof is restricted in step (iii) and (iv): Step (iii) profits from a coarse topology (in which more sequences are convergent), while step (iv) profits from a fine topology (the fewer sequences are convergent, the easier it is to satisfy the $\liminf$ conditions). Since in the cases of interest to us no more than boundedness of a minimizing sequence can be expected, we cannot use a finer than the weak topology. We thus have to ask whether a sufficiently large class of (interesting) functionals are weakly lower semicontinuous.

A first example is the class of bounded linear functionals: For any $x^*\in X^*$, the functional
\begin{equation*}
    F:X\to\Rbar, \qquad x\mapsto \dual{x^*,x}_X,
\end{equation*}
is weakly continuous by definition of weak convergence and hence \emph{a fortiori} weakly lower semicontinuous.
Another advantage of (weak) lower semicontinuity is that it is preserved under certain operations.
\begin{lemma}\label{lem:variation:wlsc}
    Let $X$ and $Y$ be Banach spaces and $F:X\to\Rbar$ be weakly lower semicontinuous. Then, the following functionals are weakly lower semicontinuous as well:
    \begin{enumerate}[(i)]
        \item $\alpha F$ for all $\alpha \geq 0$;
        \item $F+G$ for $G:X\to\Rbar$ weakly lower semicontinuous;
        \item $\phi\circ F$ for $\phi:\Rbar\to\Rbar$ lower semicontinuous and increasing.
        \item $F\circ \Phi$ for $\Phi:Y\to X$ weakly continuous, i.e., $y_n\weakto y$ implies $\Phi(y_n)\weakto \Phi(y)$;
        \item $x\mapsto \sup_{i\in I} F_i(x)$ with $F_i:X\to\Rbar$ weakly lower semicontinuous for an arbitrary set $I$.
    \end{enumerate}
\end{lemma}
Note that (v) does \emph{not} hold for continuous functions.
\begin{proof}
    Statements (i) and (ii) follow directly from the properties of the limes inferior.

    For statement (iii), it first follows from the weak lower semicontinuity of $F$ and the monotonicity of $\phi$ that $x_n\weakto x$ implies
    \begin{equation*}
        \phi(F(x)) \leq \phi(\liminf_{n\to\infty} F(x_n)).
    \end{equation*}
    It remains to show that the right-hand side can be bounded by $\liminf_{n\to\infty} \phi(F(x_n))$. For that purpose, we consider the subsequence $\{\phi(F(x_{n_k})\}_{k\in\N}$ realizing the $\liminf$, i.e., for which $\liminf_{n\to\infty} \phi(F(x_n)) = \lim_{k\to\infty} \phi(F(x_{n_k}))$. By passing to a further subsequence (which we index by $k'$ for simplicity) we can also obtain that $\liminf_{k\to \infty} F(x_{n_k}) = \lim_{k'\to\infty} F(x_{n_{k'}})$. Since the $\liminf$ restricted to a subsequence can never be smaller than that of the full sequence, the monotonicity of $\phi$ together with its weak lower semicontinuity now implies that
    \begin{equation*}
        \phi(\liminf_{n\to\infty} F(x_n)) \leq \phi(\lim_{k'\to\infty} F(x_{n_{k'}}))
        \leq \liminf_{k'\to\infty} \phi(F(x_{n_{k'}}))
        = \liminf_{n\to\infty} \phi(F(x_{n})),
    \end{equation*}
    where we have used in the last step that a subsequence of a convergent sequence has the same limit (which coincides with the $\liminf$).

    Statement (iv) follows directly from the weak continuity of $\Phi$: $y_n\weakto y$ implies that $x_n := \Phi(y_n) \weakto \Phi(y) =: x$, and the lower semicontinuity of $F$ yields
    \begin{equation*}
        F(\Phi(y_n)) \leq \liminf_{n\to\infty} F(\Phi(y)).
    \end{equation*}

    Finally, let $\{x_n\}_{n\in\N}$ be a weakly converging sequence with limit $x\in X$. Then the weak lower semicontinuity of the $F_i$ together with the definition of the supremum implies that
    \begin{equation*}
        F_j(x) \leq \liminf_{n\to\infty} F_j(x_n) \leq \liminf_{n\to\infty} \sup_{i\in I} F_i(x_n) \qquad\text{for all }j\in I.
    \end{equation*}
    Taking the supremum over all $j\in I$ on both sides yields statement (v).
\end{proof}

\begin{cor}\label{cor:variation:norm}
    If $X$ is a Banach space, the norm $\norm{\cdot}_X$ is proper, coercive, and weakly lower semicontinuous.
\end{cor}
\begin{proof}
    Coercivity and $\dom \norm{\cdot}_X = X$ follow directly from the definition. Weak lower semicontinuity follows from \cref{lem:variation:wlsc}\,(v) and \cref{cor:functan:norm_dual} since
    \begin{equation*}
        \begin{split}
            \norm{x}_X = \sup_{\norm{x^*}_{X^*}\leq 1} |\dual{x^*,x}_X|.
            \qedhere
        \end{split}
    \end{equation*}
\end{proof}

Another frequently occurring functional is the \emph{indicator function}\footnote{not to be confused with the \emph{characteristic function} $\1_U$ with $\1_U(x) = 1$ for $x\in U$ and $0$ else} of a set $U\subset X$, defined as
\begin{equation*}
    \delta_U(x) = \begin{cases} 0 & x\in U,\\ \infty& x \in X\setminus U.\end{cases}
\end{equation*}
The purpose of this definition is of course to reduce the minimization of a functional $F:X\to\R$ over $U$ to the minimization of $\overline F:=F+\delta_U$ over $X$. The following result is therefore important for showing the existence of a minimizer.
\begin{lemma}
    Let $X$ be a Banach space and $U\subset X$. Then, $\delta_U$ is
    \begin{enumerate}[(i)]
        \item proper if $U$ is non-empty;
        \item weakly lower semicontinuous if $U$ is convex and closed;
        \item coercive if $U$ is bounded.
    \end{enumerate}
\end{lemma}
\begin{proof}
    Statement (i) is clear. For (ii), consider a weakly converging sequence $\{x_n\}_{n\in\N}\subset X$ with limit $x\in X$. If $x\in U$, then $\delta_U\geq 0$ immediately yields
    \begin{equation*}
        \delta_U(x) = 0 \leq \liminf_{n\to\infty} \delta_U(x_n).
    \end{equation*}
    Let now $x\notin U$. Since $U$ is convex and closed and hence by \cref{lem:convex_closed} also weakly closed, there must be an $N\in\N$ with $x_n\notin U$ for all $n\geq N$ (otherwise we could -- by passing to a subsequence if necessary -- construct a sequence with $x_{n}\weakto x\in U$, in contradiction to the assumption). Thus, $\delta_U(x_n) = \infty$ for all $n\geq N$, and therefore
    \begin{equation*}
        \delta_U(x) = \infty = \liminf_{n\to\infty} \delta_U(x_n).
    \end{equation*}

    For (iii), let $U$ be bounded, i.e., there exist an $M>0$ with $U\subset K_M(0)$.
    If $\norm{x_n}_X\to\infty$, then there exists an $N\in\N$ with $\norm{x_n}_X >M$ for all $n\geq N$, and thus $x_n\notin K_M(0)\supset U$ for all $n\geq N$. Hence, $\delta_U(x_n) \to \infty$ as well.
\end{proof}

\section{Differential calculus in Banach spaces}

To characterize minimizers of functionals on infinite-dimensional spaces using the Fermat principle, we transfer the classical derivative concepts to Banach spaces.

Let $X$ and $Y$ be Banach spaces, $F:X\to Y$ be a mapping, and $x,h\in X$ be given.
\begin{itemize}
    \item If the one-sided limit
        \begin{equation*}
            F'(x;h) := \lim_{t\to 0^+} \frac{F(x+th)-F(x)}{t}\in Y,
        \end{equation*}
        exists, it is called the \emph{directional derivative} of $F$ in $x$ in direction $h$.
    \item If $F'(x;h)$ exists for all $h\in X$ and
        \begin{equation*}
            DF(x):X\to Y,\qquad h\mapsto F'(x;h)
        \end{equation*}
        defines a bounded linear operator, we call $F$ \emph{Gâteaux differentiable} (in $x$) and $DF\in L(X,Y)$ its \emph{Gâteaux derivative}.
    \item If additionally
        \begin{equation*}
            \lim_{\norm{h}_X\to 0} \frac{\norm{F(x+h) - F(x) - DF(x)h}_Y}{\norm{h}_X} = 0,
        \end{equation*}
        then $F$ is called \emph{Fréchet differentiable} (in $x$) and $F'(x):=DF(x)\in L(X,Y)$ its \emph{Fréchet derivative}.
    \item If the mapping $x\mapsto F'(x)$ is (Lipschitz) continuous, we call $F$ \emph{(Lipschitz) continuously differentiable}.
\end{itemize}
The difference between Gâteaux and Fréchet differentiable lies in the approximation error of $F$ near $x$ by $F(x)+DF(x)h$: While it only has to be bounded in $\norm{h}_X$ -- i.e., linear in $\norm{h}_X$ -- for a Gâteaux differentiable function, it has to be superlinear in $\norm{h}_X$ if $F$ is Fréchet differentiable. (For a \emph{fixed} direction $h$, this of course also the case for Gâteaux differentiable functions; Fréchet differentiability thus additionally requires a uniformity in $h$.)

If $F$ is Gâteaux differentiable, the Gâteaux derivative can be computed via
\begin{equation*}
    DF(x) h = \left(\tfrac{d}{dt}F(x+th)\right)\Big|_{t=0}.
\end{equation*}
Bounded linear operators $F\in L(X,Y)$ are obviously Fréchet differentiable with derivative $F'(x) = F \in L(X,Y)$ for all $x\in X$.
Further derivatives can be obtained through the usual calculus, whose proof in Banach spaces is exactly as in $\R^n$. As an example, we prove a chain rule.
\begin{theorem}\label{thm:frechet_chain}
    Let $X$, $Y$, and $Z$ be Banach spaces, and let $F:X\to Y$ be Fréchet differentiable in $x\in X$ and $G:Y\to Z$ be Fréchet differentiable in $y:=F(x)\in Y$. Then, $G\circ F$ is Fréchet differentiable in $x$ and
    \begin{equation*}
        (G\circ F)'(x) = G'(F(x))\circ F'(x).
    \end{equation*}
\end{theorem}
\begin{proof}
    For $h\in X$ with $x+h\in\dom F$ we have
    \begin{equation*}
        (G\circ F)(x+h ) - (G\circ F)(x) = G(F(x+h))-G(F(x)) = G(y+g)  - G(y)
    \end{equation*}
    with $g := F(x+h)-F(x)$. The Fréchet differentiability of $G$ thus implies that
    \begin{equation*}
        \norm{(G\circ F)(x+h ) - (G\circ F)(x) - G'(y)g }_Z =  r_1(\norm{g}_Y)
    \end{equation*}
    with $r_1(t)/t \to 0$ for $t\to 0$. The Fréchet differentiability of $F$ further implies
    \begin{equation*}
        \norm{g - F'(x)h}_Y = r_2(\norm{h}_X)
    \end{equation*}
    with $r_2(t)/t \to 0$ for $t\to 0$. In particular, the reverse triangle inequality yields
    \begin{equation}\label{eq:frechet_chain:est}
        \norm{g}_Y \leq \norm{F'(x)h}_Y + r_2(\norm{h}_X).
    \end{equation}
    Hence, with $c:=\norm{G'(F(x))}_{L(Y,Z)}$ we have
    \begin{equation*}
        \norm{(G\circ F)(x+h) - (G\circ F)(x) -  G'(F(x)) F'(x)h}_Z \leq  r_1(\norm{g}_Y) +  c\, r_2(\norm{h}_X).
    \end{equation*}
    If $\norm{h}_X\to 0$, we obtain from \eqref{eq:frechet_chain:est} and $F'(x)\in L(X,Y)$ that $\norm{g}_Y\to 0$ as well, and the claim follows.
\end{proof}
A similar rule for Gâteaux derivatives does not hold, however.

We will also need the following variant of the mean value theorem.
Let $[a,b]\subset \R$ be a bounded interval and $f:[a,b]\to X$ be continuous. Then the  \emph{Bochner integral} $\int_a^b f(t)\,dt\in X$ is well-defined and by construction satisfies
\begin{equation} \label{eq:bochner}
    \left\langle x^*,\int_a^b f(t)\,dt\right\rangle_X = \int_a^b\dual{x^*,f(t)}_X\,dt \qquad\text{for all }x^*\in X^*,
\end{equation}
as well as
\begin{equation}
    \label{eq:bochner_est}
    \left\|\int_a^b f(t)\,dt\right\|_X \leq \int_a^b \norm{f(t)}_X\,dt,
\end{equation}
see, e.g., \cite[Corollary \textsc{v}.1]{Yosida}.
\begin{theorem}\label{thm:frechet:mean}
    Let $F:U\to Y$ be Fréchet differentiable, and let $y\in U$ and $h\in Y$ be given with $y+th\in U$ for all $t\in[0,1]$. Then
    \begin{equation*}
        F(y+h) - F(y) =\int_0^1 F'(y+th)h\,dt.
    \end{equation*}
\end{theorem}
\begin{proof}
    Consider for arbitrary $y^*\in Y^*$ the function
    \begin{equation*}
        f:[0,1]\to \R,\qquad t\mapsto\dual{y^*,F(y+th)}_Y.
    \end{equation*}
    From \cref{thm:frechet_chain} we obtain that $f$ (as a composition of mappings on Banach spaces) is differentiable with
    \begin{equation*}
        f'(t) = \dual{y^*,F'(y+th)h}_Y,
    \end{equation*}
    and the fundamental theorem of calculus in $\R$ yields that
    \begin{equation*}
        \dual{y^*,F(y+h)-F(y)}_Y = f(1)-f(0) = \int_0^1f'(t)\,dt = \left\langle y^*,\int_0^1 F'(y+th)h\,dt\right\rangle_Y,
    \end{equation*}
    where the last equality follows from \eqref{eq:bochner}.
    Since $y^*\in Y^*$ was arbitrary, the claim follows from this together with \cref{cor:functan:norm_dual}.
\end{proof}

We now turn to the characterization of minimizers of a differentiable functions $F:X\to\R$.\footnote{The \emph{indirect method} of the calculus of variations uses this to show existence of minimizers as well, e.g., as the solution of a partial differential equation.}
\begin{theorem}[Fermat principle]\label{thm:variation:fermat}
    Let $F:X\to\R$ be Gâteaux differentiable and $\bar x \in X$ be a minimizer of $F$. Then $DF(\bar x) = 0$, i.e.,
    \begin{equation*}
        \dual{DF(\bar x), h}_X = 0 \qquad\text{for all }h\in X.
    \end{equation*}
\end{theorem}
\begin{proof}
    Let $h\in X$ be arbitrary. Since $\bar x$ is a local minimizer, the core--int \cref{lem:functan:coreint} implies that there exists an $\epsilon >0$ such that $F(\bar x) \leq F(\bar x+th)$ for all $t\in (0,\eps)$, i.e.,
    \begin{equation}
        0\leq \frac{F(\bar x+th) - F(\bar x)}{t} \to F'(\bar x; h) = \dual{DF(\bar x),h}_X \quad\text{for }t\to 0,
    \end{equation}
    where we have used the Gâteaux differentiability and hence directional differentiability of $F$. Since the right-hand side is linear in $h$, the same argument for $-h$ yields $\dual{DF(\bar x), h}_X\leq 0$ and therefore the claim.
\end{proof}
Of particular relevance in optimization is of course the special case $F:X\to\R$, where $DF(x)\in L(X;\R)=X^*$ (if the Gâteaux derivative exists).
Note that since the Gâteaux derivative of $F:X\to\R$ is an element of $X^*$, it cannot be added to elements in $X$ (as required for, e.g., a steepest descent method). However, in Hilbert spaces (and in particular in $\R^N$), we can use the Fréchet--Riesz \cref{thm:frechetriesz} to identify $DF(x)\in X^*$ with an element $\nabla F(x) \in X$, called  the \emph{gradient} of $F$ at $x$, in a canonical way via
\begin{equation*}
    \dual{DF(x),h}_X = \inner{\nabla F(x),h}_X \qquad\text{for all } h\in X.
\end{equation*}
We illustrate this with a simple example.
\begin{example}\label{ex:variation:gradient1}
    Let $F(x) = \frac12\norm{x}_X^2=\frac12\inner{x,x}_X$. Then we have for all $x,h\in X$ that
    \begin{equation*}
        F'(x;h) = \lim_{t\to 0} \frac{\frac12\inner{x+th,x+th}_X - \frac12\inner{x,x}_X}{t} = \inner{x,h}_X = \dual{DF(x),h}_X,
    \end{equation*}
    since the inner product is linear in $h$ for fixed $x$.
    Hence, the squared norm is Gâteaux differentiable at every $x\in X$ with derivative $DF(x) = h\mapsto \inner{x,h}_X\in X^*$; it is even Fréchet differentiable since
    \begin{equation*}
        \lim_{\norm{h}_X\to 0} \frac{\left|\frac12\norm{x+h}_X^2 - \frac12\norm{x}_X^2 - (x,h)_X\right|}{\norm{h}_X} = \lim_{\norm{h}_X\to 0}\frac12 \norm{h}_X = 0.
    \end{equation*}
    The gradient $\nabla F(x) \in X$ by definition is given by
    \begin{equation*}
        \inner{\nabla F(x),h}_X = \dual{DF(x),h}_X = \inner{x,h}_X \qquad\text{for all } h\in X,
    \end{equation*}
    i.e., $\nabla F(x) = x$.
\end{example}
The following example demonstrates how the gradient (in contrast to the derivative) depends on the inner product on $X$ -- which may be different from the inner product inducing the squared norm.
\begin{example}\label{ex:variation:gradient2}
    Let $M\in L(X; X)$ be self-adjoint and positive definite (and thus continuously invertible). Then $\inner{x,y}_Z:=\inner{Mx,y}_X$ also defines an inner product on the vector space $X$ and induces an (equivalent) norm $\norm{x}_Z := \inner{x,x}_Z^{1/2}$ on $X$. Hence $(X,\inner{\cdot,\cdot}_Z)$ is a Hilbert space as well, which we will denote by $Z$. Consider now the functional $\tilde F:Z\to\R$ with $\tilde F(x) := \frac12\norm{x}_X^2$ (which is well-defined since $\norm{\cdot}_X$ is also an equivalent norm on $Z$). Then, the derivative $D\tilde{F}(x)\in Z^*$ is still given by $\dual{D \tilde F(x),h}_Z = \inner{x,h}_X$ for all $h\in Z$ (or, equivalently, for all $h\in X$ since we defined $Z$ via the same vector space). However, $\nabla \tilde F(x)\in Z$ is now characterized by
    \begin{equation*}
        \inner{x,h}_X = \dual{D\tilde{F}(x),h}_Z = \inner{\nabla \tilde F(x),h}_{Z} = \inner{M\nabla \tilde F(x),h}_X \qquad\text{for all } h\in Z,
    \end{equation*}
    i.e., $\nabla \tilde F(x) = M^{-1} x\neq \nabla F(x)$.
\end{example}
(The situation is even more delicate if $M$ is only positive definite on a subspace, as in the case of $X=L^2(\Omega)$ and $Z=H^1(\Omega)$.)

\section{Superposition operators}

A special class of operators on function spaces arise from pointwise application of a real-valued function, e.g., $u(x)\mapsto\sin(u(x))$. We thus consider for $f:\Omega\times \R\to\R$ with $\Omega\subset\R^n$ open and bounded as well as $p,q\in[1,\infty]$ the corresponding \emph{superposition} or \emph{Nemytskii operator}
\begin{equation}\label{eq:superposition}
    F:L^p(\Omega)\to L^q(\Omega),\qquad [F(u)](x) = f(x,u(x)) \quad\text{for almost every }x\in \Omega.
\end{equation}
For this operator to be well-defined requires certain restrictions on $f$. We call $f$ a \emph{Carathéodory function}, if
\begin{enumerate}[(i)]
    \item for all $z\in\R$, the mapping $x\mapsto f(x,z)$ is measurable;
    \item for almost every $x\in\Omega$, the mapping $z\mapsto f(x,z)$ is continuous.
\end{enumerate}
We additionally require the following growth condition: For given $p,q\in[1,\infty)$ there exist $a\in L^q(\Omega)$ and $b\in L^\infty(\Omega)$ with
\begin{equation}\label{eq:superpos:growth}
    |f(x,z)| \leq a(x) +b(x)|z|^{p/q}.
\end{equation}
Under these conditions, $F$ is well-defined and even continuous.
\begin{theorem}\label{thm:superpos:continuous}
    If the Carathéodory function $f:\Omega\times \R\to\R$ satisfies the growth condition \eqref{eq:superpos:growth} for $p,q\in[1,\infty)$, then the superposition operator  $F:L^p(\Omega)\to L^q(\Omega)$ defined via \eqref{eq:superposition} is continuous.
\end{theorem}
\begin{proof}
    We sketch the essential steps; a complete proof can be found in, e.g., \cite[Theorems 3.1, 3.7]{Appell:1990}. First, one shows for given $u\in L^p(\Omega)$ the measurability of $F(u)$ using the Carathéodory properties. It then follows from \eqref{eq:superpos:growth} and the triangle inequality that
    \begin{equation*}
        \norm{F(u)}_{L^q} \leq \norm{a}_{L^q} + \norm{b}_{L^\infty}\norm{|u|^{p/q}}_{L^q}  = \norm{a}_{L^q} + \norm{b}_{L^\infty}\norm{u}^{p/q}_{L^p}<\infty,
    \end{equation*}
    i.e., $F(u)\in L^q(\Omega)$.

    To show continuity, we consider a sequence $\{u_n\}_{n\in\N}\subset L^p(\Omega)$ with $u_n\to u\in L^p(\Omega)$. Then there exists a subsequence, again denoted by $\{u_n\}_{n\in \N}$, that converges pointwise almost everywhere in $\Omega$, as well as a $v\in L^p(\Omega)$ with $|u_n(x)| \leq |v(x)| + |u_1(x)|=:g(x)$ for all $n\in \N$ and almost every $x\in\Omega$ (see, e.g., \cite[Lemma 3.22 as well as (3-14) in the proof of Theorem 3.17]{Alt:2016}).
    The continuity of $z\mapsto f(x,z)$ then implies $F(u_n)\to F(u)$ pointwise almost everywhere as well as
    \begin{equation*}
        |[F(u_n)](x)| \leq a(x) + b(x)|u_n(x)|^{p/q} \leq  a(x) + b(x)|g(x)|^{p/q}\quad\text{for almost every }x\in\Omega.
    \end{equation*}
    Since $g\in L^p(\Omega)$, the right-hand side defines a function in $L^q(\Omega)$, and we can apply Lebesgue's dominated convergence theorem to deduce that $F(u_n)\to F(u)$ in $L^q(\Omega)$.
    As this argument can be applied to any subsequence, the whole sequence must converge to $F(u)$, which yields the claimed continuity.
\end{proof}
In fact, the growth condition \eqref{eq:superpos:growth} is also necessary for continuity; see \cite[Theorem 3.2]{Appell:1990}.
In addition, it is straightforward to show that for $p=q=\infty$, the growth condition \eqref{eq:superpos:growth} (with $p/q := 0$ in this case) implies that $F$ is even locally Lipschitz continuous.

Similarly, one would like to show that differentiability of $f$ implies differentiability of the corresponding superposition operator $F$, ideally with pointwise derivative $[F'(u)h](x) = f'(u(x))h(x)$.
However, this does not hold in general; for example, the superposition operator defined by $f(x,z)=\sin(z)$ is \emph{not} differentiable in $u=0$ for  $1\leq p=q<\infty$. The reason is that for a Fréchet differentiable superposition operator $F:L^p(\Omega)\to L^q(\Omega)$ and a direction $h\in L^p(\Omega)$, the pointwise(!) product $F'(u)h$ has to be in $L^q(\Omega)$.
This leads to additional conditions on the superposition operator $F'$ defined by $f'$, which is known as \emph{two norm discrepancy}.
\begin{theorem}\label{thm:superpos:differentiable}
    Let $f:\Omega\times\R\to\R$ be a Carathéodory function that satisfies the growth condition \eqref{eq:superpos:growth} for $1\leq q<p<\infty$. If the partial derivative $f'_z$ is a Carathéodory function as well and satisfies \eqref{eq:superpos:growth} for $p'=p-q$, the superposition operator $F:L^p(\Omega)\to L^q(\Omega)$ is continuously Fréchet differentiable, and its derivative in $u\in L^p(\Omega)$ in direction $h\in L^p(\Omega)$ is given by
    \begin{equation*}
        [F'(u)h](x) = f_z'(x,u(x))h(x) \qquad\text{for almost every }x\in\Omega.
    \end{equation*}
\end{theorem}
\begin{proof}
    \Cref{thm:superpos:continuous} yields that for $r := \frac{pq}{p-q}$ (i.e., $\frac{r}{p} = \frac{p'}{q}$), the superposition operator
    \begin{equation*}
        G:L^p(\Omega)\to L^r(\Omega),\qquad [G(u)](x)=f'_z(x,u(x))\quad\text{for almost every }x\in\Omega,
    \end{equation*}
    is well-defined and continuous.
    The Hölder inequality further implies that for any $u\in L^p(\Omega)$,
    \begin{equation}\label{eq:superpos:hoelder}
        \norm{G(u)h}_{L^q} \leq \norm{G(u)}_{L^r}\norm{h}_{L^p}\qquad\text{for all }h\in L^p(\Omega),
    \end{equation}
    i.e., $h\mapsto G(u)h$ defines a bounded linear operator $DF(u):L^p(\Omega)\to L^q(\Omega)$.

    Let now $h\in L^p(\Omega)$ be arbitrary. Since $z\mapsto f(x,z)$ is continuously differentiable by assumption, the classical mean value theorem together with \eqref{eq:bochner_est} and \eqref{eq:superpos:hoelder} implies that
    \begin{equation*}
        \begin{multlined}[c][0.9\displaywidth]
            \norm{F(u+h)-F(u)-DF(u)h}_{L^q}\\
            \begin{aligned}[t]
                &= \left(\int_\Omega |f(x,u(x)+h(x)) - f(x,u(x)) - f'_z(x,u(x))h(x)|^q\,dx\right)^{\frac1q}\\
                &= \left(\int_\Omega \left|\int_0^1 f'_z(x,u(x)+th(x))h(x)\,dt  - f'_z(x,u(x))h(x)\right|^q\,dx\right)^{\frac1q}\\
                &=
                \left\|\int_0^1 G(u+th)h\,dt - G(u)h\right\|_{L^q}\\
                & \leq \int_0^1 \norm{(G(u+th)-G(u))h}_{L^q}\,dt\\
                &\leq \int_0^1 \norm{G(u+th)-G(u)}_{L^r}\,dt\ \norm{h}_{L^p}.
            \end{aligned}
        \end{multlined}
    \end{equation*}
    Due to the continuity of $G:L^p(\Omega)\to L^r(\Omega)$, the integral tends to zero for $\norm{h}_{L^p}\to 0$, and hence $F$ is by definition Fréchet differentiable with derivative $F'(u) = DF(u)$ (whose continuity we have already shown).
\end{proof}
In fact, this result is sharp: except for the case $p=q=\infty$, no superposition operator is differentiable from $L^p(\Omega)$ to $L^p(\Omega)$ (unless it is affine-linear); see, e.g., \cite[Theorem 3.12]{Appell:1990}.

\part{Convex analysis}

\chapter{Convex functions}\label{chap:convex}

The classical derivative concepts from the previous chapter are not sufficient for our purposes, since many interesting functionals are not differentiable in this sense; also, they cannot handle functionals with values in $\Rbar$. We therefore need a derivative concept that is more general than Gâteaux and Fréchet derivatives and still allows a Fermat principle and a rich calculus.

We first consider a general class of functionals that admit such a generalized derivative.
A proper functional $F:X\to\Rbar$ is called \emph{convex} if
\begin{equation}\label{eq:convex:def}
    F(\lambda x + (1-\lambda)y)\leq \lambda F(x) + (1-\lambda)F(y)\quad\text{for all }x,y\in X \text{ and } \lambda\in [0,1]
\end{equation}
(where the function value $\infty$ is allowed on both sides).
If for $x\neq y$ and $\lambda \in (0,1)$ we even have
\begin{equation*}
    F(\lambda x + (1-\lambda)y)< \lambda F(x) + (1-\lambda)F(y),
\end{equation*}
we call $F$ \emph{strictly convex}.

An alternative characterization of the convexity of a functional $F:X\to\Rbar$ is based on its \emph{epigraph}
\begin{equation*}
    \epi F :=\setof{(x,t)\in X\times \R}{F(x)\leq t}.
\end{equation*}
\begin{lemma}\label{lem:convex:epi}
    Let $F:X\to\Rbar$. Then $\epi F$ is
    \begin{enumerate}[(i)]
        \item nonempty if and only if $F$ is proper;
        \item convex if and only if $F$ is convex;
        \item (weakly) closed if and only if $F$ is (weakly) lower semicontinuous.
    \end{enumerate}
\end{lemma}
\begin{proof}
    Statement (i) follows directly from the definition: $F$ is proper if and only if there exists an $x\in X$ and a $t\in \R$ with $F(x)\leq t <\infty$, i.e., $(x,t)\in\epi F$.

    For (ii), let $F$ be convex and $(x,r),(y,s)\in \epi F$ be given. For any $\lambda\in[0,1]$, the definition \eqref{eq:convex:def} then implies that
    \begin{equation*}
        F(\lambda x + (1-\lambda)y)\leq \lambda F(x) + (1-\lambda)F(y) \leq \lambda r + (1-\lambda)s,
    \end{equation*}
    i.e., that
    \begin{equation*}
        \lambda(x,r) + (1-\lambda)(y,s) = (\lambda x + (1-\lambda)y,\lambda r + (1-\lambda)s) \in \epi F,
    \end{equation*}
    and hence $\epi F$ is convex.
    Let conversely $\epi F$ be convex and $x,y\in X$ be arbitrary, where we can assume that $F(x)<\infty$ and $F(y)<\infty$ (otherwise \eqref{eq:convex:def} is trivially satisfied). We  clearly have $(x,F(x)),(y,F(y))\in\epi F$. The convexity of $\epi F$ then implies for all  $\lambda\in[0,1]$ that
    \begin{equation*}
        (\lambda x + (1-\lambda)y,\lambda F(x) + (1-\lambda)F(y)) = \lambda(x,F(x)) + (1-\lambda)(y,F(y)) \in \epi F,
    \end{equation*}
    and hence by definition of $\epi F$ that \eqref{eq:convex:def} holds.

    Finally, we show (iii): Let first $F$ be lower semicontinuous and $\{(x_n,t_n)\}_{n\in\N}\subset \epi F$ be an arbitrary sequence with $(x_n,t_n)\to(x,t)\in X\times \R$. Then we have that
    \begin{equation*}
        F(x)\leq \liminf_{n\to\infty} F(x_n) \leq  \limsup_{n\to\infty} t_n = t,
    \end{equation*}
    i.e., $(x,t)\in \epi F$. Let conversely $\epi F$ be closed and assume that $F$ is not lower semicontinuous. Then there exists a sequence $\{x_n\}_{n\in\N}\subset X$ with $x_n\to x\in X$ and
    \begin{equation*}
        F(x) > \liminf_{n\to\infty} F(x_n) =: M \in [-\infty,\infty).
    \end{equation*}
    We now distinguish two cases.
    \begin{enumerate}[a)]
        \item $x\in \dom F$: In this case, we can select a subsequence, again denoted by  $\{x_n\}_{n\in\N}$, such that there exists an $\eps>0$ with $F(x_n) \leq F(x)-\eps$ and thus $(x_n,F(x)-\eps)\in \epi F$ for all $n\in \N$.
            From $x_n\to x$ and the closedness of $\epi F$, we deduce that $(x,F(x)-\eps)\in \epi F$ and hence $F(x) \leq F(x) - \eps$, contradicting $\eps>0$.

        \item $x\not\in \dom F$: In this case, we can argue similarly using $F(x_n) \leq M+\eps$ for $M>-\infty$ or $F(x_n) \leq \eps$ for $M=-\infty$ to obtain a contradiction with $F(x)=\infty$.
    \end{enumerate}
    The equivalence of weak lower semicontinuity and weak closedness follows in exactly the same way.
\end{proof}
Note that $(x,t)\in \epi F$ implies that $x\in \dom F$; hence the effective domain of a proper, convex, and lower semicontinuous functional is always nonempty, convex, and closed as well.
Also, together with \cref{lem:convex_closed} we immediately obtain
\begin{cor}\label{cor:convex:uhs}
    Let $F:X\to\Rbar$ be convex. Then, $F$ is weakly lower semicontinuous if and only $F$ is lower semicontinuous.
\end{cor}

Also useful for the study of a functional $F:X\to\Rbar$ are the corresponding \emph{sublevel sets}
\begin{equation*}
    F_\alpha := \setof{x\in X}{F(x)\leq \alpha},\qquad\alpha\in\R,
\end{equation*}
for which one shows as in \cref{lem:convex:epi} the following properties.
\clearpage
\begin{lemma}\label{lem:convex:sublevel}
    Let $F:X\to\Rbar$.
    \begin{enumerate}[(i)]
        \item If $F$ is convex, $F_\alpha$ is convex for all $\alpha\in\R$, but the converse does not hold.
        \item $F$ is (weakly) lower semicontinuous if and only if $F_\alpha$ is (weakly) closed  for all $\alpha\in\R$.
    \end{enumerate}
\end{lemma}

Directly from the definition we obtain the convexity of
\begin{enumerate}[(i)]
    \item \emph{affine functionals} of the form $x\mapsto \dual{x^*,x}_X - \alpha$ for fixed $x^*\in X^*$ and $\alpha\in\R$;
    \item the norm $\norm{\cdot}_X$ in a normed vector space $X$;
    \item the indicator function $\delta_C$ for a convex set $C$.
\end{enumerate}
If $X$ is a Hilbert space, $F(x) = \norm{x}_X^2$ is even strictly convex: For $x,y\in X$ with $x\neq y$ and any $\lambda\in (0,1)$,
\begin{equation*}
    \begin{aligned}
        \norm{\lambda x+(1-\lambda )y}_X^2 &= \inner{\lambda x+(1-\lambda )y,\lambda x+(1-\lambda )y}_X \\
        &= \lambda ^2\inner{x,x}_X + 2\lambda (1-\lambda )\inner{x,y}_X+(1-\lambda )^2\inner{y,y}_X\\
        &=\lambda \Big(\lambda \inner{x,x}_X-(1-\lambda )\inner{x-y,y}_X+(1-\lambda )\inner{y,y}_X\Big)\\
        \MoveEqLeft[-1]
        +(1-\lambda )\Big(\lambda \inner{x,x}_X+\lambda \inner{x-y,y}_X+(1-\lambda )\inner{y,y}_X\Big)\\
        &=(\lambda +(1-\lambda ))\Big(\lambda \inner{x,x}_X+(1-\lambda )\inner{y,y}_X\Big)-\lambda (1-\lambda )\inner{x-y,x-y}_X\\
        &=\lambda \norm{x}_X^2 +(1-\lambda )\norm{y}_X^2-\lambda (1-\lambda )\norm{x-y}_X^2\\
        &<\lambda \norm{x}_X^2 +(1-\lambda )\norm{y}_X^2.
    \end{aligned}
\end{equation*}

Further examples can be constructed as in \cref{lem:variation:wlsc} through the following operations.
\begin{lemma}\label{lem:convex:func}
    Let $X$ and $Y$ be normed vector spaces and let $F:X\to\Rbar$ be convex. Then the following functionals are convex as well:
    \begin{enumerate}[(i)]
        \item $\alpha F$ for all $\alpha \geq 0$;
        \item $F+G$ for $G:X\to\Rbar$ convex (strictly if $F$ \emph{or} $G$ is strictly convex);
        \item $\phi\circ F$ for $\phi:\Rbar\to\Rbar$ convex and increasing;
        \item $F\circ A$ for $A:Y\to X$ linear;
        \item $x\mapsto \sup_{i\in I} F_i(x)$ with $F_i:X\to\Rbar$ convex for an arbitrary set $I$.
    \end{enumerate}
\end{lemma}

\Cref{lem:convex:func}\,(v) in particular implies that the pointwise supremum of affine functionals is always convex. In fact, any convex functional can be written in this way. To show this, we define for a proper functional $F:X\to\Rbar$ the \emph{convex hull}
\begin{equation*}
    F^\Gamma(x) := \sup\setof{a(x)}{a \text{ affine with } a(\tilde x)\leq F(\tilde x)\text{ for all }\tilde x\in X}.
\end{equation*}
Note that $F^\Gamma:X\to [-\infty,\infty]$ without further assumptions of $F$.
\begin{lemma}\label{lem:convex:gamma}
    Let $F:X\to\Rbar$ be proper. Then $F$ is convex and lower semicontinuous if and only if $F=F^\Gamma$.
\end{lemma}
\begin{proof}
    Since affine functionals are convex and continuous, \cref{lem:convex:func}\,(v) and \cref{lem:variation:wlsc}\,(v) imply that $F=F^\Gamma$ is always convex and lower semicontinuous.

    Let now $F:X\to\Rbar$ be proper, convex, and lower semicontinuous.
    It is obvious from the definition of $F^\Gamma$ as a supremum that $F^\Gamma\leq F$ always holds pointwise. Assume
    that $F^\Gamma<F$. Then there exists an $x_0\in X$ and a $\lambda\in\R$ with
    \begin{equation*}
        F^\Gamma(x_0) < \lambda < F(x_0).
    \end{equation*}
    We now use the Hahn--Banach separation theorem to construct an affine functional $a$ with $a\leq F$ but $a(x_0)> \lambda>F^\Gamma(x_0)$, which would contradict the definition of $F^\Gamma$.
    Since $F$ is proper, convex, and lower semicontinuous, $\epi F$ is nonempty, convex, and closed by \cref{lem:convex:epi}. Furthermore, $\{(x_0,\lambda)\}$ is compact and, as $\lambda < F(x_0)$, disjoint with $\epi F$. \Cref{thm:hb_separation}\,(ii) hence yields a $z^*\in (X\times \R)^*$ and an $\alpha\in \R$ with
    \begin{equation*}
        \dual{z^*,(x,t)}_{X\times \R} \leq \alpha <  \dual{z^*,(x_0,\lambda)}_{X\times \R} \qquad\text{for all }(x,t)\in\epi F.
    \end{equation*}
    We now define an $x^*\in X^*$ via $\dual{x^*,x}_X = \dual{z^*,(x,0)}_{X\times \R}$ for all $x\in X$ and set $s:=\dual{z^*,(0,1)}_{X\times \R}\in \R$. Then, $\dual{z^*,(x,t)}_{X\times \R} =\dual{x^*,x}_X + st$ and hence
    \begin{equation}\label{eq:convex:func1}
        \dual{x^*,x}_X + st  \leq \alpha < \dual{x^*,x_0}_X +s\lambda  \qquad\text{for all } (x,t) \in\epi F.
    \end{equation}
    Now for $(x,t)\in \epi F$ we also have $(x,t')\in \epi F$ for all $t'>t$, and the first inequality in \eqref{eq:convex:func1} implies that for all sufficiently large $t'>0$,
    \begin{equation*}
        s \leq \frac{\alpha - \dual{x^*,x}_X}{t'} \to 0 \qquad\text{for }t'\to\infty.
    \end{equation*}
    Hence $s\leq 0$. We continue with a case distinction.
    \begin{enumerate}[(i)]
        \item $s<0$: We set
            \begin{equation*}
                a:X\to\R,\qquad x\mapsto  \frac{\alpha - \dual{x^*,x}_X}{s},
            \end{equation*}
            which is affine and continuous. Furthermore, using the \enquote{productive zero} (i.e., adding and subtracting the same term) in the first inequality in \eqref{eq:convex:func1} for $(x,F(x))\in \epi F$ implies (noting $s<0$!) that
            \begin{equation*}
                a(x) = \tfrac1s \left(\alpha - \dual{x^*,x}_X -sF(x)\right)+F(x) \leq F(x).
            \end{equation*}
            (For $x\notin \dom F$ this holds trivially.) But the second inequality in \eqref{eq:convex:func1} implies that
            \begin{equation*}
                a(x_0) = \tfrac1s\left(\alpha - \dual{x^*,x_0}_X\right) > \lambda.
            \end{equation*}
        \item $s=0$: Then $\dual{x^*,x}_X\leq \alpha<\dual{x^*,x_0}_X$ for all $x\in\dom F$, which can only hold for $x_0\notin \dom F$. But $F$ is proper, and hence we can find a $y_0\in \dom F$, for which we can construct as in case (i) by separating $\epi F$ and $(y_0,\mu)$ for sufficiently small $\mu$ a continuous affine functional $a_0:X\to\R$ with $a_0 \leq F$ pointwise. For $\rho>0$ we now set
            \begin{equation*}
                a_\rho:X\to\R, \qquad x\mapsto a_0(x) + \rho\left( \dual{x^*,x}_X-\alpha\right),
            \end{equation*}
            which is affine and continuous as well. Since $\dual{x^*,x}_X\leq \alpha$, we also have that $a_\rho(x) \leq a_0(x) \leq F(x)$ for all $x\in\dom F$ and any $\rho >0$. But due to $\dual{x^*,x_0}_X>\alpha$, we can choose $\rho>0$ with $a_\rho(x_0)>\lambda$.
    \end{enumerate}
    In both cases, the definition of $F^\Gamma$ as a supremum implies that $F^\Gamma(x_0)> \lambda$ as well, contradicting the assumption $F^\Gamma(x_0) <\lambda$.
\end{proof}

A particularly useful class of convex functionals in the calculus of variations arises from integral functionals with convex integrands defined through superposition operators.
\begin{lemma}\label{lem:lebesgue:lsc}
    Let $f:\R\to\Rbar$ be proper, convex, and lower semicontinuous. If $\Omega\subset\R^n$ is bounded and $1\leq p\leq \infty$, this also holds for
    \begin{equation*}
        F:L^p(\Omega)\to\Rbar,\qquad u\mapsto
        \begin{cases}
            \int_\Omega f(u(x))\,dx &\text{if }f\circ u \in L^1(\Omega),\\
            \infty &\text{else.}
        \end{cases}
    \end{equation*}
\end{lemma}
\begin{proof}
    First, \cref{lem:convex:gamma} implies that there exist $a,\alpha\in\R$ such that
    \begin{equation}\label{eq:lebesgue:lsc_bound}
        f(t) \geq at - \alpha \qquad\text{for all }t\in\R.
    \end{equation}
    Since $\Omega$ is bounded and hence $L^p(\Omega)\subset L^1(\Omega)$ for any $p\geq 1$, this implies that
    \begin{equation*}
        F(u) \geq \int_\Omega a u(x) - \alpha \,dx \in \R \qquad \text{for any }u\in L^p(\Omega).
    \end{equation*}
    In particular, $F(u) > -\infty$ for all $u \in L^p(\Omega)$.
    Since $f$ is proper, there is a $t_0\in \dom f$. Hence (using again that $\Omega$ is bounded) the constant function $u_0 \equiv t_0 \in \dom F$ satisfies $F(u_0) < \infty$. This shows that $F$ is proper.

    To show convexity, we take $u,v\in \dom F$ (since otherwise \eqref{eq:convex:def} is trivially satisfied) and $\lambda\in[0,1]$ arbitrary.
    The convexity of $f$ now implies that
    \begin{equation*}
        f(\lambda u(x) + (1-\lambda)v(x)) \leq \lambda f(u(x))+(1-\lambda )f(v(x))\quad\text{for almost every }x\in\Omega.
    \end{equation*}
    Since $u,v\in \dom F$ and $L^1(\Omega)$ is a vector space, $\lambda f(u(x)) + (1-\lambda)f(v(x)) \in L^1(\Omega)$ as well. Similarly, the left-hand side is bounded from below by $a(\lambda u(x) + (1-\lambda)v(x))-\alpha\in L^1(\Omega)$ by \eqref{eq:lebesgue:lsc_bound}. We can thus integrate the inequality over $\Omega$ to obtain the convexity of $F$.

    To show lower semicontinuity, we use \cref{lem:convex:epi}. Let $\{(u_n,t_n)\}_{n\in\N}\subset \epi F$ with $u_n\to u$ in $L^p(\Omega)$ and $t_n\to t$ in $\R$.
    Then there exists a subsequence $\{u_{n_k}\}_{k\in\N}$ with $u_{n_k}(x)\to u(x)$ almost everywhere. Hence, the lower semicontinuity of $f$ together with Fatou's Lemma implies that
    \begin{equation*}
        \begin{aligned}
            \int_\Omega f(u(x))-(au(x)-\alpha) \,dx
            &\leq \int_\Omega \liminf_{k\to\infty}(f(u_{n_k}(x))-(au_{n_k}(x)-\alpha))\,dx\\
            &\leq  \liminf_{k\to\infty}\int_\Omega f(u_{n_k}(x))-(au_{n_k}(x)-\alpha)\,dx\\
            &= \liminf_{k\to\infty}\int_\Omega f(u_{n_k}(x))\,dx-\int_\Omega au(x)-\alpha\,dx
        \end{aligned}
    \end{equation*}
    as the integrands are nonnegative due to \eqref{eq:lebesgue:lsc_bound}.
    Since $(u_{n_k},t_{n_k})\in \epi F$, this yields
    \begin{equation*}
        F(u) =  \int_\Omega f(u(x))\,dx \leq \liminf_{k\to\infty} \int_\Omega f(u_{n_k}(x))\,dx = \liminf_{k\to\infty} F(u_{n_k}) \leq \lim_{k\to\infty} t_{n_k} = t,
    \end{equation*}
    i.e., $(u,t)\in \epi F$. Hence $\epi F$ is closed, and the lower semicontinuity of $F$ follows from  \cref{lem:convex:epi}\,(iii).
\end{proof}

\bigskip

After all this preparation, we can quickly prove the main result on existence of solutions to convex minimization problems.
\begin{theorem}\label{thm:convex:existence}
    Let $X$ be a reflexive Banach space and let
    \begin{enumerate}[(i)]
        \item $U\subset X$ be nonempty, convex, and closed;
        \item $F:U\to\Rbar$ be proper, convex, and lower semicontinuous with $\dom F \cap U \neq \emptyset$;
        \item $U$ be bounded or $F$ be coercive.
    \end{enumerate}
    Then the problem
    \begin{equation*}
        \min_{x\in U} F(x)
    \end{equation*}
    admits a solution $\bar x\in U\cap \dom F$. If $F$ is strictly convex, the solution is unique.
\end{theorem}
\begin{proof}
    We consider the extended functional $\overline F = F + \delta_U:X\to\Rbar$.
    Assumption (i) together with \cref{lem:variation:wlsc} implies that $\delta_U$ is proper, convex, and weakly lower semicontinuous. From (ii) we obtain an $x_0\in U$ with $\overline F(x_0)<\infty$, and hence $\overline F$ is proper, convex, and weakly lower semicontinuous. Finally, $\overline F$ is coercive since for bounded $U$, we can use that $F>-\infty$, and for coercive $F$, we can use that $\delta_U\geq 0$. Hence we can apply \cref{thm:variation:existence} to obtain the existence of a minimizer $\bar x \in \dom \overline F =  U\cap \dom F$ of $\overline F$ with
    \begin{equation*}
        F(\bar x) = \overline F(\bar x) \leq \overline F(x) = F(x) \qquad\text{for all }x\in U,
    \end{equation*}
    i.e., $\bar x$ is the claimed solution.

    Let now $F$ be strictly convex, and let $\bar x$ and $\bar x'\in U$ be two different minimizers, i.e., $F(\bar x) = F(\bar x') = \min_{x\in U}F(x)$ and $\bar x\neq \bar x'$.
    Then by the convexity of $U$ we have for all $\lambda\in (0,1)$ that
    \begin{equation*}
        x_\lambda:=\lambda \bar x + (1-\lambda) \bar x' \in U,
    \end{equation*}
    while the strict convexity of $F$ implies that
    \begin{equation*}
        F(x_\lambda) < \lambda F(\bar x) + (1-\lambda) F(\bar x') = F(\bar x).
    \end{equation*}
    But this is a contradiction to $F(\bar x)\leq F(x)$ for all $x\in U$.
\end{proof}

Note that for a sum of two convex functionals to be coercive, it is in general not sufficient that only one of them is. Functionals for which this is the case -- such as the indicator function of a bounded set -- are called \emph{supercoercive}; another example which will be helpful later is the squared norm.
\begin{lemma}\label{lem:convex:supercoercive}
    Let $F:X\to\Rbar$ be proper, convex, and lower semicontinuous, and $x_0\in X$ be given. Then the functional
    \begin{equation*}
        J:X\to\Rbar,\qquad x\mapsto F(x) + \frac12\norm{x-x_0}_X^2
    \end{equation*}
    is coercive.
\end{lemma}
\begin{proof}
    Since $F$ is proper, convex, and lower semicontinuous, it follows from \cref{lem:convex:gamma} that $F$ is bounded from below by an affine functional, i.e., there exists an $x^*\in X^*$ and an $\alpha\in\R$ with $F(x)\geq \dual{x^*,x}_X -\alpha$ for all $x\in X$. Together with the reverse triangle inequality and \eqref{eq:functan:cs_banach}, we obtain that
    \begin{equation*}
        \begin{aligned}
            J(x) &\geq \dual{x^*,x}_X - \alpha + \tfrac12\left(\norm{x}_X-\norm{x_0}_X\right)^2\\
            &\geq -\norm{x^*}_{X^*}\norm{x}_X -\alpha + \tfrac12\norm{x}_X^2 -\norm{x}_X\norm{x_0}_X\\
            &= \norm{x}_X\left(\tfrac12\norm{x}_X -\norm{x^*}_{X^*} - \norm{x_0}_X\right) - \alpha.
        \end{aligned}
    \end{equation*}
    Since $x^*$ and $x_0$ are fixed, the term in parentheses is positive for $\norm{x}_X$ sufficiently large, and hence $J(x)\to \infty$ for $\norm{x}_X\to \infty$ as claimed.
\end{proof}

\bigskip

To close this chapter, we show the following remarkable result: \emph{Any (locally) bounded convex functional is (locally) continuous.} (An extended real-valued proper functional must necessarily be discontinuous at some point.)
Besides being of use in later chapters, this result illustrates the beauty of convex analysis: an algebraic but global property (convexity) connects two topological but local properties (neighborhood and continuity).
Here we consider of course the strong topology in a normed vector space.
\begin{lemma}\label{thm:convex:cont_bounded}
    Let $X$ be a normed vector space, $F:X\to\Rbar$ be convex, and $x\in X$. If there is a $\rho>$ such that $F$ is bounded from above on $O_\rho(x)$, then $F$ is locally Lipschitz continuous in $x$.
\end{lemma}
\begin{proof}
    By assumption, there exists an $M\in\R$ with $F(y)\leq M$ for all $y\in O_\rho(x)$.
    We first show that $F$ is locally bounded from below as well. Let $y\in O_\rho(x)$ be arbitrary. Since $\norm{x-y}_X<\rho$, we also have that $z:=2x-y = x-(y-x)\in O_\rho(x)$, and the convexity of $F$ implies that
    $F(x) = F\left(\tfrac12 y + \tfrac12 z\right) \leq \tfrac12 F(y)+\tfrac12 F(z)$
    and hence that
    \begin{equation*}
        -F(y) \leq F(z) - 2 F(x) \leq M - 2F(x) =:m,
    \end{equation*}
    i.e., $-m\leq F(y) \leq M$ for all $y\in O_\rho(x)$.

    We now show that this implies Lipschitz continuity on $O_{\frac\rho2}(x)$. Let $y_1,y_2\in O_{\frac\rho2}(x)$ with $y_1\neq y_2$ and set
    \begin{equation*}
        z:=y_1 + \frac\rho2 \frac{y_1-y_2}{\norm{y_1-y_2}_X} \in O_\rho(x),
    \end{equation*}
    which holds because $\norm{z-x}_X \leq \norm{y_1-x}_X +\frac\rho2 < \rho$. By construction, we thus have that
    \begin{equation*}
        y_1 = \lambda z + (1-\lambda) y_2 \quad\text{for}\quad
        \lambda:=\frac{\norm{y_1-y_2}_X}{\norm{y_1-y_2}_X + \tfrac\rho2} \in(0,1),
    \end{equation*}
    and the convexity of $F$ now implies that $F(y_1) \leq \lambda F(z) + (1-\lambda)F(y_2)$.
    Together with the definition of $\lambda$ as well as $F(z)\leq M$ and $-F(y_1)\leq m=M-2F(x)$, this yields the estimate
    \begin{equation*}
        \begin{aligned}
            F(y_1) -F(y_2) \leq \lambda(F(z)-F(y_2)) &\leq \lambda(2M-2F(x))\\
            &\leq \frac{2(M-F(x))}{\norm{y_1-y_2}_X + \frac\rho2} \norm{y_1-y_2}_X\\
            &\leq \frac{2(M-F(x))}{\rho/2} \norm{y_1-y_2}_X.
        \end{aligned}
    \end{equation*}
    Exchanging the roles of $y_1$ and $y_2$, we obtain that
    \begin{equation*}
        |F(y_1)-F(y_2)| \leq \frac{2(M-F(x))}{\rho/2} \norm{y_1-y_2}_X \quad\text{for all }y_1,y_2\in O_{\frac\rho2}(x)
    \end{equation*}
    and hence the local Lipschitz continuity with constant $L(x,\rho/2):= 4(M-F(x))/\rho$.
\end{proof}
It thus remains to show that convex functions are bounded from above. We start with the scalar case.
\begin{lemma}\label{cor:convex:cont_r}
    If $f:\R\to\Rbar$ is convex, then $f$ is locally bounded from above on $(\dom f)^o$.
\end{lemma}
\begin{proof}
    Let $x\in (\dom f)^o$, i.e., there exist $a,b\in\R$ with $x\in(a,b)\subset \dom f$; by possibly shrinking the interval we can even assume that $[a,b]\subset \dom f$. Let now $z\in(a,b)$. Since intervals are convex, there exists a $\lambda\in(0,1)$ with $z=\lambda a+(1-\lambda)b$. By convexity, we thus have
    \begin{equation*}
        f(z) \leq \lambda f(a) + (1-\lambda)f(b) \leq \max\{|f(a)|,|f(b)|\} < \infty.
    \end{equation*}
    Hence $f$ is locally bounded from above in $x$.
\end{proof}
With a bit more effort, one can show that the claim holds for $F:\R^n\to\Rbar$ with arbitrary $n\in\N$; see, e.g., \cite[Corollary 1.4.2]{Schirotzek:2007}.

The proof of the general case requires further assumptions on $X$ and $F$.
\begin{lemma}\label{lem:convex:bounded}
    Let $X$ be a Banach space. If $F:X\to\Rbar$ is convex and lower semicontinuous, then $F$ is locally bounded from above on $(\dom F)^o$.
\end{lemma}
\begin{proof}
    We first show the claim for the case $x=0\in(\dom F)^o$, which implies in particular that $M:=|F(0)|<\infty$. Consider now for arbitrary $h\in X$ the mapping
    \begin{equation*}
        f:\R \to\Rbar,\qquad t\mapsto F(th).
    \end{equation*}
    It is straightforward to verify that $f$ is convex and lower semicontinuous as well and satisfies $0\in(\dom f)^o$.
    By \cref{cor:convex:cont_r,thm:convex:cont_bounded}, $f$ is thus locally Lipschitz continuous in $0$; in particular, $|f(t)-f(0)|\leq L t \leq 1$ for sufficiently small $t>0$. The reverse triangle inequality therefore yields a $\delta>0$ with
    \begin{equation*}
        F(0+t h) \leq |F(0+th)| = |f(t)| \leq |f(0)|+1 = M+1 \qquad\text{for all } t\in [0,\delta].
    \end{equation*}
    Hence $0$ lies in the algebraic interior of the sublevel set $F_{M+1}$, which is convex and closed by \cref{lem:convex:sublevel}. The core--int \cref{lem:functan:coreint} thus yields that $0\in(F_{M+1})^o$, i.e., there exists a $\rho>0$ with $F(z)\leq M+1$ for all $z\in O_\rho(0)$.

    For the general case $x\in (\dom F)^o$, consider
    \begin{equation*}
        \tilde F:X\to\Rbar,\qquad y\mapsto F(y-x).
    \end{equation*}
    Again, it is straightforward to verify convexity and lower semicontinuity of  $\tilde F$ and that $0\in (\dom \tilde F)^o$.
    It follows from what we've shown before that $\tilde F$ is locally bounded from above on $O_\rho(0)$, which also implies that $F$ is locally bounded from above on $O_\rho(x)$.
\end{proof}
Together with \cref{thm:convex:cont_bounded}, we thus obtain the desired result.
\begin{theorem}\label{thm:convex:cont}
    Let $X$ be a Banach space. If $F:X\to\Rbar$ is convex and lower semicontinuous, then $F$ is locally Lipschitz continuous on $(\dom F)^o$.
\end{theorem}
We shall have several more occasions to observe the unreasonably nice behavior of convex functions on the interior of their effective domain.

\chapter{Convex subdifferentials}

We now turn to the characterization of minimizers of convex functionals via a Fermat principle. A first candidate for the required notion of derivative is the directional derivative, since it exists (at least in the extended real-valued sense) for any convex function.
\begin{lemma}\label{lem:convex:direct}
    Let $F:X\to\Rbar$ be convex and let $x\in \dom F$ and $h\in X$ be given. Then:
    \begin{enumerate}[(i)]
        \item the function
            \begin{equation*}
                \phi:(0,\infty) \to \Rbar,\qquad t\mapsto \frac{F(x+th)-F(x)}{t},
            \end{equation*}
            is increasing;
        \item there exists a limit $F'(x;h)=\lim_{t\to 0^+}\phi(t)\in[-\infty,\infty]$, which satisfies
            \begin{equation*}
                F'(x;h) \leq F(x+h) - F(x);
            \end{equation*}
        \item if $x\in (\dom F)^o$, the limit $F'(x;h)$ is finite.
    \end{enumerate}
\end{lemma}
\begin{proof}
    \emph{(i):} Inserting the definition and sorting terms shows that for all $0<s<t$, the condition $\phi(s)\leq \phi(t)$ is equivalent to
    \begin{equation*}
        F(x+sh) \leq \frac{s}{t} F(x+th) + \left(1-\frac{s}{t} \right)F(x),
    \end{equation*}
    which follows from the convexity of $F$ since $x+sh = (1-\frac{s}{t})x+\frac{s}{t}(x+th)$.

    \emph{(ii):} The claim immediately follows from (i) since
    \begin{equation*}
        F'(x;h) = \lim_{t\to 0^+} \phi(t) = \inf_{t>0} \phi(t) \leq \phi(1) = F(x+h)-F(x).
    \end{equation*}

    \emph{(iii):} Since $(\dom F)^o$ is contained in the algebraic interior of $\dom F$, there exists an $\eps>0$ such that $x +t h\in \dom F$ for all $t\in(-\eps,\eps)$. Proceeding as in (i), we obtain that $\phi(s)\leq \phi(t)$ for all $s<t<0$ as well. From $x=\frac12(x+th)+\frac12(x-th)$ for $t>0$, we also obtain that
    \begin{equation*}
        \phi(-t) = \frac{F(x-th)-F(x)}{-t} \leq \frac{F(x+th)-F(x)}{t} = \phi(t)
    \end{equation*}
    and hence that $\phi$ is increasing on all $\R\setminus\{0\}$. As in (ii), the choice of $\eps$ now implies that
    \begin{equation*}
        \begin{split}
            -\infty < \phi(-\eps) \leq F'(x;h) \leq \phi(\eps) < \infty.
            \qedhere
        \end{split}
    \end{equation*}
\end{proof}

Unfortunately, this concept can't yet be what we are looking for, since the convex function $f:\R\to\R$, $f(t) = |t|$ has a minimum in $t=0$, but $f'(0;h) = |h|>0$ for $h\in \R\setminus\{0\}$. We thus don't have $f'(0;h) = 0$ for some $h\neq 0$ -- but we at least have  $0\leq f'(0;h)$ for all $h\in\R$. It is this condition that we now generalize to normed vector spaces. For this purpose, consider for convex $F:X\to\Rbar$ and any $x\in\dom F$ the set
\begin{equation}\label{eq:convex:subdiff_dir}
    \setof{x^*\in X^*}{\dual{x^*,h}_X \leq F'(x;h) \quad \text{for all }h\in X}.
\end{equation}
With the help of \cref{lem:convex:direct}, this set (which can be empty!) can also be expressed without directional derivatives.
\begin{lemma}\label{lem:convex:equiv}
    Let $F:X\to\Rbar$ be convex and $x\in \dom F$. For any $x^*\in X^*$, the following statements are equivalent:
    \begin{enumerate}[(i)]
        \item $\dual{x^*,h}_X \leq F'(x;h)$ \qquad for all $h\in X$;
        \item $\dual{x^*,h}_X \leq F(x+h) - F(x)$ \quad for all $h\in X$.
    \end{enumerate}
\end{lemma}
\begin{proof}
    If (i) holds, we immediately obtain from \cref{lem:convex:direct}\,(ii) that
    \begin{equation*}
        \dual{x^*,h}_X \leq F'(x;h) \leq F(x+h) - F(x)\qquad \text{for all }h\in X.
    \end{equation*}
    Conversely, if (ii) holds for all $h\in X$, it also holds for $th$ for all $h\in X$ and $t>0$. Dividing by $t$ and passing to the limit then yields that
    \begin{equation*}
        \begin{split}
            \dual{x^*,h}_X \leq \lim_{t\to 0^+} \frac{F(x+th)-F(x)}{t} = F'(x;h).
            \qedhere
        \end{split}
    \end{equation*}
\end{proof}
If we introduce $\tilde x = x+h\in X$, the second statement leads to our desired derivative concept: For $F:X\to\Rbar$ and $x\in \dom F$, we define the \emph{(convex) subdifferential} as
\begin{equation}\label{eq:convex:def_subdiff}
    \partial F(x) :=  \setof{x^*\in X^*}{\dual{x^*,\tilde x - x}_X \leq F(\tilde x) - F(x)\quad\text{for all }\tilde x\in X}.
\end{equation}
(Note that $\tilde x\notin \dom F$ is allowed since in this case the inequality is trivially satisfied.) For $x\notin\dom F$, we set $\partial F(x) = \emptyset$.\footnote{We will later show that $\partial F(x)$ is nonempty and bounded for all $x\in(\dom F)^o$; see \cref{cor:convex:nonempty}.}
It follows directly from the definition that $\partial F(x)$ is convex and weakly-$*$ closed.
An element $\xi\in \partial F(x)$ is called a \emph{subderivative}.%
\footnote{Following the terminology for classical derivatives, we reserve the more common term \emph{subgradient} for its Riesz representation $z_{x^*} \in X$ when $X$ is a Hilbert space.}
\begin{theorem}[Fermat principle]\label{thm:convex:fermat}
    Let $F:X\to\Rbar$ and $\bar x \in \dom F$. Then the following statements are equivalent:
    \begin{enumerate}[(i)]
        \item $\displaystyle  0\in \partial F(\bar x)$;
        \item $\displaystyle F(\bar x) = \min_{x\in X} F(x)$.
    \end{enumerate}
\end{theorem}
\begin{proof}
    This is a direct consequence of the definitions: $0\in\partial F(\bar x)$ if and only if
    \begin{equation*}
        0= \dual{0,\tilde x-\bar x}_X \leq F(\tilde x) - F(\bar x) \qquad\text{for all } \tilde x\in X,
    \end{equation*}
    i.e., $F(\bar x)\leq F(\tilde x)$ for all $\tilde x\in X$.%
    \footnote{Note that convexity of $F$ is not required for \cref{thm:convex:fermat}! The condition $0\in \partial F(\bar x)$ therefore characterizes the global(!) minimizers of \emph{any} function $F$. However, nonconvex functionals can also have local minimizers, for which the subdifferential inclusion is not satisfied.
    In fact, (convex) subdifferentials of nonconvex functionals are usually empty. (And conversely, one can show that $\partial F(x)\neq \emptyset$ for all $x\in\dom F$ implies that $F$ is convex.) This leads to problems in particular for the proof of calculus rules, for which we will indeed have to assume convexity.}
\end{proof}
This matches the geometrical intuition: If $X=\R\cong X^*$, the affine function $f(\tilde x) := f(x) + \xi(\tilde x - x)$ with $\xi\in \partial f(x)$ describes a tangent at $(x,f(x))$ with slope $\xi$; die condition $\xi=0\in \partial f(\tilde x)$ thus means that $f$ has a horizontal tangent in $\bar x$.

\bigskip

We now look at some examples. First, the construction from the directional derivative indicates that the subdifferential is indeed a generalization of the Gâteaux derivative.
\begin{theorem}\label{thm:convex:gateaux}
    Let $F:X\to\Rbar$ be convex and Gâteaux differentiable in $x$. Then, $\partial F(x) = \{DF(x)\}$.
\end{theorem}
\begin{proof}
    By definition of the Gâteaux derivative, we have that
    \begin{equation*}
        \dual{DF(x),h}_X = DF(x)h =  F'(x;h) \quad\text{for all } h\in X.
    \end{equation*}
    \Cref{lem:convex:equiv} with $\tilde x := x+h$ now immediately yields $DF(x)\in \partial F(x)$.

    Conversely, the definition of $\xi\in \partial F(x)$ with $h:=\tilde x -x\in X$ implies that
    \begin{equation*}
        \dual{\xi,h}_X \leq F'(x;h) =\dual{DF(x),h}_X.
    \end{equation*}
    Since $\tilde x\in X$ was arbitrary, this has to hold for all $h\in X$. Taking the supremum over all $h$ with $\norm{h}_X\leq 1$ now yields that $\norm{\xi - DF(x)}_{X^*} \leq 0$, i.e., $\xi = DF(x)$.
\end{proof}
Of course, we also want to compute subdifferentials of functionals that are not differentiable. The canonical example is the norm $\norm{\cdot}_X$ in a normed vector space, which even for $X=\R$ is not differentiable in $x=0$.
\begin{theorem}\label{thm:subdifferential:norm}
    For any $x\in X$,
    \begin{equation*}
        \partial(\norm{\cdot}_X)(x) =
        \begin{cases}
            \setof{x^*\in X^*}{\dual{x^*,x}_X = \norm{x}_X \text{ and } \norm{x^*}_{X^*} = 1} &\text{if } x\neq 0,\\
            B_{X^*} &\text{if } x = 0.
        \end{cases}
    \end{equation*}
\end{theorem}
\begin{proof}
    For $x=0$, we have $\xi\in\partial(\norm{\cdot}_X)(x)$ by definition if and only if
    \begin{equation*}
        \dual{\xi,\tilde x}_X \leq \norm{\tilde x}_X\qquad\text{for all }\tilde x\in X\setminus\{0\}
    \end{equation*}
    (since the inequality is trivial for $\tilde x=0$), which by definition of the operator norm holds if and only if $\norm{\xi}_{X^*} \leq 1$.

    Let now $x\neq 0$ and consider $\xi\in\partial(\norm{\cdot}_X)(x)$. Successively inserting $\tilde x = 0$ and $\tilde x = 2x$ in the definition \eqref{eq:convex:def_subdiff} yields
    \begin{equation*}
        \norm{x}_X \leq \dual{\xi,x}_X = \dual{\xi,2x -x } \leq \norm{2x}_X - \norm{x}_X = \norm{x}_X,
    \end{equation*}
    i.e., $\dual{\xi,x}_X=\norm{x}_X$. Similarly, we have for all $\tilde x \in X$ that
    \begin{equation*}
        \dual{\xi,\tilde x}_X = \dual{\xi,(\tilde x+x)- x}_X
        \leq \norm{\tilde x+x}_X - \norm{x}_X \leq \norm{\tilde x}_X,
    \end{equation*}
    As in the case $x=0$, this implies that $\norm{\xi}_{X^*}\leq 1$. For $\tilde x = x/\norm{x}_{X}$ we further have that
    \begin{equation*}
        \dual{\xi,\tilde x}_X = \norm{x}_X^{-1}\dual{\xi,x}_X = \norm{x}_X^{-1} \norm{x}_X = 1.
    \end{equation*}
    Hence, $\norm{\xi}_{X^*}=1$ is in fact attained.

    Conversely, let $x^*\in X^*$ with $\dual{x^*,x}_X =\norm{x}_X$ and $\norm{x^*}_{X^*} = 1$. Then we obtain for all $\tilde x \in X$ from \eqref{eq:functan:cs_banach} the relation
    \begin{equation*}
        \dual{x^*,\tilde x- x}_X = \dual{x^*,\tilde x}_X -  \dual{x^*, x}_X\leq \norm{\tilde x}_X - \norm{x}_X,
    \end{equation*}
    and hence $x^*\in\partial(\norm{\cdot}_X)(x)$ by definition.
\end{proof}
In particular, we obtain for $X=\R$ the subdifferential of the absolute value function as
\begin{equation}\label{eq:convex:subdiff_abs}
    \partial(|\cdot|)(t) = \sign(t) :=
    \begin{cases}
        \{1\} & \text{if }t>0,\\
        \{{-}1\} & \text{if }t<0,\\
        [-1,1] & \text{if }t=0.
    \end{cases}
\end{equation}
We can also give a more explicit characterization of the subdifferential of the indicator functional of a convex set $C\subset X$:
For any $x\in C=\dom\delta_C$, we have that
\begin{equation*}
    \begin{aligned}
        x^*\in\partial\delta_C(x) &\equivalent \dual{x^*,\tilde x - x}_X \leq \delta_C(\tilde x) &&\text{for all }\tilde x \in X\\
        &\equivalent \dual{x^*,\tilde x - x}_X \leq 0 &&\text{for all }\tilde x \in C,
    \end{aligned}
\end{equation*}
since the first inequality is trivially satisfied for all $\tilde x\notin C$.
The set $\partial\delta_C(x)$ is also called the \emph{normal cone} to $C$ at $x$. Depending on the set $C$, this can be made even more explicit. Let $X=\R$ and $C=[-1,1]$, and let $t\in C$. Then we have  $\xi \in \partial\delta_{[-1,1]}(t)$ if and only if $\xi(\tilde t -t)\leq 0$ for all $\tilde t \in [-1,1]$. We proceed by distinguishing three cases.
\begin{enumerate}[{Case }1:]
    \item  $t=1$. Then $\tilde t - t \in [-2,0]$, and hence the product is positive if and only if $\xi\geq 0$.
    \item  $t=-1$. Then $\tilde t - t \in [0,2]$, and hence the product is positive if and only if $\xi\leq 0$.
    \item  $t\in(-1,1)$. Then $\tilde t -t$ can be positive as well as negative, and hence only $\xi = 0$ is possible.
\end{enumerate}
We thus obtain that
\begin{equation*}
    \partial\delta_{[-1,1]}(t) = \begin{cases}
        [0,\infty) & \text{if }t=1,\\
        (-\infty,0] & \text{if }t=-1,\\
        \{0\} & \text{if }t\in (-1,1),\\
        \emptyset & \text{if }t \in \R\setminus[-1,1].
    \end{cases}
\end{equation*}
Readers familiar with (non)linear optimization will recognize these as the \emph{complementarity conditions} for Lagrange multipliers corresponding to the inequalities $-1\leq t \leq 1$.

The following result furnishes a crucial link between finite- and infinite-dimensional convex optimization. We again assume (as we will from now on) that $\Omega\subset \R^n$ is open and bounded.
\begin{theorem}\label{lem:lebesgue:subdiff}
    Let $f:\R\to\Rbar$ be proper, convex, and lower semicontinuous, and let $F:L^p(\Omega)\to\Rbar$ with $1\leq p <\infty$ be as in \cref{lem:lebesgue:lsc}. Then we have for all $u\in\dom F$ with $q:=\frac{p}{p-1}$ that
    \begin{equation*}
        \partial F(u)  = \setof{u^*\in L^q(\Omega)}{u^*(x) \in\partial f(u(x))\quad\text{for almost every } x\in\Omega}.
    \end{equation*}
\end{theorem}
\begin{proof}
    Let $u, \tilde u\in \dom F$, i.e., $f\circ u, f \circ \tilde u\in L^1(\Omega)$ (otherwise there is nothing to show), and let $u^*\in L^q(\Omega)$ be arbitrary.
    If $u^*\in L^q(\Omega)$ satisfies $u^*(x)\in \partial f(u(x))$ almost everywhere, we can insert $\tilde u(x)$ into the definition and integrate over all $x\in\Omega$ to obtain
    \begin{equation*}
        F(\tilde u)-F(u) =
        \int_\Omega f(\tilde u(x))-f(u(x))\,dx \geq
        \int_\Omega u^*(x)(\tilde u(x)-u(x))\,dx =
        \dual{u^*,\tilde u-u}_{L^p},
    \end{equation*}
    i.e., $u^*\in\partial F(u)$.

    Conversely, let $u^*\in\partial F(u)$. Then by definition it holds that
    \begin{equation*}
        \int_\Omega u^*(x)(\tilde u(x)-u(x))\,dx \leq \int_\Omega f(\tilde u(x))-f(u(x))\,dx \quad\text{for all }\tilde u\in L^p(\Omega).
    \end{equation*}
    Let now $t\in\R$ be arbitrary and let $A\subset \Omega$ be an arbitrary measurable set. Setting
    \begin{equation*}
        \tilde u(x) := \begin{cases} t & \text{if }x\in A,\\ u(x) & \text{if }x\notin A, \end{cases}
    \end{equation*}
    the above inequality implies due to $\tilde u\in L^p(\Omega)$ that
    \begin{equation*}
        \int_A u^*(x)(t-u(x))\,dx \leq \int_A f(t)-f(u(x))\,dx.
    \end{equation*}
    Since $A$ was arbitrary, it must hold that
    \begin{equation*}
        u^*(x)(t-u(x))\leq f(t)-f(u(x)) \qquad\text{for almost every }x\in\Omega.
    \end{equation*}
    Furthermore, since $t\in\R$ was arbitrary, we obtain that $u^*(x) \in\partial u(x)$ for almost every $x\in\Omega$.
\end{proof}
A similar proof shows that for $F:\R^N\to\Rbar$ with $F(x) = \sum_{i=1}^N f_i(x_i)$ and $f_i:\R\to\Rbar$ convex, we have for any $x\in \dom F$ that
\begin{equation*}
    \partial F(x) = \setof{x^*\in\R^N}{x^*_i\in \partial f_i(x_i),\quad 1\leq i \leq N}.
\end{equation*}
Together with the above examples, this yields componentwise expressions for the subdifferential of the norm $\norm{\cdot}_1$ as well as of the indicator functional of the unit ball with respect to the supremum norm in $\R^N$.

\bigskip

As for classical derivatives, one rarely obtains subdifferentials from the fundamental definition but rather by applying calculus rules. It stands to reason that these are more difficult to derive the weaker the derivative concept is (i.e., the more functionals are differentiable in that sense).
For convex subdifferentials, the following two rules still follow directly from the definition.
\begin{lemma}\label{lem:convex:subdiff_calc}
    Let $F:X\to\Rbar$ be convex and $x\in \dom F$. Then,
    \begin{enumerate}[(i)]
        \item $\partial(\lambda F)(x) = \lambda(\partial F(x)):= \setof{\lambda\xi}{\xi\in\partial F(x)}$ for $\lambda >0$;
        \item $\partial F(\cdot + x_0)(x) = \partial F(x+x_0)$ for $x_0\in X$ with $x+x_0\in\dom F$.
    \end{enumerate}
\end{lemma}

The sum rule is already considerably more delicate.
\begin{theorem}[sum rule]\label{thm:convex:sum}
    Let $F,G:X\to\Rbar$ be convex and lower semicontinuous, and $x\in \dom F\cap \dom G$. Then,
    \begin{equation*}
        \partial F(x) + \partial G(x) \subset \partial (F+G)(x),
    \end{equation*}
    with equality if there exists an $x_0 \in (\dom F)^o\cap \dom G$.
\end{theorem}
\begin{proof}
    The inclusion follows directly from adding the definitions of the two subdifferentials.
    Let therefore $x\in \dom F\cap \dom G$ and $\xi\in \partial(F+G)(x)$, i.e., satisfying
    \begin{equation}\label{eq:convex:sum:fpg}
        \dual{\xi,\tilde x - x}_X \leq(F(\tilde x) + G(\tilde x)) - (F(x) + G(x)) \quad\text{for all }\tilde x \in X.
    \end{equation}
    Our goal is now to use (as in the proof of \cref{lem:convex:gamma}) the characterization of convex functionals via their epigraph together with the Hahn--Banach separation theorem to construct a bounded linear functional $\zeta\in \partial G(x)\subset X^*$ with $\xi-\zeta\in \partial F(x)$, i.e.,
    \begin{equation*}
        \begin{aligned}
            F(\tilde x) - F(x) - \dual{\xi,\tilde x-x}_X &\geq \dual{\zeta,x- \tilde x}_X \quad\text{for all }\tilde x \in \dom F,\\
            G(x) - G(\tilde x) &\leq\dual{\zeta,x- \tilde x}_X \quad\text{for all }\tilde x \in \dom G.
        \end{aligned}
    \end{equation*}
    For that purpose, we define the sets
    \begin{align*}
        C_1 &:=  \setof{(\tilde x,t-(F(x) -\dual{\xi,x}_X))}{F(\tilde x) - \dual{\xi,\tilde x}_X\leq t},\\
        C_2 &:= \setof{(\tilde x,G(x)-t)}{G(\tilde x) \leq t},
    \end{align*}
    i.e.,
    \begin{equation*}
        C_1 =\epi (F-\xi) - (0,F(x) - \dual{\xi,x}_X),\qquad
        C_2 = -(\epi G - (0,G(x))).
    \end{equation*}
    To apply \cref{lem:convex:eidelheit} to these sets, we have to verify its conditions.
    \begin{enumerate}
        \item Since $x\in \dom F\cap \dom G$, both $C_1$ and $C_2$ are nonempty.
            Furthermore, since $F$ and $G$ are convex, it is straightforward (if tedious) to verify from the definition that $C_1$ and $C_2$ are convex.
        \item The critical point is of course the nonemptiness of $C_1^o$, for which we argue as follows. Since $x_0\in (\dom F)^o$, we know from \cref{lem:convex:bounded} that $F$ is bounded in an open neighborhood $U\subset (\dom F)^o$ of $x_0$. We can thus find an open interval $I\subset \R$ such that $U\times I \subset C_1$. Since $U\times I$ is open by the definition of the product topology on $X\times \R$, any $(x_0,\alpha)$ with $\alpha\in I$ is an interior point of $C_1$.
        \item It remains to show that $C_1^o\cap C_2=\emptyset$.
            Assume there exists a $(\tilde x,\alpha) \in C_1^o \cap C_2$. But then the definitions of these sets and of the product topology imply that
            \begin{equation*}
                F(\tilde x) - F(x) - \dual{\xi,\tilde x-x}_X < \alpha
                \leq G(x) - G(\tilde x),
            \end{equation*}
            contradicting \eqref{eq:convex:sum:fpg}. Hence $C_1^o$ and $C_2$ are disjoint.
    \end{enumerate}
    \Cref{lem:convex:eidelheit} therefore yields a pair $(x^*,s)\in (X\times \R)^*\setminus\{(0,0)\}$ and a $\lambda\in\R$ with
    \begin{align}
        \label{eq:convex:sum:hb1}
        \dual{x^*,\tilde x}_X + s(t-(F(x) -\dual{\xi,x}_X)) &\leq \lambda, \quad\tilde{x}\in \dom F, t\geq F(\tilde x) - \dual{\xi,\tilde x}_X,\\
        \label{eq:convex:sum:hb2}
        \dual{x^*,\tilde x}_X + s(G(x)-t) &\geq \lambda, \quad \tilde{x}\in \dom G, t\geq G(\tilde x).
    \end{align}

    We now show that $s<0$. If $s=0$, we can insert $\tilde x = x_0\in \dom F\cap \dom G$ to obtain the contradiction
    \begin{equation*}
        \dual{x^*,x_0}_X < \lambda \leq \dual{x^*,x_0}_X,
    \end{equation*}
    which follows since $(x_0,\alpha)$ for $\alpha$ large enough is an interior point of $C_1$ and hence can be \emph{strictly} separated from $C_2$ by \cref{thm:hb_separation}.
    If $s>0$, choosing $t>F(x) - \dual{\xi,x}_X$ makes the term in parentheses in \eqref{eq:convex:sum:hb1} strictly positive, and taking $t\to \infty$ with fixed $\tilde x$ leads to a contradiction to the boundedness by $\lambda$.

    Hence $s<0$, and \eqref{eq:convex:sum:hb1} with $t=F(\tilde x) - \dual{\xi,\tilde x}_X$ and \eqref{eq:convex:sum:hb2} with $t=G(\tilde x)$ imply that
    \begin{align*}
        F(\tilde x) - F(x) + \dual{\xi,\tilde x - x}_X &\geq s^{-1}(\lambda -\dual{x^*,\tilde x}_X),\quad\text{for all } \tilde x\in \dom F,\\
        G(x) - G(\tilde x) &\leq s^{-1}(\lambda -\dual{x^*,\tilde x}_X),\quad \text{for all }\tilde x\in \dom G.
    \end{align*}
    Taking $\tilde x = x\in \dom F\cap \dom G$ in both inequalities immediately yields that
    $\lambda = \dual{x^*,x}_X$. Hence, $\zeta = s^{-1}x^*$ is the desired functional with $(\xi - \zeta)\in \partial F(x)$ and $\zeta \in \partial G(x)$, i.e., $\xi \in \partial F(x) + \partial G(x)$.
\end{proof}
By induction, we obtain from this sum rules for an arbitrary (finite) number of functionals (where $x_0$ has to be an interior point of all but one effective domain). A chain rule for linear operators can be proved similarly.
\begin{theorem}[chain rule]\label{thm:convex:chain}
    Let $A\in L(X,Y)$, $F:Y\to\Rbar$ be convex and lower semicontinuous, and $x\in \dom (F\circ A)$. Then,
    \begin{equation*}
        \partial (F\circ A)(x) \supset A^*\partial F(Ax) := \setof{A^*y^*}{y^* \in \partial F(Ax)}
    \end{equation*}
    with equality if there exists an $x_0\in X$ with $Ax_0\in (\dom F)^o$.
\end{theorem}
\begin{proof}
    The inclusion is again a direct consequence of the definition: If $\eta \in \partial F(Ax)\subset Y^*$, we in particular have for all $\tilde y = A\tilde x\in Y$ with $\tilde x\in X$ that
    \begin{equation*}
        F(A\tilde x) - F(Ax) \geq \dual{\eta,A\tilde x - Ax}_Y = \dual{A^*\eta,\tilde x-x}_X,
    \end{equation*}
    i.e., $\xi :=  A^*\eta\in \partial (F\circ A) \subset X^*$.

    Let now $x\in\dom (F\circ A)$ and $\xi\in \partial(F\circ A)(x)$, i.e.,
    \begin{equation*}
        F(Ax) + \dual{\xi,\tilde x-x}_X \leq F(A\tilde x) \quad\text{for all }\tilde x\in X.
    \end{equation*}
    We now construct a $y^*\in\partial F(Kx)$ with $x^* = K^*y^*$ by applying the sum rule to
    \begin{equation*}
        H:X\times Y\to\Rbar, \qquad  (x,y) \mapsto F(y) + \delta_{\graph A}(x,y).
    \end{equation*}
    (This technique of getting rid of the operator composition by working in the graph space is called \enquote{lifting}.)
    Since $A$ is linear, $\graph A$ and hence $\delta_{\graph A}$ are convex. Furthermore, $Ax\in\dom F$ by assumption and thus $(x,Ax) \in \dom H$.

    We begin by showing that $\xi\in\partial (F\circ A)(x)$ if and only if $(\xi,0)\in\partial H(x,Ax)$. First, let $(\xi, 0)\in \partial H(x,Ax)$. Then we have for all $\tilde x\in X,\tilde y\in Y$ that
    \begin{equation*}
        \dual{\xi,\tilde x- x}_X + \dual{0,\tilde y -Ax}_Y \leq F(\tilde y) - F(Ax) + \delta_{\graph A}(\tilde x,\tilde y) - \delta_{\graph A}(x,Ax).
    \end{equation*}
    In particular, this holds for all $\tilde y \in \rg(A)=\setof{A\tilde x}{\tilde x\in X}$. By $\delta_{\graph A}(\tilde x,A\tilde x) = 0$ we thus obtain that
    \begin{equation*}
        \dual{\xi,\tilde x- x}_X \leq F(A \tilde x) - F(Ax) \quad\text{for all }\tilde x\in X,
    \end{equation*}
    i.e., $\xi \in \partial(F\circ A)(x)$. Conversely, let $\xi\in \partial(F\circ A)(x)$. Since $\delta_{\graph A}(x,Ax) = 0$ and $\delta_{\graph A}(\tilde x,\tilde y)\geq 0$, it then follows for all $\tilde x \in X$ and $\tilde y \in Y$ that
    \begin{equation*}
        \begin{aligned}
            \dual{\xi,\tilde x- x}_X + \dual{0,\tilde y -Ax}_Y
            &= \dual{\xi,\tilde x- x}_X\\
            &\leq F(A \tilde x) - F(Ax) + \delta_{\graph A}(\tilde x,\tilde y)  - \delta_{\graph A}(x,Ax)\\
            &= F( \tilde y) - F(Ax) + \delta_{\graph A}(\tilde x,\tilde y) - \delta_{\graph A}(x,Ax),
        \end{aligned}
    \end{equation*}
    where we have used that the last equality holds trivially as $\infty=\infty$ for $\tilde y \neq A\tilde x$. Hence, $(\xi,0)\in \partial H(x,Ax)$.

    We now consider $\tilde F:X\times Y\to \Rbar$, $(x,y)\mapsto F(y)$, and  $(x_0,Ax_0)\in \graph A=\dom\delta_{\graph A}$.
    Since $Ax_0\in (\dom F)^o \subset Y$ by assumption, $(x_0,Ax_0)\in (\dom \tilde F)^o=X\times (\dom F)^o\subset X\times Y$ as well.
    We can thus apply \cref{thm:convex:sum} to obtain
    \begin{equation*}
        (\xi,0) \in \partial H(x,Ax) = \partial \tilde F(x,Ax) + \partial \delta_{\graph A}(x,Ax),
    \end{equation*}
    i.e., $(\xi,0)=(x^*,y^*)+(w^*,z^*)$ for some $(x^*,y^*)\in \partial \tilde F(x,Ax)$ and $(w^*,z^*)\in\partial \delta_{\graph A}(x,Ax)$.

    Finally, we \enquote{collapse} these subdifferentials back to the individual spaces to obtain the desired characterization.
    First, we have $(x^*,y^*)\in \partial \tilde F(x,Ax)$ if and only if
    \begin{equation*}
        \dual{x^*,\tilde x -x}_X + \dual{y^*,\tilde y-Ax}_Y \leq F(\tilde y)- F(Ax)\quad\text{for all }\tilde x\in X,\tilde y \in Y.
    \end{equation*}
    Fixing in turn $\tilde x = x$ and $\tilde y=Ax$ implies that $y^*\in\partial F(Ax)$  and $x^* =0$, respectively.
    Second, $(w^*,z^*)\in \partial \delta_{\graph A}(x,Ax)$ if and only if
    \begin{equation*}
        \dual{w^*,\tilde x -x}_X + \dual{z^*,\tilde y-Ax}_Y \leq 0 \quad\text{for all }(\tilde x,\tilde y)\in \graph A,
    \end{equation*}
    i.e., for all $\tilde x\in X$ and $\tilde y = A\tilde x$. Therefore,
    \begin{equation*}
        \dual{w^*+A^*z^*,\tilde x - x}_X \leq 0 \quad\text{for all }\tilde x\in X
    \end{equation*}
    and hence $w^*=-A^*z^*\in X^*$. Together we obtain
    \begin{equation*}
        (\xi,0) = (0,y^*) + (-A^*z^*,z^*),
    \end{equation*}
    which implies $y^* = -z^*$ and thus that $\xi = -A^*z^* = A^*y^*$ with $y^*\in\partial F(Ax)$ as claimed.
\end{proof}

The Fermat principle together with the sum rule yields the following characterization of minimizers of convex functionals under convex constraints.
\begin{cor}
    Let $U\subset X$ be nonempty, convex, and closed, and let $F:X\to\Rbar$ be proper, convex, and lower semicontinuous. If there exists an $x_0 \in U^o \cap \dom F$, then $\bar x\in U$ solves
    \begin{equation*}
        \min_{x\in U} F(x)
    \end{equation*}
    if and only if there exists a $\xi\in X^*$ with
    \begin{equation}\label{eq:convex:kkt}
        \left\{\begin{aligned}
                &\xi\in \partial F(\bar x),\\
                &\dual{\xi,\tilde x -x} \geq 0 \quad\text{for all } \tilde x \in U.
        \end{aligned}\right.
    \end{equation}
\end{cor}
\begin{proof}
    Due to the assumptions on $F$ and $U$, we can apply \cref{thm:convex:fermat} to $J:=F+\delta_U$. Furthermore, since $x_0 \in U^o = (\dom \delta_U)^o$, we can also apply \cref{thm:convex:sum}. Hence $F$ has a minimum in $\bar x$ if and only if
    \begin{equation*}
        0\in \partial J(\bar x) = \partial F(\bar x) + \partial \delta_U(\bar x).
    \end{equation*}
    Together with the characterization of subdifferentials of indicator functionals as normal cones, this yields \eqref{eq:convex:kkt}.
\end{proof}
If $F:X\to\R$ is Gâteaux differentiable (and hence finite-valued), \eqref{eq:convex:kkt} coincide with the classical \emph{Karush--Kuhn--Tucker conditions}; the existence of an interior point $x_0\in U^o$ is an analogue of the \emph{Slater condition} needed to show existence of the Lagrange multiplier $\xi$ for the inequality constraints.

\chapter{Fenchel duality}

A particularly useful calculus rule connects the convex subdifferential with the so-called Fenchel--Legendre transform.
Let $X$ be a normed vector space and $F:X\to\Rbar$ be proper but not necessarily convex. We then define the \emph{Fenchel conjugate} of $F$ as
\begin{equation*}
    F^*:X^*\to \Rbar,\qquad F^*(x^*) = \sup_{x\in X}\, \dual{x^*,x}_X - F(x).
\end{equation*}
(Since $\dom F\neq \emptyset$ is excluded, we have that $F^*(x^*)>-\infty$ for all $x^*\in X^*$, and hence the definition is meaningful.) \Cref{lem:convex:func}\,(v) and \cref{lem:variation:wlsc}\,(v) immediately imply that $F^*$ is always convex and lower semicontinuous (as long as $F$ is indeed proper). If $F$ is bounded from below by an affine functional (which is always the case if $F$ is proper, convex, and lower semicontinuous by \cref{lem:convex:gamma}), then $F^*$ is proper as well.
Finally, the definition directly yields the \emph{Fenchel--Young inequality}
\begin{equation}\label{eq:convex:fenchel-young}
    \dual{x^*,x}_X \leq F(x) + F^*(x^*)\qquad\text{for all }x\in X, x^*\in X^*.
\end{equation}

Intuitively, $F^*(x^*)$ is the (negative of the) affine part of the tangent to $F$ (in the point $x$ in which the supremum is attained) with slope $x^*$. Similarly, we define the Fenchel conjugate of $F:X^*\to\Rbar$ (i.e., if $F$ is defined on some dual space) as
\begin{equation*}
    F^*:X\to \Rbar,\qquad F^*(x) = \sup_{x^*\in X^*} \dual{x^*,x}_X - F(x^*).
\end{equation*}
The point of this convention is that even in nonreflexive spaces, the \emph{biconjugate} $F^{**}:=(F^*)^*$ is again defined on $X$ (rather than $X^{**}\supset X$). Intuitively, $F^{**}$ is the convex hull of $F$, which by \cref{lem:convex:gamma} coincides with $F$ itself if $F$ is convex.
\begin{theorem}[Fenchel--Moreau--Rockafellar]\label{thm:convex:moreau}
    Let $F:X\to\Rbar$ be proper. Then,
    \begin{enumerate}[(i)]
        \item $F^{**}\leq F$;
        \item $F^{**} = F^{\Gamma}$;
        \item $F^{**} = F$ if and only if $F$ is convex and lower semicontinuous.
    \end{enumerate}
\end{theorem}
\begin{proof}
    For (i), we take the supremum over all $x^*\in X^*$ in the Fenchel--Young inequality \eqref{eq:convex:fenchel-young} and obtain that
    \begin{equation*}
        F(x) \geq \sup_{x^*\in X^*} \dual{x^*,x}_X -F^*(x^*) = F^{**}(x).
    \end{equation*}

    For (ii), we first note that $F^{**}$ is convex and lower semicontinuous by definition as a Fenchel conjugate as well as proper by (i). Hence, \cref{lem:convex:gamma} yields that
    \begin{equation*}
        F^{**}(x) = (F^{**})^\Gamma(x) = \sup\setof{a(x)}{a:X\to\R\text{ affine with } a\leq F^{**}}.
    \end{equation*}
    We now show that we can replace $F^{**}$ with $F$ on the right-hand side. For this, let $a(x)=\dual{x^*,x}_X-\alpha$ with arbitrary $x^*\in X^*$ and $\alpha\in\R$.
    If $a\leq F^{**}$, then (i) implies that $a\leq F$. Conversely, if $a\leq F$, we have that  $\dual{x^*,x}_X-F(x)\leq \alpha$ for all $x\in X$, and taking the supremum over all $x\in X$ yields that $\alpha\geq F^{*}(x^*)$. By definition of $F^{**}$, we thus obtain that
    \begin{equation*}
        a(x) = \dual{x^*,x}_X - \alpha \leq \dual{x^*,x}_X - F^*(x^*) \leq F^{**}(x) \quad\text{for all }x\in X,
    \end{equation*}
    i.e., $a\leq F^{**}$.

    Statement (iii) now directly follows from (ii) and \cref{lem:convex:gamma}.
\end{proof}

We again consider some relevant examples.
\begin{example}\label{ex:convex:fenchel}~
    \begin{enumerate}[(i)]
        \item Let $X$ be a Hilbert space and $F(x) = \frac12\norm{x}_{X}^2$.
            Using the \nameref{thm:frechetriesz} \cref{thm:frechetriesz}, we identify $X$ with its dual $X^*$ and can express the duality pairing using the inner product. Since $F$ is Fréchet differentiable with gradient $\nabla F(x) = x$, the solution $\bar x\in X$ to
            \begin{equation*}
                \sup_{x\in X} \inner{x^*,x}_X - \tfrac12\inner{x,x}_X
            \end{equation*}
            for given $x^*\in X$ has to satisfy the Fermat principle, i.e., $\bar x=x^*$. Inserting this into the definition and simplifying yields the Fenchel conjugate
            \begin{equation*}
                F^*:X\to\R,\qquad F^*(x^*) = \tfrac12\norm{x^*}_{X}^2.
            \end{equation*}

        \item Let $B_X$ be the unit ball in the normed vector space $X$ and take $F=\delta_{B_X}$. Then we have for any $x^*\in X^*$ that
            \begin{equation*}
                (\delta_{B_X})^*(x^*) = \sup_{x\in X} \, \dual{x^*,x}_X - \delta_{B_X}(x) = \sup_{\norm{x}_{X}\leq 1}  \dual{x^*,x}_X = \norm{x^*}_{X^*}.
            \end{equation*}
            Similarly, one shows using the definition of the Fenchel conjugate in dual spaces and \cref{cor:functan:norm_dual} that $(\delta_{B_{X^*}})^*(x) = \norm{x}_X$.

        \item Let $X$ be a normed vector space and take $F(x) = \norm{x}_X$. We now distinguish two cases for a given $x^*\in X^*$.
            \begin{enumerate}[{Case} 1:]
                \item $\norm{x^*}_{X^*} \leq 1$. Then it follows from \eqref{eq:functan:cs_banach} that $\dual{x^*,x}_X- \norm{x}_X\leq 0$ for all $x\in X$. Furthermore, $\dual{x^*,0} = 0 = \norm{0}_X$, which implies that
                    \begin{equation*}
                        F^*(x^*) = \sup_{x\in X}  \dual{x^*,x}_X -\norm{x}_X = 0.
                    \end{equation*}
                \item $\norm{x^*}_{X^*} >1$. Then by definition of the dual norm, there exists an $x_0\in X$ with $\dual{x^*,x_0}_X > \norm{x_0}_X$. Hence, taking $t\to\infty$ in
                    \begin{equation*}
                        0<t(\dual{x^*,x_0}_X - \norm{x_0}_X) = \dual{x^*,tx_0}_X - \norm{tx_0}_X \leq F^*(x^*)
                    \end{equation*}
                    yields $F^*(x^*) = \infty$.
            \end{enumerate}
            Together we obtain that $F^* = \delta_{B_{X^*}}$.
            As above, a similar argument shows that $(\norm{\cdot}_{X^*})^* = \delta_{B_X}$.
    \end{enumerate}
\end{example}

As for convex subdifferentials, Fenchel conjugates of integral functionals can be computed pointwise.
\begin{theorem}\label{thm:lebesgue:fenchel}
    Let $f:\R\to\Rbar$ be measurable, proper and lower semicontinuous, and let $F:L^p(\Omega)\to\Rbar$ with $1\leq p <\infty$ be defined as in \cref{lem:lebesgue:lsc}. Then we have for $q=\frac{p}{p-1}$ that
    \begin{equation*}
        F^*:L^q(\Omega) \to\Rbar, \qquad F^*(u^*) =
        \int_\Omega f^*(u^*(x))\,dx .
    \end{equation*}
\end{theorem}
\begin{proof}
    We argue similarly as in the proof of \cref{lem:lebesgue:subdiff}, with some changes that are needed since measurability of $f\circ u$ does not immediately imply that of $f^*\circ u^*$. Let $u^*\in L^q(\Omega)$ be arbitrary and consider for all $x\in \Omega$ the functions
    \begin{align*}
        \phi(x) &:= \sup_{t\in\R} tu^*(x) - f(t) = f^*(u^*(x)),
        \shortintertext{as well as for $n\in \N$}
        \phi_n(x) &:=\sup_{|t|\leq n} tu^*(x) - f(t) \leq f^*(u^*(x)).
    \end{align*}
    By a measurable selection theorem (\cite[Theorem VIII.1.2]{Ekeland:1999a}), the pointwise supremum in the definition of $\phi_n$ is attained at some $t^*_x$ for almost every $x\in \Omega$ and defines a measurable mapping $x\mapsto u_n(x):=t^*_x$ with $\norm{u_n}_{L^\infty}\leq n$. This also implies that $\phi_n = u_n \cdot u^*-f\circ u_n$ is measurable.
    Furthermore, by assumption there exists a $t_0\in \dom f$, and hence
    $u_0 := t_0u^*(x) - f(t_0)$ is measurable and satisfies $u_0 \leq \phi_n(x)$ for all $n\geq |t_0|$.
    Finally, by construction, $\phi_n(x)$ is monotonically increasing and converges to $\phi(x)$ for all $x\in \Omega$. The sequence $\{\phi_n-u_0\}_{n\in\N}$ of functions is thus measurable and nonnegative, and the monotone convergence theorem yields that
    \begin{equation*}
        \int_\Omega \phi(x) - u_0(x) \,dx = \int_\Omega \sup_{n\in\N} \phi_n(x)-u_0(x)\,dx =
        \sup_{n\in\N} \int_\Omega \phi_n(x)-u_0(x)\,dx.
    \end{equation*}
    Hence the pointwise limit $\phi=f^*\circ u^*$ is measurable as well.

    The measurable selection theorem also yields that
    \begin{equation*}
        \begin{aligned}[t]
            \int_\Omega f^*(u^*(x))\,dx &= \sup_{n\in\N} \int_\Omega \sup_{|t|\leq n}\left\{ tu^*(x)-f(t)\right\}\,dx \\
            &= \sup_{n\in\N} \int_\Omega u^*(x)u_n(x)-f(u_n(x))\,dx \\
            &\leq \sup_{u\in L^p(\Omega)} \int_\Omega u^*(x)u(x)-f(u(x))\,dx = F^*(u^*),
        \end{aligned}
    \end{equation*}
    since $u_n \in L^\infty(\Omega)\subset L^p(\Omega)$ for all $n\in\N$.

    For the converse inequality, we can now proceed as in the proof of \cref{lem:lebesgue:subdiff}. For any $u\in L^p(\Omega)$ and $u^*\in L^q(\Omega)$, we have by the Fenchel--Young inequality \eqref{eq:convex:fenchel-young} applied to $f$ and $f^*$ that
    \begin{equation*}
        f(u(x)) + f^*(u^*(x)) \geq u^*(x)u(x)\quad \text{for almost every }x\in \Omega.
    \end{equation*}
    Since both sides are measurable, this implies that
    \begin{equation*}
        \int_\Omega f^*(u^*(x))\,dx \geq \int_\Omega u^*(x)u(x) - f(u(x))\,dx,
    \end{equation*}
    and taking the supremum over all $u\in L^p(\Omega)$ yields the claim.
\end{proof}

Fenchel conjugates satisfy a number of useful calculus rules, which follow directly from the properties of the supremum.
\begin{lemma}\label{lem:convex:fenchel_calc}
    Let $F:X\to\Rbar$ be proper. Then,
    \begin{enumerate}[(i)]
        \item $(\alpha F)^* = \alpha F^*\circ (\alpha^{-1} \Id)$ for any $\alpha >0$;
        \item $(F(\cdot + x_0) + \dual{x_0^*,\cdot}_X)^* = F^*(\cdot - x_0^*) - \dual{\cdot - x_0^*,x_0}_X$ for all $x_0\in X$, $x_0^*\in X^*$;
        \item $(F\circ A)^* = F^* \circ A^{-*}$ for continuously invertible $A\in L(Y,X)$ and $A^{-*}:=(A^{-1})^*$.
    \end{enumerate}
\end{lemma}
\begin{proof}
    \emph{(i):} For any $\alpha >0$, we have that
    \begin{equation*}
        (\alpha F)^*(x^*) =  \sup_{x\in X} \left(\alpha\dual{\alpha^{-1}x^*,x}_X - \alpha F(x)\right)
        = \alpha \sup_{x\in X} \left(\dual{\alpha^{-1}x^*,x}_X - F(x)\right)
        = \alpha F^* (\alpha^{-1} x^*).
    \end{equation*}

    \emph{(ii):} Since $\setof{x+x_0}{x\in X}=X$, we have that
    \begin{equation*}
        \begin{aligned}
            (F(\cdot + x_0) + \dual{x_0^*,\cdot}_X)^*(x^*) &= \sup_{x\in X}\  \dual{x^*,x}_X - F(x^* +x_0) - \dual{x_0^*,x_0}_X\\
            &= \sup_{x\in X} \big( \dual{x^*-x_0^*,x+x_0}_X - F(x +x_0)\big) - \dual{x^*-x_0^*,x_0}_X\\
            &= \sup_{\tilde x=x+x_0,x\in X}\, \big( \dual{x^*-x^*_0,\tilde x}_X - F(\tilde x)\big) - \dual{x^*-x_0^*,x_0}_X\\
            &= F^*(x^*-x_0^*) - \dual{x^*-x_0^*,x_0}_X.
        \end{aligned}
    \end{equation*}

    \emph{(iii):} Since $X=\rg A$, we have that
    \begin{equation*}
        \begin{split}
            \begin{aligned}[b]
                (F\circ A)^*(y^*) &= \sup_{y\in Y}\ \dual{y^*,A^{-1}Ay}_Y - F(Ay)\\
                & = \sup_{x=Ay,y\in Y} \dual{A^{-*}y^*,x}_X - F(x) = F^*(A^{-*}y^*).
            \end{aligned}
            \qedhere
        \end{split}
    \end{equation*}
\end{proof}

\bigskip

There are some obvious similarities between the definitions of the Fenchel conjugate and of the subdifferential, which yield the following very useful property.
\begin{lemma}[Fenchel--Young]\label{lem:convex:fenchel-young}
    Let $F:X\to\Rbar$ be proper, convex, and lower semicontinuous. Then the following statements are equivalent for any $x\in X$ and $x^*\in X^*$:
    \begin{enumerate}[(i)]
        \item $\dual{x^*,x}_X = F(x) + F^*(x^*)$;
        \item $x^*\in\partial F(x)$;
        \item $x\in\partial F^*(x^*)$.
    \end{enumerate}
\end{lemma}
\begin{proof}
    If (i) holds, the definition of $F^*$ as a supremum immediately implies that
    \begin{equation}\label{eq:convex:fy1}
        \dual{x^*,x}_X - F(x) = F^*(x^*) \geq  \dual{x^*,\tilde x}_X - F(\tilde x) \qquad\text{for all }\tilde x\in X,
    \end{equation}
    which again by definition is equivalent to (ii). Conversely, taking the supremum over all $\tilde x\in X$ in \eqref{eq:convex:fy1} yields
    \begin{equation*}
        \dual{x^*,x}_X \geq F(x) + F^*(x^*),
    \end{equation*}
    which together with the Fenchel--Young inequality \eqref{eq:convex:fenchel-young} leads to (i).

    Similarly, (i) in combination with \cref{thm:convex:moreau} implies that
    \begin{equation*}
        \dual{x^*,x}_X - F^*(x^*) = F(x) = F^{**}(x) \geq \dual{\tilde x^*,x}-F^*(\tilde x^*) \qquad\text{for all }\tilde x^*\in X^*,
    \end{equation*}
    yielding as above the equivalence of (i) and (iii).
\end{proof}
\begin{remark}\label{rem:convex:f-nonconvex}
    If $F$ is not convex, the above proof shows that we still have the equivalence (i)\,$\Leftrightarrow$\,(ii). Furthermore since always $F^{**}\leq F$ by \cref{thm:convex:moreau}\,(i), it still holds that (ii)\,$\implies$\,(iii). However, we can only conclude from (iii) that (ii) holds for $F^{**}\neq F$ in place of $F$. Applying \cref{lem:convex:fenchel-young} to nonconvex functionals therefore inevitably introduces a \emph{convexification} (by replacing the nonconvex $F$ with its convex envelope $F^{**}$).
\end{remark}
\begin{remark}
    If $X$ is not reflexive, $x\in\partial F^*(x^*)\subset X^{**}$ in (iii) has to be understood via the canonical injection, i.e., as
    \begin{equation*}
        \dual{J(x),\tilde x^*-x^*}_{X^*} = \dual{\tilde x^* -x^*,x}_X \leq F^*(\tilde x^*) - F^*(x^*) \quad\text{for all $\tilde x^* \in X$.}
    \end{equation*}
    Using (iii) to conclude equality in (i) or, equivalently, the subdifferential inclusion (ii) therefore requires the additional condition that $x\in X\subset X^{**}$. Conversely, if (i) or (ii) hold, (iii) also guarantees that the subgradient $x\in \partial F^*(x^*)\cap X$, which is a stronger fact.
    (Similar statements apply to $F:X^*\to\Rbar$ and $F^*:X\to\Rbar$.)
\end{remark}

\Cref{lem:convex:fenchel-young} plays the role of a \enquote{convex inverse function theorem} and can be used to, e.g., replace the subdifferential of a (complicated) norm with that of a (simpler) conjugate indicator functional (or vice versa).
For example, given a problem of the form
\begin{equation}\label{eq:convex:primal}
    \inf_{x\in X} F(x) + G(Ax)
\end{equation}
for $F:X\to\Rbar$ and $G:Y\to\Rbar$ proper, convex, and lower semicontinuous, and $A\in L(X,Y)$, we can use \cref{thm:convex:moreau} to replace $G$ with the definition of $G^{**}$ and obtain
\begin{equation*}
    \inf_{x\in X}\sup_{Y^*\in Y^*} F(x) + \dual{y^*,Ax}_Y - G^*(y^*).
\end{equation*}
If(!) we were now able to exchange $\inf$ and $\sup$, we could write (with $\inf F = -\sup (-F)$)
\begin{equation*}
    \begin{aligned}
        \inf_{x\in X}\sup_{y^*\in y^*} F(x) + \dual{y^*,Ax}_Y - G^*(y^*)
        &=\sup_{y^*\in Y^*} \inf_{x\in X} F(x) + \dual{y^*,Ax}_Y - G^*(y^*) \\
        &=\sup_{y^*\in Y^*} -\left(\sup_{x\in X} -F(x) + \dual{-A^*y^*,x}_X\right) - G^*(y^*).
    \end{aligned}
\end{equation*}
By definition of $F^*$, we thus obtain the \emph{dual problem}
\begin{equation}\label{eq:convex:dual}
    \sup_{y^*\in Y^*} -F^*(-A^*y^*) - G^*(y^*).
\end{equation}
As a side effect, we have shifted the operator $A$ from $G$ to $F^*$ without having to invert it.

The following theorem uses in an elegant way the Fermat principle, the sum and chain rules, and the Fenchel--Young equality to derive sufficient conditions for the exchangeability.
\begin{theorem}[Fenchel--Rockafellar]\label{thm:convex:fenchel}
    Let $X$ and $Y$ be normed vector spaces, $F:X\to\Rbar$ and $G:Y\to \Rbar$ be proper, convex, and lower semicontinuous, and $A\in L(X,Y)$. Assume furthermore that
    \begin{enumerate}[(i)]
        \item the \emph{primal problem} \eqref{eq:convex:primal} admits a solution $\bar x \in X$;
        \item there exists an $x_0\in \dom F \cap \dom (G\circ A)$ with $Ax_0\in (\dom G)^o$ .
    \end{enumerate}
    Then, the dual problem \eqref{eq:convex:dual} admits a solution $\bar {y}^*\in Y^*$ and
    \begin{equation}\label{eq:convex:fenchel:equal}
        \min_{x\in X} F(x) + G(Ax) = \max_{y^*\in Y^*} -F^*(-A^*y^*) - G^*(y^*).
    \end{equation}
    Furthermore, $\bar x$ and $\bar{y}^*$ are solutions to \eqref{eq:convex:primal} and \eqref{eq:convex:dual}, respectively, if and only if
    \begin{equation}\label{eq:convex:extremal}
        \left\{
            \begin{aligned}
                -A^*\bar{y}^* &\in \partial F(\bar x),\\
                \bar{y}^* &\in \partial G(A\bar x).
            \end{aligned}
        \right.
    \end{equation}
\end{theorem}
\begin{proof}
    \Cref{thm:convex:fermat} states that $\bar x\in X$ is a solution to \eqref{eq:convex:primal} if and only if $0\in \partial (F + G\circ A)(\bar x)$. By assumption (ii), \cref{thm:convex:sum,thm:convex:chain} are applicable, and we thus obtain that
    \begin{equation*}
        0\in  \partial (F + G\circ A)(\bar x) = \partial F(\bar x) + A^*\partial G(A\bar x),
    \end{equation*}
    which implies that there exists a $\bar{y}^*\in \partial G(A\bar x)$ with $-A^*\bar{y}^* \in \partial F(\bar x)$, i.e., satisfying \eqref{eq:convex:extremal}.

    The relations \eqref{eq:convex:extremal} together with \cref{lem:convex:fenchel-young} further imply equality in the Fenchel--Young inequalities for $F$ and $G$, i.e.,
    \begin{equation}\label{eq:convex:fenchel0}
        \left\{
            \begin{aligned}
                \dual{-A^*\bar y^*,\bar x}_X &=F(\bar x) + F^*(-A^*\bar y^*),\\
                \dual{\bar y^*,A\bar x}_Y &=G(A\bar x) + G^*(\bar y^*).
            \end{aligned}
        \right.
    \end{equation}
    Adding both equations now yields
    \begin{equation}\label{eq:convex:fenchel1}
        F(\bar x) + G(A\bar x) =  -F^*(-A^*\bar y^*) - G^*(\bar y^*).
    \end{equation}

    It remains to show that $\bar y^*$ is a solution to \eqref{eq:convex:dual}.
    For this purpose, we introduce the \emph{Lagrangian}
    \begin{equation*}
        L:X\times Y^*\to\Rbar,\qquad L(x,y^*) =  F(x) + \dual{y^*,Ax}_Y - G^*(y^*).
    \end{equation*}
    For all  $\tilde x\in X$ and $\tilde y^*\in Y^*$, we always have that
    \begin{equation*}
        \sup_{y^*\in Y^*} L(\tilde x,y^*) \geq L(\tilde x,\tilde y^*) \geq \inf_{x\in X}  L(x,\tilde y^*),
    \end{equation*}
    and hence (taking the infimum over all $\tilde x$ in the first and the supremum over all $\tilde y^*$ in the second inequality) that
    \begin{equation*}
        \inf_{x\in X} \sup_{y^*\in Y^*} L(x,y^*) \geq  \sup_{y^*\in Y^*} \inf_{x\in X} L(x,y^*).
    \end{equation*}
    We thus obtain that
    \begin{equation}\label{eq:convex:fenchel2}
        \begin{aligned}[t]
            F(\bar x) + G(A\bar x) &=   \inf_{x\in X}\sup_{Y^*\in Y^*} F(x) + \dual{y^*,Ax}_Y - G^*(y^*)\\
            &\geq  \sup_{Y^*\in Y^*} \inf_{x\in X} F(x) + \dual{y^*,Ax}_Y - G^*(y^*)\\
            & = \sup_{y^*\in Y^*} -F^*(-A^*y^*) - G^*(y^*).
        \end{aligned}
    \end{equation}
    Combining this with \eqref{eq:convex:fenchel1} yields that
    \begin{equation*}
        -F^*(-A^*\bar{y}^*) - G^*(\bar{y}^*) =  F(\bar x) + G(A\bar x) \geq \sup_{y^*\in Y^*} -F^*(-A^*y^*) - G^*(y^*),
    \end{equation*}
    i.e., $\bar{y}^*$ is a solution to \eqref{eq:convex:dual}, and hence \eqref{eq:convex:fenchel:equal} follows from \eqref{eq:convex:fenchel1}.

    Finally, if $\bar x\in X$ and $\bar y^*\in Y^*$ are solutions to \eqref{eq:convex:primal} and \eqref{eq:convex:dual}, respectively, then \eqref{eq:convex:fenchel:equal} implies \eqref{eq:convex:fenchel1}. Together with the productive zero, this implies that
    \begin{equation*}
        0 = \left[ F(\bar x) + F^*(-A^*\bar y^*) -\dual{-A^*\bar y^*,\bar x}_X \right]+
        \left[G(A\bar x) + G^*(\bar y^*) - \dual{\bar y^*,A\bar x}_Y\right].
    \end{equation*}
    Since both brackets have to be nonnegative due to the Fenchel--Young inequality, they each have to be zero. We therefore deduce that \eqref{eq:convex:fenchel0} holds, and hence \cref{lem:convex:fenchel-young} implies~\eqref{eq:convex:extremal}.
\end{proof}
The relations \eqref{eq:convex:extremal} are referred to as \emph{Fenchel extremality conditions}; we can use \cref{lem:convex:fenchel-young} to generate further, equivalent, optimality conditions by inverting one or the other subdifferential inclusion.
We will later exploit this to derive implementable algorithms for solving optimization problems of the form \eqref{eq:convex:primal}.

\chapter{Monotone operators and proximal points}

Any minimizer $\bar x\in X$ of the convex functional $F:X\to \Rbar$ satisfies by \cref{thm:convex:fermat} the Fermat principle $0\in\partial F(\bar x)$.
To obtain from this useful information about (and, later, implementable algorithms for the computation of) $\bar x$, we thus have to study the mapping $x\mapsto \partial F(x)$.
To avoid mechnical difficulties -- and since we will use the following results mainly for numerical algorithms, i.e., for $X=\R^N$ -- we restrict the discussion in this and the next chapter to Hilbert spaces. This allows identifying $X^*$ with $X$; in particular, we will from now on identify the set $\partial F(x)\subset X^*$ of subderivatives with the corresponding set in $X$ of Riesz representations (\emph{subgradients}).

\section{Monotone operators}

For two normed vector spaces $X$ and $Y$ we consider a \emph{set-valued mapping} $A:X\to\calP(Y)$, also denoted by $A:X \setto Y$, and define
\begin{itemize}
    \item its \emph{domain of definition} $\dom A = \setof{x\in X}{Ax\neq \emptyset}$;
    \item its \emph{range} $\rg A = \bigcup_{x\in X} Ax$;
    \item its \emph{graph} $\graph A = \setof{(x,y)\in X\times Y}{y\in Ax}$;
    \item its \emph{inverse} $A^{-1}:Y\setto X$ via $A^{-1}(y) = \setof{x\in X}{y\in Ax}$ for all $y\in Y$.
\end{itemize}
(Note that $A^{-1}(y) = \emptyset$ is allowed by the definition; hence for set-valued mappings, the inverse always exists.) We say that set-valued mapping $A$ is \emph{surjective} if $\rg A=Y$.
For $A,B:X\setto Y$, $C:Y\setto Z$, and $\lambda\in\R$ we further define
\begin{itemize}
    \item $\lambda A:X\setto Y$ via $(\lambda A)(x) = \setof{\lambda y}{y\in Ax}$;
    \item $A+B:X\setto Y$ via $(A+B)(x) = \setof{y+z}{y\in Ax, z\in Bx}$;
    \item $C\circ A:X\setto Z$ via $(C\circ A)(x) = \setof{z}{\text{there is }y\in Ax \text{ with } z\in Cy}$.
\end{itemize}

Let from now on $X$ be a Hilbert space. A set-valued mapping $A:X\setto X$ is called \emph{monotone} if
\begin{equation}\label{eq:monoton:def}
    \inner{x^*_1-x^*_2,x_1-x_2}_X \geq 0 \quad\text{for all }(x_1,x_1^*),(x_2,x_2^*) \in \graph A.
\end{equation}

\begin{example}
    \begin{enumerate}[(i)]
        \item It follows directly from the definition that the identity mapping $\Id:X\setto X$, $x\mapsto\{x\}$, is monotone.
        \item Similarly, if $A,B:X\setto X$ are monotone and $\lambda\geq 0$, then $\lambda A$ and $A+B$ are monotone as well.
        \item If $F:X\to\Rbar$ is proper, then $\partial F:X\setto X$, $x\mapsto \partial F(x)$, is monotone: For any $x_1,x_2\in X$ with $x^*_1 \in \partial F(x_1)$ and $x^*_2\in\partial F(x_2)$, we have by definition that
            \begin{align*}
                &\inner{x_1^*,\tilde x - x_1}_X \leq F(\tilde x)-F(x_1)\qquad\text{for all } \tilde x\in X,\\
                &\inner{x_2^*,\tilde x - x_2}_X \leq F(\tilde x)-F(x_2)\qquad\text{for all } \tilde x\in X.
            \end{align*}
            Adding the first inequality for $\tilde x=x_2$ and the second for $\tilde x=x_1$ and rearranging the result yields \eqref{eq:monoton:def}.
    \end{enumerate}
\end{example}

In fact, we will need the following, stronger, property, which guarantees that $A$ is closed: A monotone operator $A: X\setto X$ is called \emph{maximally monotone} if for any $x\in X$ and $x^*\in X$ the condition
\begin{equation}\label{eq:monoton:max_char}
    \inner{x^*-\tilde x^*,x-\tilde x}_X \geq 0 \qquad \text{for all }(\tilde x , \tilde x^*) \in \graph A
\end{equation}
implies that $x^*\in A x$. (In other words, \eqref{eq:monoton:max_char} holds if \emph{and only if} $(x,x^*)\in\graph A$.)
For fixed $x\in X$ and $x^*\in X$, the condition claims that if $A$ is monotone, so is the extension
\begin{equation*}
    \tilde A:X\setto X,\qquad
    \tilde x \mapsto
    \begin{cases}
        A x \cup \{x^*\}&  \text{if } \tilde x =  x,\\
        A \tilde x & \text{if }\tilde x\neq  x.
    \end{cases}
\end{equation*}
For $A$ to be maximally monotone means that this is not a true extension, i.e., $\tilde A=A$.
\begin{example}
    The operator
    \begin{equation*}
        A:\R\setto\R,\qquad t \mapsto
        \begin{cases} \{1\} & \text{if }t> 0,\\
            \{0\} & \text{if } t= 0,\\
        \{-1\} &\text{if } t< 0,\end{cases}
    \end{equation*}
    is monotone but not maximally monotone, since $A$ is a proper subset of the monotone operator defined by $\tilde At = \sign(t)=\partial (|\cdot|)(t)$.
\end{example}

Several useful properties follow directly from the definition.
\begin{lemma}
    If $A:X\setto X$ is maximally monotone, then so is $\lambda A$ for all $\lambda >0$.
\end{lemma}
\begin{proof}
    Let $x,x^*\in X$ and assume that
    \begin{equation*}
        0 \leq \inner{x^*-\tilde x^*,x-\tilde x}_X = \lambda\inner{\lambda^{-1}x^*-\lambda^{-1}\tilde x^*,x-\tilde x}_X \quad \text{for all }(\tilde x,  \tilde x^*) \in \graph \lambda A.
    \end{equation*}
    Since $\tilde x^*\in \lambda Ax$ if and only if $\lambda^{-1}\tilde x^* \in Ax$ and $A$ is maximally monotone, this implies that $\lambda^{-1}\bar x^* \in A\bar x$, i.e., $\bar x^* \in (\lambda A)\bar x$. Hence, $\lambda A$ is maximally monotone.
\end{proof}
\begin{lemma}\label{cor:monoton:closed}
    Let $A:X\setto X$ be maximally monotone. Then $A$ is weakly--strongly closed, i.e., $x_n\weakto x$ and $Ax_n \ni x^*_n \to x^*$ imply that $x^*\in Ax$.
\end{lemma}
\begin{proof}
    For arbitrary $\tilde x\in X$ and $\tilde x^*\in A\tilde x$, the monotonicity of $A$ implies that
    \begin{equation*}
        0\leq \inner{x^*_n - \tilde x^*,x_n - \tilde x}_X \to \inner{x^* - \tilde x^*,x - \tilde x}_X
    \end{equation*}
    since the duality pairing and hence the inner product of weakly and strongly converging sequences is convergent. Since $A$ is maximally monotone, we obtain that $x^*\in A x$.
\end{proof}

Of central importance to the theory of monotone operators is \emph{Minty's theorem}, which states that a monotone operator $A$ is maximally monotone if and only if $\Id+A$ is surjective.
As a preparation, we first prove an important partial result.
\begin{lemma}\label{lem:monoton:subdiff_surj}
    Let $F:X\to\Rbar$ be proper, convex and lower semicontinuous. Then $\Id+\partial F$ is surjective.
\end{lemma}
\begin{proof}
    We consider for given $z \in X$ the functional
    \begin{equation*}
        J:X\to\Rbar, \qquad x\mapsto \frac12\norm{x -z}_X^2 + F(x),
    \end{equation*}
    which is proper, (strictly) convex and lower semicontinuous by the assumptions on $F$. Furthermore, $J$ is coercive by \cref{lem:convex:supercoercive}.
    \Cref{thm:convex:existence} thus yields a (unique) $\bar x\in X$ with $J(\bar x) = \min_{x\in X} J(x)$, which by \cref{thm:convex:fermat,thm:convex:sum,thm:convex:gateaux} satisfies that
    \begin{equation*}
        0\in \partial J(\bar x)  = \{\bar x - z\} + \partial F(\bar x),
    \end{equation*}
    i.e., $z\in \{\bar x\} + \partial F(\bar x)=(\Id+\partial F)(\bar x)$. Hence $\rg(\Id + \partial F)=X$ as claimed.
\end{proof}

We now turn to the general case.
\begin{theorem}[Minty]\label{thm:monoton:max_surj}
    Let $A:X\setto X$ be monotone with $\graph A \neq 0$. Then $A$ is maximally monotone if and only if\/ $\Id+A$ is surjective.
\end{theorem}
\begin{proof}
    First, assume that $\Id+A$ is surjective, and consider $x\in X$ and $x^*\in X$ with
    \begin{equation}\label{eq:monoton:max_surj1}
        \inner{x^*-\tilde x^*,x-\tilde x}_X \geq 0 \qquad \text{for all }(\tilde x,  \tilde x^*) \in \graph A.
    \end{equation}
    The assumption now implies that for $x  + x^*\in X$, there exist a $z \in X$ and a $z^* \in Az$ with
    \begin{equation}\label{eq:monoton:max_surj2}
        x+ x^* = z + z^* \in (\Id + A)z.
    \end{equation}
    Inserting $(\tilde x,\tilde x^*)=(z,z^*)$ into \eqref{eq:monoton:max_surj1}
    then yields that
    \begin{equation*}
        0\leq \inner{x^* - z^*,x-z}_X = \inner{z-x,x-z}_X = -\norm{x-z}_X^2\leq 0,
    \end{equation*}
    i.e., $x=z$. From \eqref{eq:monoton:max_surj2} we further obtain $x^* = z^*\in Az = Ax$, and hence $A$ is maximally monotone.

    The proof of the converse implication is significantly more involved.
    The special case $A=\partial F$ for a convex functional $F$ was already shown in \cref{lem:monoton:subdiff_surj}; for the general case, we proceed similarly by constructing a functional $F_A$ that plays the same role for $A$ as $F$ does for $\partial F$.
    Specifically, we define for a maximally monotone operator $A:X\setto X$ with $\graph A \neq \emptyset$ the \emph{Fitzpatrick functional}
    \begin{equation}\label{eq:monoton:fitzpatrick1}
        F_A:X\times X\to [-\infty,\infty],\qquad (x,y)\mapsto \sup_{(z,w)\in \graph A} \left(\inner{z,y}_X + \inner{x,w}_X - \inner{z,w}_X\right),
    \end{equation}
    which can be written equivalently as
    \begin{equation}\label{eq:monoton:fitzpatrick2}
        F_A(x,y) = \inner{x,y}_X - \inf_{(z,w)\in\graph A} \inner{x-z,y-w}_X.
    \end{equation}
    Each characterization implies useful properties.
    \begin{enumerate}[(i)]
        \item By maximal monotonicity of $A$, we have by definition that $(x-z,y-w)_X\geq 0$ for all $(z,w)\in\graph A$ if and only if $(x,y)\in\graph A$; in particular, $(x-z,y-w)_X< 0$ for all $(x,y)\notin\graph A$. Hence, \eqref{eq:monoton:fitzpatrick2} implies that $F_A(x,y)\geq \inner{x,y}_X$ with equality for $(x,y)\in\graph A$ (since in this case the infimum is attained in $(z,w)=(x,y)$). Since $\graph A\neq \emptyset$, this shows that $F_A$ is proper.

        \item On the other hand, the definition \eqref{eq:monoton:fitzpatrick1} yields that
            \begin{equation*}
                F_A = (G_A)^* \qquad\text{for}\qquad G_A(w,z) = \inner{w,z}_X + \delta_{\graph A^{-1}}(w,z)
            \end{equation*}
            (since $(z,w)\in\graph A$ if and only if $(w,z)\in\graph A^{-1}$).
            Since we have required that $\graph A\neq \emptyset$, the Fitzpatrick functional $F_A$ is the Fenchel conjugate of a proper functional and therefore convex and lower semicontinuous.
    \end{enumerate}

    As a first step, we now show that $0\in\rg (\Id+A)$. We set $Z:=X\times X$ as well as $\xi:=(x,y)\in Z$ and consider the functional
    \begin{equation*}
        J_A:Z\to\Rbar,\qquad \xi \mapsto F_A(\xi) + \frac12\norm{\xi}_Z^2.
    \end{equation*}
    We first note that property (i) implies for all $\xi\in Z$ that
    \begin{equation}\label{eq:monoton:fitzpatrick3}
        \begin{aligned}[t]
            J_A(\xi) = F_A(\xi) + \frac12\norm{\xi}_Z^2 &= F_A(x,y) + \frac12\norm{x}_X^2 +\frac12 \norm{y}_X^2 \\
            &\geq \inner{x,y}_X +\frac12 \norm{x}_X^2 +\frac12 \norm{y}_X^2 = \frac12\norm{x+y}_X^2\\
            &\geq 0.
        \end{aligned}
    \end{equation}
    Furthermore, $J_A$ is proper, (strictly) convex, lower semicontinuous, and (by \cref{lem:convex:supercoercive}) coercive.
    \Cref{thm:convex:existence} thus yields a (unique) $\bar \xi:=(\bar x,\bar y)\in Z$ with $J_A(\bar \xi) = \min_{\xi\in Z} J_A(\xi)$, which by \cref{thm:convex:fermat,thm:convex:sum,thm:convex:gateaux} satisfies that
    \begin{equation*}
        0\in \partial J_A(\bar \xi)  = \{\bar \xi\} + \partial F_A(\bar \xi),
    \end{equation*}
    i.e., $-\bar \xi\in \partial F_A(\bar \xi)$.
    By definition of the subdifferential, we thus have for all $\xi\in Z$ that
    \begin{equation*}
        \begin{aligned}
            F_A(\xi) \geq F_A(\bar \xi) + \inner{-\bar \xi,\xi-\bar \xi}_Z
            &= J_A(\bar \xi) + \frac12\norm{{-\bar \xi}}_Z^2 + \inner{-\bar \xi,\xi}_Z\\
            &\geq \frac12\norm{{-\bar \xi}}_Z^2 + \inner{-\bar \xi,\xi}_Z,
        \end{aligned}
    \end{equation*}
    where the last step follows from \eqref{eq:monoton:fitzpatrick3}. For the sake of presentation, we will replace $\bar \xi \mapsto -\bar \xi$ from now on; property (i) then implies for all $(x,y)\in\graph A$ that
    \begin{equation}\label{eq:monoton:fitzpatrick4}
        \begin{aligned}[t]
            \inner{x,y}_X =F_A(x,y)&\geq \frac12\norm{\bar x}_X^2 + \inner{\bar x,x}_X + \frac12\norm{\bar y}_X^2 + \inner{\bar y,y}_X\\
            &\geq -\inner{\bar x,\bar y}_X + \inner{\bar x,x}_X + \inner{\bar y,y}_X,
        \end{aligned}
    \end{equation}
    and hence $\inner{y-\bar x,x-\bar y}_X\geq 0$. The maximal monotonicity of  $A$ thus yields that $\bar x \in A \bar y$, i.e., $(\bar y,\bar x)\in\graph A$. Inserting this into the first inequality of \eqref{eq:monoton:fitzpatrick4} then implies that
    \begin{equation*}
        \inner{\bar y,\bar x}_X \geq \frac12\norm{\bar x}_X^2 + \inner{\bar x,\bar y}_X + \frac12\norm{\bar y}_X^2 + \inner{\bar y,\bar x}_X
        = \frac12\norm{\bar x+\bar y}_X^2 + \inner{\bar y,\bar x}_X \geq \inner{\bar y,\bar x}_X
    \end{equation*}
    and hence $\norm{\bar x+\bar y}_X=0$, i.e., $0=\bar y + \bar x\in (\Id + A)(\bar y)$.

    Finally, let $z\in X$ be arbitrary and set $B:X\setto X$, $x\mapsto \{-z\}+Ax$.
    Using the definition, it is straightforward to verify that $B$ is maximally monotone with $\graph B \neq \emptyset$ as well. As we have just shown, there now exists a $\bar y\in X$ with $0\in (\Id + B)(\bar y) = \{\bar y\} + \{-z\} + A\bar y$, i.e., $z\in (\Id + A)(\bar y)$. Hence $\Id+A$ is surjective.
\end{proof}
Together with \cref{lem:monoton:subdiff_surj} (which in particular implies $\graph \partial F\neq \emptyset$ for proper, convex, and lower semicontinuous $F$), this yields the maximal monotonicity of convex subdifferentials.
\begin{cor}\label{cor:monoton:subdiff}
    Let $F:X\to\Rbar$ be proper, convex, and lower semicontinuous. Then $\partial F:X\setto X$ is maximally monotone.
\end{cor}

\section{Resolvents and proximal points}

We know from \cref{lem:monoton:subdiff_surj} that $\Id+\partial F$ is surjective for any proper, convex, and lower semicontinuous functional $F$; the proof even shows that each preimage is unique.
Hence $(\Id+\partial F)^{-1}$ is single-valued even if $\partial F$ is not.
We can thus hope to use this object instead of a subdifferential -- or, more generally, a maximally monotone operator -- for algorithms.

We thus define for a maximally monotone operator $A:X\setto X$ with $\graph A \neq \emptyset$ the \emph{resolvent}
\begin{equation*}
    \calR_A : X\setto X,\qquad \calR_A(x) = (\Id + A)^{-1}x,
\end{equation*}
as well as for a proper, convex, and lower semicontinuous functional $F:X\to\Rbar$ the \emph{proximal point mapping}
\begin{equation}
    \label{eq:proximal:proximal}
    \prox_F:X\to X,\qquad \prox_F(x) = \argmin_{z\in X} \frac12\norm{z-x}_X^2 + F(z).
\end{equation}
Since $w\in\calR_{\partial F}(x)$ are the necessary and sufficient conditions for the \emph{proximal point} $w$ to be a minimizer of the strictly convex functional in \eqref{eq:proximal:proximal}, we have that
\begin{equation}
    \label{eq:proximal:resolvent}
    \prox_F = (\Id +\partial F)^{-1} = \calR_{\partial F}.
\end{equation}
It remains to show that the resolvent of arbitrary maximally monotone operators is single-valued on $X$ as well and we can thus write $\calR_A:X\to X$.
\begin{lemma}\label{lem:proximal:resolvent}
    Let $A:X \setto X$ be maximally monotone with $\graph A\neq \emptyset$. Then $\calR_A:X\to X$.
\end{lemma}
\begin{proof}
    Since $A$ is maximally monotone with $\graph A \neq \emptyset$, $\Id+A$ is surjective by \cref{thm:monoton:max_surj}, which implies that $\calR_A(x)\neq \emptyset$ for all $x\in X$, i.e., $\dom \calR_A=X$. Let now $x,z\in X$ with $x^* \in\calR_A(x)$ and $z^*\in \calR_A(z)$, i.e., $x\in \{x^*\} + Ax^*$ and $z\in \{z^*\}+Az^*$. For $x-x^*\in Ax^*$ and $z-z^*\in Az^*$, the monotonicity of $A$ implies that
    \begin{equation}\label{eq:proximal:nonexpansive}
        \norm{x^*-z^*}_X^2 \leq \inner{x-z,x^*-z^*}_X.
    \end{equation}
    Hence $x=z$ implies $x^*=z^*$, i.e., $\calR_A$ is single-valued.
\end{proof}
The inequality \eqref{eq:proximal:nonexpansive} together with the Cauchy--Schwarz inequality shows that resolvents are Lipschitz continuous with constant $L=1$; such mappings are called \emph{nonexpansive}. Since \eqref{eq:proximal:nonexpansive} is in fact a stronger property, a mapping $T:X\to X$ is called \emph{firmly nonexpansive} if it satisfies this inequality, i.e., if
\begin{equation*}
    \norm{Tx-Tz}_X^2 \leq \inner{Tx-Tz,x-z}_X \qquad\text{for all }x,z\in X.
\end{equation*}
Firm nonexpansivity implies another useful inequality.
\begin{lemma}\label{lem:proximal:nonexpansive}
    Let $A:X\setto X$ be maximally monotone with $\graph A \neq \emptyset$. Then $\calR_A:X\to X$ is firmly nonexpansive and satisfies that
    \begin{equation*}
        \norm{\calR_A x-\calR_A z}_X^2 +  \norm{(\Id-\calR_A)x-(\Id-\calR_A)z}_X^2 \leq \norm{x-z}_X^2\quad\text{for all }x,z\in X.
    \end{equation*}
\end{lemma}
\begin{proof}
    Firm nonexpansivity of $\calR_A$ was already shown in \eqref{eq:proximal:nonexpansive}, which further implies that
    \begin{align*}
        \norm{(\Id-\calR_A)x-(\Id-\calR_A)z}_X^2 &= \norm{x-z}_X^2 - 2\inner{x-z,\calR_Ax-\calR_A z}_X + \norm{\calR_A x-\calR_Az}^2_X\\
        &\leq  \norm{x-z}_X^2 - \norm{\calR_A x-\calR_Az}^2_X.
        \qedhere
    \end{align*}
\end{proof}
\begin{cor}\label{lem:proximal:lipschitz}
    Let $F:X\to\Rbar$ be proper, convex, and lower semicontinuous. Then $\prox_F:X\to X$ is Lipschitz continuous with constant $L=1$.
\end{cor}

\bigskip

The following useful result allows characterizing minimizers of convex functionals as proximal points.
\begin{lemma}\label{lem:proximal:subdiff}
    Let $F:X\to \Rbar$ be proper, convex, and lower semicontinuous, and $x,x^*\in X$. Then for any $\gamma>0$,
    \begin{equation*}
        x^*\in \partial  F(x) \quad\equivalent\quad x = \prox_{\gamma F}(x + \gamma x^*).
    \end{equation*}
\end{lemma}
\begin{proof}
    Multiplying both sides of the subdifferential inclusion by $\gamma>0$ and adding $x$ yields that
    \begin{equation*}
        \begin{aligned}[b]
            x^* \in \partial F(x) &\Leftrightarrow x + \gamma x^* \in (\Id + \gamma\partial F)(x)\\
            &\Leftrightarrow x \in (\Id + \gamma\partial F)^{-1}(x+\gamma x^*)\\
            &\Leftrightarrow x = \prox_{\gamma F}(x+\gamma x^*),
        \end{aligned}
    \end{equation*}
    where in the last step we have used that $\gamma\partial F=\partial(\gamma F)$ by \cref{lem:convex:subdiff_calc}\,(i) and hence that $\prox_{\gamma F} = \calR_{\gamma\partial F}$.
\end{proof}
\begin{cor}\label{thm:proximal:fermat}
    Let $F:X\to \Rbar$ be proper, convex and lower semicontinuous, and $\gamma>0$ be arbitrary. Then $\bar x\in \dom F$ is a minimizer of $F$ if and only if
    \begin{equation*}
        \bar x = \prox_{\gamma F}(\bar x).
    \end{equation*}
\end{cor}
\begin{proof}
    Simply apply \cref{lem:proximal:subdiff} to the Fermat principle $0\in\partial F(\bar x)$.
\end{proof}
This simple result should not be underestimated: It allows replacing (explicit) set inclusions by (implicit) Lipschitz continuous equations in optimality conditions, thus opening the door to fixed point iterations or Newton-type methods.

We can also derive a generalization of the orthogonal decomposition of vector spaces.
\begin{theorem}[Moreau decomposition]\label{thm:proximal:moreau}
    Let $F:X\to \Rbar$ be proper, convex, and lower semicontinuous. Then we have for all $x\in X$ that
    \begin{equation*}
        x = \prox_F(x) + \prox_{F^*}(x).
    \end{equation*}
\end{theorem}
\begin{proof}
    Setting $w = \prox_{F}(x)$, \cref{lem:proximal:subdiff,lem:convex:fenchel-young} imply that
    \begin{equation*}
        \begin{split}
            \begin{aligned}[b]
                w = \prox_F(x) = \prox_F(w+(x-w)) &\Leftrightarrow x-w \in \partial F(w)\\
                &\Leftrightarrow w \in \partial F^*(x-w)\\
                &\Leftrightarrow x-w = \prox_{F^*}((x-w)+w)
                =\prox_{F^*}(x).
            \end{aligned}
            \qedhere
        \end{split}
    \end{equation*}
\end{proof}

The following calculus rules will prove useful.
\begin{lemma}\label{lem:proximal:calculus}
    Let $F:X\to\Rbar$ be proper, convex, and lower semicontinuous. Then,
    \begin{enumerate}[(i)]
        \item for $\lambda \neq 0$ and $z\in X$ we have with $H(x) := F(\lambda x + z)$ that
            \begin{equation*}
                \prox_H(x) = \lambda^{-1} (\prox_{\lambda^2 F}(\lambda x + z) -z);
            \end{equation*}
        \item for $\gamma>0$ we have that\vspace*{-2ex}
            \begin{equation*}
                \prox_{\gamma F^*}(x) = x - \gamma\, \prox_{\gamma^{-1} F}(\gamma^{-1}x);
            \end{equation*}
        \item for proper, convex, lower semicontinuous $G:Y\to\Rbar$ and $\gamma>0$ we have with $H(x,y) := F(x)+G(y)$ that
            \begin{equation*}
                \prox_{\gamma H}(x,y) = \begin{pmatrix}\prox_{\gamma F}(x)\\\prox_{\gamma G}(y)
                \end{pmatrix}.
            \end{equation*}
    \end{enumerate}
\end{lemma}
\begin{proof}
    \emph{(i):} By definition,
    \begin{equation*}
        \prox_H(x) = \argmin_{w\in X} \frac12\norm{w-x}^2_X + F(\lambda w + z) =: \bar w.
    \end{equation*}
    Now note that since $X$ is a vector space,
    \begin{equation*}
        \min_{w\in X} \frac12\norm{w-x}^2_X + F(\lambda w + z) = \min_{v\in X} \frac12\norm{\lambda^{-1}(v-z) -x}^2_X + F(v),
    \end{equation*}
    and the respective minimizers $\bar w$ and $\bar v$ are related by $\bar v=\lambda \bar w+z$. The claim then follows from
    \begin{equation*}
        \begin{aligned}
            \bar v &= \argmin_{v\in X} \frac12\norm{\lambda^{-1}(v-z) -x}^2_X + F(v)\\
            &= \argmin_{v\in X} \frac1{2\lambda^2}\norm{v -(\lambda x+z)}^2_X + F(v)\\
            &= \argmin_{v\in X} \frac1{2}\norm{v -(\lambda x+z)}^2_X + \lambda^2F(v)\\
            & = \prox_{\lambda^2 F}(\lambda x+z).
        \end{aligned}
    \end{equation*}
    Hence, $\bar w := \lambda^{-1}(\bar v - z)$ is the desired minimizer.

    \emph{(ii):} \Cref{thm:proximal:moreau}, \cref{lem:convex:fenchel_calc}\,(i), and (i) for $\lambda=\gamma^{-1}$ and $z=0$ together imply that
    \begin{equation*}
        \begin{aligned}
            \prox_{\gamma F}(x) &= x- \prox_{(\gamma F)^*}(x)\\
            &= x - \prox_{\gamma F^*\circ (\gamma^{-1}\Id)}(x)\\
            &= x - \gamma \,\prox_{\gamma( \gamma^{-2} F^*)}(\gamma^{-1}x).
        \end{aligned}
    \end{equation*}
    Applying this to $F^*$ and using that $F^{**} = F$ now yields the claim.

    \emph{(iii):} By definition of the norm on the product space $X\times Y$, we have that
    \begin{equation*}
        \begin{aligned}
            \prox_{\gamma H}(x,y) &= \argmin_{(u,v)\in X\times Y}  \frac12 \norm{(u,v) - (x,y)}^2_{X\times Y} + \gamma H(u,v)\\
            &= \argmin_{u\in X, v\in Y}  \left(\frac12 \norm{u - x}^2_{X} + \gamma F(u) \right)
            +\left(\frac12 \norm{v - y}^2_{Y} + \gamma G(v)\right).
        \end{aligned}
    \end{equation*}
    Since there are no mixed terms in $u$ and $v$, the two terms in parentheses can be minimized separately. Hence, $\prox_{\gamma H}(x,y) = (\bar u, \bar v)$ for
    \begin{equation*}
        \begin{split}
            \begin{aligned}[b]
                \bar u =  \argmin_{u\in X}  \frac12 \norm{u - x}^2_{X} + \gamma F(u)  = \prox_{\gamma F(x)},\\
                \bar v =  \argmin_{v\in Y}  \frac12 \norm{v - y}^2_{Y} + \gamma G(v)  = \prox_{\gamma G(x)}.
            \end{aligned}
            \qedhere
        \end{split}
    \end{equation*}
\end{proof}

Computing proximal points is difficult in general since evaluating $\prox_F$ by its definition entails minimizing $F$. In some cases, however, it is possible to give an explicit formula for $\prox_F$.
\begin{example}\label{ex:proximal:reell}
    We first consider scalar functions $f:\R\to\Rbar$.
    \begin{enumerate}[(i)]
        \item $f(t) = \frac12|t|^2$. Since $f$ is differentiable, we can set the derivative of $\frac12(s-t)^2+\frac\gamma2s^2$ to zero and solve for $s$ to obtain $\prox_{\gamma f}(t) = (1+\gamma)^{-1}t$.

        \item $f(t) = |t|$. By \eqref{eq:convex:subdiff_abs} we have that $\partial f(t) = \sign(t)$; hence
            $s :=\prox_{\gamma f}(t)=(\Id + \gamma\sign)^{-1}(t)$
            if and only if $t\in \{s\}+\gamma\sign(s)$. Let $t$ be given and assume this holds for some $\bar s$. We now proceed by case distinction.
            \begin{enumerate}[{Case }1:]
                \item $\bar s>0$. This implies that $t = \bar s+\gamma$, i.e., $\bar s = t-\gamma$, and hence that $t>\gamma$.
                \item $\bar s<0$. This implies that $t = \bar s-\gamma$, i.e., $\bar s = t+\gamma$, and hence that $t<-\gamma$.
                \item $\bar s=0$. This implies that $t\in \gamma[-1,1]= [-\gamma,\gamma]$.
            \end{enumerate}
            Since this yields a complete and disjoint case distinction for $t$, we can conclude that
            \begin{equation*}
                \prox_{\gamma f}(t) = \begin{cases} t-\gamma & \text{if }t>\gamma,\\
                    0 &\text{if } t\in [-\gamma,\gamma],\\
                    t+\gamma &\text{if } t < -\gamma.
                \end{cases}
            \end{equation*}
            This mapping is also known as the \emph{soft-shrinkage} or \emph{soft-thresholding} operator.

        \item $f(t) = \delta_{[-1,1]}(t)$. By \cref{ex:convex:fenchel}\,(iii) we have that $f^*(t) = |t|$. Hence \cref{lem:proximal:calculus}\,(ii) yields that
            \begin{equation*}
                \begin{aligned}
                    \prox_{\gamma f}(t) &= t-\gamma\,\prox_{\gamma^{-1} f^*}(\gamma^{-1} t) \\
                    &=\begin{cases}
                        t - \gamma(\gamma^{-1} t - \gamma^{-1})  &\text{if } \gamma^{-1}t > \gamma^{-1},\\
                        t - 0  &\text{if } \gamma^{-1}t\in [-\gamma^{-1},\gamma^{-1}],\\
                        t - \gamma(\gamma^{-1} t + \gamma^{-1})  &\text{if } \gamma^{-1}t <- \gamma^{-1}
                    \end{cases}\\
                    &=\begin{cases}
                        \phantom{-}1&\text{if } t>1, \\
                        \phantom{-}t   &\text{if } t\in [-1,1],\\
                        -1&\text{if } t<-1.
                    \end{cases}
                \end{aligned}
            \end{equation*}
            For every $\gamma>0$, the proximal point of $t$ is thus its projection onto $[-1,1]$.
    \end{enumerate}
\end{example}
\begin{example}\label{ex:proximal:rn}
    We can generalize \cref{ex:proximal:reell} to $X=\R^N$ (endowed with the Euclidean inner product) by applying \cref{lem:proximal:calculus}\,(iii) $N$ times. We thus obtain componentwise
    \begin{enumerate}[(i)]
        \item for $F(x)=\frac12\norm{x}_2^2=\sum_{i=1}^N \tfrac12 x_i^2$ that
            \begin{equation*}
                [\prox_{\gamma F}(x)]_i = \left(\frac1{1+\gamma}\right) x_i,\quad 1\leq i\leq N;
            \end{equation*}
        \item for $F(x) = \norm{x}_1=\sum_{i=1}^N|x_i|$ that
            \begin{equation*}
                [\prox_{\gamma F}(x)]_i = (|x_i|-\gamma)^+\sign(x_i),\quad 1\leq i\leq N;
            \end{equation*}
        \item for $F(x) = \delta_{B_\infty}(x)=\sum_{i=1}^N \delta_{[-1,1]}(x_i)$ that
            \begin{equation*}
                [\prox_{\gamma F}(x)]_i = x_i-(x_i-1)^+-(x_i+1)^-
                = \frac{x_i}{\max\{1,|x_i|\}},\qquad 1\leq i\leq N.
            \end{equation*}
    \end{enumerate}
    Here we have used the convenient notation $(t)^+ := \max\{t,0\}$ and $(t)^-:=\min\{t,0\}$.
\end{example}
Many more examples can be found in \cite[\S\,6.5]{Boyd:2014}.

Since the subdifferential of convex integral functionals can be evaluated pointwise by \cref{lem:lebesgue:subdiff}, the same holds for the definition \eqref{eq:proximal:resolvent} of the proximal point mapping.
\begin{cor}\label{lem:lebesgue:proximal}
    Let $f:\R\to\Rbar$ be proper, convex, and lower semicontinuous, and $F:L^2(\Omega)\to\Rbar$ be defined as in \cref{lem:lebesgue:lsc}. Then we have for all $\gamma>0$ and $u\in L^2(\Omega)$ that
    \begin{equation*}
        [\prox_{\gamma F}(u)](x) = \prox_{\gamma f}(u(x))\qquad\text{for almost every }x\in \Omega.
    \end{equation*}
\end{cor}
\begin{example}\label{eq:proximal:hilbert}
    Let $X$ be a Hilbert space. Similarly to \cref{ex:proximal:reell} one can show
    \begin{enumerate}[(i)]
        \item for $F=\frac12\norm{\cdot}_X^2 = \frac12\inner{\cdot,\cdot}_X$, that
            \begin{equation*}
                \prox_{\gamma F}(x) = \left(\frac{1}{1+\gamma}\right)x;
            \end{equation*}
        \item for $F=\norm{\cdot}_X$, using a case distinction as in \cref{thm:subdifferential:norm}, that
            \begin{equation*}
                \prox_{\gamma F}(x) = \left(1-\frac\gamma{\norm{x}_X}\right)^+ x;
            \end{equation*}
        \item for $F=\delta_C$ with $C\subset X$ nonempty, convex, and closed, that by definition
            \begin{equation*}
                \prox_{\gamma F}(x) = \proj_C(x) := \argmin_{z\in C} \norm{z-x}_X
            \end{equation*}
            the \emph{metric projection} of $x$ onto $C$; the proximal point mapping thus generalizes the concept projection onto convex sets. Explicit or at least constructive formulas for the projection onto different classes of sets can be found in \cite[Chapter 4.1]{Cegielski:2012}.
    \end{enumerate}
\end{example}

\section{Moreau--Yosida regularization}

Before we turn to algorithms for the minimization of convex functionals, we will look at another way to reformulate optimality conditions using proximal point mappings. Although these are no longer equivalent reformulations, they will serve as a link to the Newton-type methods introduced in \cref{chap:semismooth}.

Let $A:X\setto X$ be a maximally monotone operator with $\graph A\neq \emptyset$ and $\gamma >0$.
Then we define the \emph{Yosida approximation} of $A$ as
\begin{equation*}
    A_\gamma := \frac1\gamma\left(\Id - \calR_{\gamma A}\right).
\end{equation*}
In particular, the Yosida approximation of the subdifferential of a proper, convex, and lower semicontinuous functional $F:X\to\Rbar$ is given by
\begin{equation*}
    (\partial F)_\gamma := \frac1\gamma\left(\Id - \prox_{\gamma F}\right),
\end{equation*}
which by \cref{lem:proximal:lipschitz} is always Lipschitz continuous with constant $L=\gamma^{-1}$.

An alternative point of view is the following. For a proper, convex, and lower semicontinuous functional $F:X\to\Rbar$ and $\gamma>0$, we define the \emph{Moreau envelope}\footnote{not to be confused with the \emph{convex} envelope $F^\Gamma$!} as
\begin{equation*}
    F_\gamma :X\to\R,\qquad x\mapsto \inf_{z\in X} \frac{1}{2\gamma}\norm{z-x}_X^2 + F(z).
\end{equation*}
Comparing this with the definition \eqref{eq:proximal:proximal} of the proximal point mapping of $F$, we see that
\begin{equation}\label{eq:proximal:moreau}
    F_\gamma(x) = \frac{1}{2\gamma}\norm{\prox_{\gamma F}(x)-x}_X^2 + F(\prox_{\gamma F}(x)).
\end{equation}
(Note that multiplying a functional by $\gamma>0$ does not change its minimizers.)
Hence $F_\gamma$ is indeed well-defined on $X$ and single-valued. Furthermore, we can deduce from \eqref{eq:proximal:moreau} that $F_\gamma$ is convex as well.
\begin{lemma}\label{lem:moreau:convex}
    Let $F:X\to\Rbar$ be proper, convex, and lower semicontinuous, and $\gamma>0$. Then $F_\gamma$ is convex.
\end{lemma}
\begin{proof}
    We first show that for any convex $G:X\to\Rbar$, the mapping
    \begin{equation*}
        H:X\times X \to \Rbar,\qquad (x,z)\mapsto F(z)+G(z-x)
    \end{equation*}
    is convex as well. Indeed, for any $(x_1,z_1),(x_2,z_2)\in X\times X$ and $\lambda\in[0,1]$, the convexity of $F$ and $G$ implies that
    \begin{equation*}
        \begin{aligned}
            H(\lambda(x_1,z_1)+(1-\lambda)(x_2,z_2)) &= F\left(\lambda z_1 + (1-\lambda )z_2\right)+G\left(\lambda(z_1-x_1) + (1-\lambda)(z_2-x_2)\right)\\
            &\leq \lambda\left(F(z_1)+G(z_1-x_1)\right)+(1-\lambda)\left(F(z_2)+G(z_2-x_2)\right)\\
            &= \lambda H(x_1,z_1) + (1-\lambda)H(x_2,z_2).
        \end{aligned}
    \end{equation*}
    Let now $x_1,x_2\in X$ and $\lambda\in[0,1]$. Since $F_\gamma(x) = \inf_{z\in X} H(x,z)$ for $G(y):=\frac1{2\gamma}\norm{y}_X^2$, there exist two minimizing sequences $\{z^1_n\}_{n\in\N},\{z^2_n\}_{n\in\N}\subset X$ with
    \begin{equation*}
        H(x_1,z^1_n)\to F_\gamma(x_1),\qquad H(x_2,z^2_n)\to F_\gamma(x_2).
    \end{equation*}
    From the properties of the infimum together with the convexity of $H$, we thus obtain for all $n\in\N$ that
    \begin{equation*}
        \begin{aligned}
            F_\gamma(\lambda x_1 + (1-\lambda)x_2) & \leq H(\lambda (x_1,z^1_n) + (1-\lambda) (x_2,z^2_n) ) \\
            &\leq \lambda H(x_1,z^1_n) + (1-\lambda)H(x_2,z^2_n),
        \end{aligned}
    \end{equation*}
    and passing to the limit $n\to\infty$ yields the desired convexity.
\end{proof}

The next theorem links the two concepts and hence justifies the term \emph{Moreau--Yosida regularization}.
\begin{theorem}
    Let $F:X\to\Rbar$ be proper, convex, and lower semicontinuous, and $\gamma>0$.
    Then $F_\gamma$ is Fréchet differentiable with
    \begin{equation*}
        \nabla (F_\gamma) = (\partial F)_\gamma.
    \end{equation*}
\end{theorem}
\begin{proof}
    Let $x,y\in X$ be arbitrary and set $x^* = \prox_{\gamma F}(x)$ and $y^* = \prox_{\gamma F}(y)$. We first show that
    \begin{equation}\label{eq:proximal:my_frechet1}
        \frac1\gamma \inner{y^*-x^*,x-x^*}_X  \leq F(y^*) - F(x^*).
    \end{equation}
    (Note that for proper $F$, the definition of proximal points as minimizers necessarily implies that $x^*,y^*\in\dom F$.)
    To this purpose, consider for $t\in (0,1)$ the point $x^*_t := ty^* + (1-t)x^*$.
    Using the minimizing property of the proximal point $x^*$ together with the convexity of $F$ and completing the square, we obtain that
    \begin{equation*}
        \begin{aligned}
            F(x^*) &\leq F(x^*_t) + \frac1{2\gamma}\norm{x^*_t - x}_X^2 - \frac1{2\gamma}\norm{x^* - x}_X^2\\
            &\leq t F(y^*) + (1-t)F(x^*) - \frac{t}{\gamma} \inner{x-x^*,y^*-x^*}_X +  \frac{t^2}{2\gamma}\norm{x^*-y^*}_X^2.
        \end{aligned}
    \end{equation*}
    Rearranging the terms, dividing by $t>0$ and passing to the limit $t\to 0$ then yields \eqref{eq:proximal:my_frechet1}.
    Combining this with \eqref{eq:proximal:moreau} implies that
    \begin{equation*}
        \begin{aligned}
            F_\gamma(y) - F_\gamma(x) &= F(y^*) - F(x^*) + \frac1{2\gamma}\left(\norm{y-y^*}_X^2 - \norm{x-x^*}_X^2\right)\\
            &\geq \frac1{2\gamma}\left(2\inner{y^*-x^*,x-x^*}_X + \norm{y-y^*}_X^2 - \norm{x-x^*}_X^2\right)\\
            &= \frac1{2\gamma}\left(2\inner{y-x,x-x^*}_X + \norm{y-y^*-x+x^*}_X^2\right)\\
            &\geq \frac1\gamma \inner{y-x,x-x^*}_X.
        \end{aligned}
    \end{equation*}
    By exchanging the roles of $x^*$ and $y^*$ in \eqref{eq:proximal:my_frechet1}, we obtain that
    \begin{equation*}
        F_\gamma(y) - F_\gamma(x) \leq \frac1\gamma \inner{y-x,y-y^*}_X.
    \end{equation*}
    Together, these two inequalities yield that
    \begin{equation*}
        \begin{aligned}
            0 &\leq F_\gamma(y) - F_\gamma(x) - \frac1\gamma \inner{y-x,x-x^*}_X\\
            &\leq \frac1\gamma \inner{y-x,(y-y^*)-(x-x^*)}_X\\
            &\leq \frac1\gamma\left(\norm{y-x}_X^2 - \norm{y^*-x^*}_X^2\right)\\
            &\leq \frac1\gamma \norm{y-x}_X^2,
        \end{aligned}
    \end{equation*}
    where the next-to-last inequality follows from the firm nonexpansivity of proximal point mappings (\cref{lem:proximal:nonexpansive}).

    If we now set $y=x+h$ for arbitrary $h\in X$, we obtain that
    \begin{equation*}
        0\leq \frac{F_\gamma(x+h) - F_\gamma(x) - \inner{\gamma^{-1}(x-x^*),h}_X}{\norm{h}_X} \leq \frac 1\gamma \norm{h}_X \to 0 \qquad\text{for } h\to 0,
    \end{equation*}
    i.e., $F_\gamma$ is Fréchet differentiable with gradient $\frac1\gamma(x-x^*)=(\partial F)_\gamma$.
\end{proof}
Since $F_\gamma$ is convex by \cref{lem:moreau:convex}, this result together with \cref{thm:convex:gateaux} yields the catchy relation $\partial(F_\gamma) = (\partial F)_\gamma$.

\begin{example}\label{ex:moreau}
    We consider again $X=\R^N$.
    \begin{enumerate}[(i)]
        \item For $F(x)=\|x\|_1$, we have from \cref{ex:proximal:rn}\,(ii) that the proximal point mapping is given by the component-wise soft-shrinkage operator. Inserting this into the definition yields that
            \begin{equation*}
                \left[(\partial \|\cdot\|_1)_\gamma(x)\right]_i =
                \begin{cases}
                    \frac1\gamma(x_i-(x_i-\gamma)) = 1 & \text{if }x_i>\gamma,\\
                    \frac1\gamma x_i & \text{if }x_i\in[-\gamma,\gamma],\\
                    \frac1\gamma(x_i-(x_i+\gamma)) = -1& \text{if }x_i<-\gamma.
                \end{cases}
            \end{equation*}
            Comparing this to the corresponding subdifferential \eqref{eq:convex:subdiff_abs}, we see that the set-valued case in the point $x_i=0$ has been replaced by a linear function on a small interval.

            Similarly, inserting the definition of the proximal point into \eqref{eq:proximal:moreau} shows that
            \begin{equation*}
                F_\gamma(x) = \sum_{i=1}^N f_\gamma(x_i)
                \ \text{ for }\
                f_\gamma(t) :=
                \begin{cases}
                    \frac1{2\gamma}|t-(t-\gamma)|^2 + |t-\gamma| = t - \frac\gamma2 & \text{if }t>\gamma,\\
                    \frac1{2\gamma}|t|^2 & \text{if }t \in [-\gamma,\gamma],\\
                    \frac1{2\gamma}|t-(t+\gamma)|^2 + |t+\gamma| = -t + \frac\gamma2 &\text{if } t<-\gamma.
                \end{cases}
            \end{equation*}
            For small values, the absolute value is thus replaced by a quadratic function (which removes the nondifferentiability at $0$). This modification is well-known under the name \emph{Huber norm}.

        \item For $F(x)=\delta_{B_\infty}(x)$, we have from \cref{ex:proximal:rn}\,(iii) that the proximal mapping is given by the component-wise projection onto $[-1,1]$ and hence that
            \begin{equation*}
                \left[(\partial \delta_{B_\infty})_\gamma(x)\right]_i = \frac1\gamma\Big(x_i-\big(x_i - (x_i-1)^+ - (x_i+1)^-\big)\Big) = \frac1\gamma (x_i-1)^+ + \frac1\gamma (x_i+1)^-.
            \end{equation*}
            Similarly, inserting this and using that $\prox_{\gamma F}(x)\in B_\infty$ and $\inner{(x+1)^+,(x-1)^-}_X=0$ yields that
            \begin{equation*}
                (\delta_{B_\infty})_\gamma(x) = \frac{1}{2\gamma}\norm{(x-1)^+}_2^2 +  \frac{1}{2\gamma}\norm{(x+1)^-}_2^2,
            \end{equation*}
            which corresponds to the classical penalty functional for the inequality constraints $x-1\leq 0$ and $x+1\geq 0$ in nonlinear optimization.

    \end{enumerate}
\end{example}

A further connection exists between the Moreau envelope  and the Fenchel conjugate.
\begin{theorem}
    Let $F:X\to\Rbar$ be proper, convex, and lower semicontinuous. Then we have for all $\gamma>0$ that
    \begin{equation*}
        (F_\gamma)^* = F^* + \frac{\gamma}{2} \norm{\cdot}_X^2.
    \end{equation*}
\end{theorem}
\begin{proof}
    We obtain directly from the definition of the Fenchel conjugate in Hilbert spaces and of the Moreau envelope that
    \begin{equation*}
        \begin{aligned}
            (F_\gamma)^*(x^*) &= \sup_{x\in X} \left(\inner{x^*,x}_X - \inf_{z\in X} \left(\tfrac{1}{2\gamma}\norm{x-z}_X^2 + F(z)\right)\right)\\
            &= \sup_{x\in X} \left(\inner{x^*,x}_X + \sup_{z\in X} \left(-\tfrac{1}{2\gamma}\norm{x-z}_X^2 - F(z)\right)\right)\\
            &= \sup_{z\in X} \left(\inner{x^*,z}_X - F(z)+\sup_{x\in X} \left(\inner{x^*,x-z}_X - \tfrac{1}{2\gamma}\norm{x-z}_X^2\right)\right)\\
            &= F^*(x^*)  + \left(\tfrac{1}{2\gamma}\norm{\cdot}_X^2\right)^*(x^*),
        \end{aligned}
    \end{equation*}
    since for any given $z\in X$, the inner supremum is always taken over the full space $X$.
    The claim now follows from \cref{ex:convex:fenchel}\,(i) and \cref{lem:convex:fenchel_calc}\,(i).
\end{proof}

\bigskip

We briefly sketch the relevance for nonsmooth optimization. For a convex functional $F:X\to\Rbar$, every minimizer $\bar x\in X$ satisfies the Fermat principle $0\in \partial F(\bar x)$, which we can write equivalently as $\bar x\in \partial F^*(0)$. If we now replace $\partial F^*$ with its Yosida approximation $(\partial F^*)_\gamma$, we obtain the regularized optimality condition
\begin{equation*}
    x_\gamma = (\partial F^*)_\gamma (0) = -\frac1\gamma \prox_{\gamma F^*}(0).
\end{equation*}
This is now an \emph{explicit} and even Lipschitz continuous relation (which, among other things, can be used to derive stability properties for $x_\gamma$ under perturbations).
Although $x_\gamma$ is no longer a minimizer of $F$, the convexity of $F_\gamma$ implies that $x_\gamma \in (\partial F^*)_\gamma (0) = \partial(F^*_\gamma)(0)$ is equivalent to
\begin{equation*}
    0\in \partial (F^*_\gamma)^*(x_\gamma) = \partial\left(F^{**} + \tfrac{\gamma}2\norm{\cdot}_X^2\right)(x_\gamma) = \partial\left(F + \tfrac{\gamma}2\norm{\cdot}_X^2\right)(x_\gamma),
\end{equation*}
i.e., $x_\gamma$ is the (unique due to the strict convexity of the squared norm) minimizer of the functional $F+ \tfrac{\gamma}2\norm{\cdot}_X^2$.
Hence, the regularization of $\partial F^*$ has not made the original problem smooth but merely (more) strongly convex.
The equivalence can also be used to show (similarly to the proof of \cref{thm:variation:existence}) that $x_\gamma \weakto \bar x$ for $\gamma\to 0$.
In practice, this straightforward approach fails due to the difficulty of computing $F^*$ and $\prox_{F^*}$ and is therefore usually combined with one of the splitting techniques introduced in the next chapter.

\chapter{Proximal point and splitting methods}\label{chap:proximal}

We now turn to algorithms for computing minimizers of functionals $J:X\to\Rbar$ of the form
\begin{equation*}
    J(x) := F(x) + G(x)
\end{equation*}
for $F,G:X\to\Rbar$ convex but not necessarily differentiable.
One of the main difficulties compared to the differentiable setting is that the naive equivalent to steepest descent, the iteration
\begin{equation*}
    x^{k+1} \in x^k - \tau_k \partial J(x^k),
\end{equation*}
does not work since even in finite dimensions, arbitrary subgradients need not be descent directions -- this can only be guaranteed for the subgradient of minimal norm; see, e.g., \cite[Example 7.1, Lemma 2.77]{Ruszczynski:2006a}. Furthermore, the minimal norm subgradient of $J$ cannot be computed easily from those of $F$ and $G$.
We thus follow a different approach and look for a root of the set-valued mapping $x\mapsto \partial J(x)\subset X^*\cong X$.

\section{Proximal point method}

We have seen in \cref{thm:proximal:fermat} that a root $\bar x$ of $\partial J:X\setto X$ can be characterized as a fixed point of $\prox_{\gamma J}$ for any $\gamma>0$. This suggests a fixed-point iteration: Choose $x^0\in X$ and set for an appropriate sequence $\{\gamma_k\}_{k\in \N}$
\begin{equation}\label{eq:proximal:ppa}
    x^{k+1} = \prox_{\gamma_k J}(x^k).
\end{equation}
To show convergence of this iteration, we have to show as usual that the fixed-point mapping is contracting in a suitable sense. As we will see, firm nonexpansivity will be sufficient, which by \cref{lem:proximal:lipschitz} is always the case for resolvents of maximally monotone operators (and hence in particular for proximal mappings of convex functionals).
For later use, we treat the general version of \eqref{eq:proximal:ppa} for arbitrary maximally monotone operators.
\begin{theorem}\label{thm:proximal:fixedpoint}
    Let $A:X\setto X$ be maximally monotone with root $x^*\in X$, and let $\{\gamma_k\}_{k\in\N}\subset(0,\infty)$ with $\sum_{k=0}^\infty \gamma_k^2 = \infty$. If $\{x^k\}_{k\in\N}\subset X$ is given by the iteration
    \begin{equation*}
        x^{k+1} = \calR_{\gamma_k A} x^k,
    \end{equation*}
    then $x^k\weakto \bar x$ with $0\in A\bar x$.
\end{theorem}
\begin{proof}
    The iteration $x^{k+1} = \calR_{\gamma_k A}x^k = (\Id + \gamma_k A)^{-1}x^k$ implies that
    \begin{equation*}
        w^k := \gamma_k^{-1}(x^k - x^{k+1}) \in Ax^{k+1}
    \end{equation*}
    and hence that $x^{k+1}-x^{k+2} = \gamma_{k+1}w^{k+1}$. (The vector $w^k$ will play the role of a residual in the generalized equation $0\in A x$.)
    By monotonicity of $A$, we have for $\gamma_{k+1}>0$ that
    \begin{equation*}
        \begin{aligned}
            0 & \leq \gamma_{k+1}^{-1}\inner{w^{k}-w^{k+1},x^{k+1}-x^{k+2}}_X\\
            & = \inner{w^k-w^{k+1},w^{k+1}}_X\\
            & = \inner{w^k,w^{k+1}}_X-\norm{w^{k+1}}_X^2\\
            & \leq \norm{w^{k+1}}_X \left(\norm{w^k}_X - \norm{w^{k+1}}_X\right).
        \end{aligned}
    \end{equation*}
    Hence, the nonnegative sequence $\{\norm{w^k}_X\}_{k\in\N}\subset \R$ is decreasing and hence convergent (as long as $w^{k+1}\neq 0$, but otherwise from $w^{k+1}\in Ax^{k+2}$ we immediately obtain that $x^{k+2}$ is the desired root.)

    Let now $x^*\in X$ be a root of $A$, i.e., $0\in Ax^*$, which exists by assumption. As in the proof of \cref{thm:proximal:fermat}, this inclusion is equivalent to $x^* = \calR_{\gamma A}x^*$ for all $\gamma>0$. From \cref{lem:proximal:nonexpansive} together with $(\Id -\calR_{\gamma_k A})x^k = x^k-x^{k+1} = \gamma_k w^k$, we now obtain that
    \begin{equation}
        \label{eq:proximal:fixed_fejer}
        \begin{aligned}[t]
            \norm{x^{k+1}-x^*}_X^2 &= \norm{\calR_{\gamma_k A}x^{k}-\calR_{\gamma_kA}x^*}_X^2 \\
            &\leq \norm{x^k - x^*}_X^2 - \norm{(\Id-\calR_{\gamma_k A})x^k-(\Id-\calR_{\gamma_k A})x^*}_X^2\\
            &= \norm{x^k - x^*}_X^2 - \gamma_k^2\norm{w^k}_X^2.
        \end{aligned}
    \end{equation}
    Hence, $\{\norm{x^k-x^*}_X\}_{k\in \N}$ is decreasing for \emph{any} root $x^*$ (such sequences are called \emph{Féjer monotone}) and thus bounded.
    This implies that $\{x^k\}_{k\in\N}\subset X$ is bounded as well and thus contains a weakly convergent subsequence $x^{k_l}\weakto \bar x$.

    Furthermore, recursive application of \eqref{eq:proximal:fixed_fejer} yields that
    \begin{equation*}
        0\leq \norm{x^{k+1}-x^*}_X^2 \leq \norm{x^0-x^*}_X^2 - \sum_{j=0}^k \gamma_j^2\norm{w^j}_X^2.
    \end{equation*}
    The (increasing) sequence of partial sums on the right-hand side is therefore bounded and hence $\sum_{k=0}^\infty \gamma_k^2 \norm{w^k}_X^2$ is finite. Since the sequence $\{\gamma_k^2\}_{k\in\N}$ is not summable by assumption, this requires that $\liminf_{k\to\infty}\norm{w^k}_X^2=0$.
    This together with the convergence of $\{\norm{w^k}_X\}_{k\in\N}$ implies that $w^k \to 0$. In particular, we have that $Ax^{k_l+1}\ni w^{k_l}\to 0$ for $x^{k_l+1}\weakto \bar x$, and the closedness of maximally monotone operators (\cref{cor:monoton:closed}) yields that $0\in A\bar x$. Hence, every weak accumulation point of $\{x^k\}_{k\in\N}$ is a root of $A$.

    We finally show convergence of the full sequence $\{x^k\}_{k\in\N}$.\footnote{The following argument in a more general setting is known as \emph{Opial's Lemma}.} Let $\bar x$ and $\hat x$ be weak accumulation points and therefore roots of $A$.
    The Féjer monotonicity of $\{x^k\}_{k\in\N}$ then implies that both $\{\norm{x^k-\bar x}_X\}_{k\in\N}$ and $\{\norm{x^k-\hat x}_X\}_{k\in\N}$ are decreasing and bounded from below and therefore convergent.
    This implies that
    \begin{equation*}
        \inner{x^k,\bar x - \hat x}_X = \frac12\left(\norm{x^k-\hat x}_X^2 - \norm{x^k-\bar x}_X^2 + \norm{\bar x}_X^2 - \norm{\hat x}_X^2\right)\to c\in \R.
    \end{equation*}
    Since $\bar x$ is a weak accumulation point, there exists a subsequence  $\{x^{k_n}\}_{n\in \N}$ with $x^{k_n}\weakto \bar x$; similarly, there exists a subsequence $\{x^{k_m}\}_{m\in \N}$ with $x^{k_m}\weakto \hat x$. Hence,
    \begin{equation*}
        \inner{\bar x,\bar x - \hat x}_X =  \lim_{n\to \infty} \inner{x^{k_n},\bar x - \hat x}_X = c =  \lim_{m\to \infty} \inner{x^{k_m},\bar x - \hat x}_X=\inner{\hat x,\bar x - \hat x}_X,
    \end{equation*}
    and therefore
    \begin{equation*}
        0 = \inner{\bar x - \hat x,\bar x - \hat x}_X=\norm{\bar x -\hat x}_X^2,
    \end{equation*}
    i.e., $\bar x =\hat x$. Every convergent subsequence thus has the same limit, which by a subsequence--subsequence argument must therefore be the limit of the full sequence $\{x^k\}_{k\in\N}$.
\end{proof}

\section{Explicit splitting}

As we have repeatedly noted, the proximal point method is not feasible for most functionals of the form  $J(x) = F(x)+G(x)$, since the evaluation of $\prox_J$ is not significantly easier than solving the original minimization problem -- even if $\prox_F$ and $\prox_G$ have a closed-form expression (i.e., are \emph{prox-simple}).
We thus proceed differently: instead of applying the proximal point reformulation directly to $0\in\partial J(\bar x)$, we first apply the sum rule and obtain a $\bar p\in X$ with
\begin{equation}\label{eq:splitting:optsys}
    \left\{\begin{aligned}
            \bar p &\in \partial F(\bar x),\\
            -\bar p &\in \partial G(\bar x).
    \end{aligned}\right.
\end{equation}
We can now replace one or both of these subdifferential inclusions by a proximal point reformulation that only involves $F$ or $G$.

Explicit splitting methods apply \cref{lem:proximal:subdiff} only to the second inclusion in \eqref{eq:splitting:optsys} to obtain
\begin{equation*}
    \left\{\begin{aligned}
            \bar p &\in \partial F(\bar x),\\
            \bar x &= \prox_{\gamma G}(\bar x -\gamma \bar p).
    \end{aligned}\right.
\end{equation*}
The corresponding fixed-point iteration then consists in
\begin{enumerate}
    \item choosing $p^k \in \partial F(x^k)$ (with minimal norm);
    \item setting $x^{k+1} =  \prox_{\gamma_k G}(x^k -\gamma_k p^k)$.
\end{enumerate}
Again, computing a subgradient with minimal norm can be complicated in general. It is, however, easy if $F$ is additionally differentiable since in this case $\partial F(x) = \{\nabla F(x)\}$ is a singleton. This leads to the \emph{proximal gradient method} or \emph{forward-backward splitting method}
\begin{equation}\label{eq:splitting:fb}
    x^{k+1} = \prox_{\gamma_k G}(x^k - \gamma_k \nabla F(x^k)).
\end{equation}
(The special case $G=\delta_C$ -- i.e., $\prox_{\gamma G}(x) = \proj_C(x)$ -- is also known as the \emph{projected gradient method}).

Showing convergence of the proximal gradient method as for the proximal point method requires assuming Lipschitz continuity of the gradient (since we are not using a proximal point mapping for $F$ which is always firmly nonexpansive and hence Lipschitz continuous). The following lemma may be familiar from nonlinear optimization.
\begin{lemma}\label{lem:splitting:lipschitz}
    Let $F:X\to\R$ be Gâteaux differentiable with Lipschitz continuous gradient. Then,
    \begin{equation*}
        F(y) \leq F(x) + \inner{\nabla F(x), x-y}_X + \frac{L}2 \norm{x-y}_X^2 \quad\text{for all }x,y\in X.
    \end{equation*}
\end{lemma}
\begin{proof}
    The Gâteaux differentiability of $F$ implies that
    \begin{equation*}
        \frac{d}{dt}F(x+t(y-x)) = \inner{\nabla F(x+t(y-x)),y-x}_X \quad\text{for all }x,y\in X,
    \end{equation*}
    and integration over all $t\in [0,1]$ yields that
    \begin{equation*}
        \int_0^1 \inner{\nabla F(x+t(y-x)),y-x}_X\,dt = F(y)-F(x).
    \end{equation*}
    From this, we obtain together with the productive zero, the Cauchy--Schwarz inequality, and the Lipschitz continuity of the gradient that
    \begin{equation*}
        \begin{split}\begin{aligned}[b]
            F(y) &= F(x) + \inner{\nabla F(x),y-x}_X + \int_0^1 \inner{\nabla F(x+t(y-x))-\nabla F(x),y-x}_X\,dt\\
            & \leq F(x) +  \inner{\nabla F(x),y-x}_X +\int_0^1 \norm{\nabla F(x+t(y-x))-\nabla F(x)}_X\norm{x-y}_X\,dt\\
            &\leq  F(x) + \inner{\nabla F(x),y-x}_X + \int_0^1 L t\norm{x-y}_X^2\,dt\\
            & =  F(x) + \inner{\nabla F(x),y-x}_X + \frac{L}{2}\norm{x-y}_X^2.
        \end{aligned}
        \qedhere
    \end{split}
\end{equation*}
\end{proof}

We can now show convergence of the proximal gradient method for sufficiently small step sizes.
\begin{theorem}\label{thm:splitting:imp_conv}
    Let $F:X\to\R$ and $G:X\to\Rbar$ be proper, convex, and lower semicontinuous. Furthermore, let $F$ be Gâteaux differentiable with Lipschitz continuous gradient. If $0<\gamma_{\min}\leq \gamma_k\leq L^{-1}$, the sequence generated by \eqref{eq:splitting:fb} converges weakly to a minimizer $\bar x\in X$ of $J$.
\end{theorem}
\begin{proof}
    We argue similarly as in the proof of \cref{thm:proximal:fixedpoint}, replacing the monotonicity of the generalized residuals $w^k\in Ax^{k+1}$ with those of the functional values $J(x^k)$.
    For this purpose, we define the operator
    \begin{equation*}
        T_\gamma:X\to X,\qquad x\mapsto \gamma^{-1} (x-\prox_{\gamma G}(x-\gamma \nabla F(x))),
    \end{equation*}
    which allows reformulating the iteration \eqref{eq:splitting:fb} as
    \begin{equation*}
        x^{k+1} = \prox_{\gamma_k G}(x^k-\gamma_k \nabla F(x^k)) = x^k - \gamma_k T_{\gamma_k}(x^k).
    \end{equation*}
    Applying \cref{lem:proximal:subdiff} to the second equality then implies that
    \begin{equation}\label{eq:splitting:conv1}
        T_{\gamma_k} (x^k) - \nabla F(x^k) \in \partial G(x^k-\gamma_k T_{\gamma_k}(x^k)).
    \end{equation}

    \Cref{lem:splitting:lipschitz} with $x=x^k$, $y=x^{k+1} =  x^k - \gamma_k T_{\gamma_k}(x^k)$, and $\gamma_k\leq L^{-1}$ further implies that
    \begin{equation}\label{eq:splitting:conv_f1}
        \begin{aligned}[t]
            F(x^k - \gamma_k T_{\gamma_k}(x^k)) &\leq F(x^k) - \gamma_k\inner{\nabla F(x^k),T_{\gamma_k}(x^k)}_X + \frac{\gamma_k^2 L}{2} \norm{T_{\gamma_k}(x^k)}_X^2\\
            &\leq F(x^k) - \gamma_k\inner{\nabla F(x^k),T_{\gamma_k}(x^k)}_X + \frac{\gamma_k}{2} \norm{T_{\gamma_k}(x^k)}_X^2.
        \end{aligned}
    \end{equation}
    Hence, using  \eqref{eq:splitting:conv1} and $\nabla F(x) \in \partial F(x)$, we obtain for all $z\in X$ that
    \begin{equation}\label{eq:splitting:conv_f2}
        \begin{aligned}[t]
            J(x^{k+1})  &= F(x^k - \gamma_k T_{\gamma_k}(x^k)) + G(x^k - \gamma_k T_{\gamma_k}(x^k))\\
            &\leq  F(x^k) - \gamma_k\inner{\nabla F(x^k),T_{\gamma_k}(x^k)}_X + \frac{\gamma_k}{2} \norm{T_{\gamma_k}(x^k)}_X^2\\
            \MoveEqLeft[-1] + G(z) + \inner{T_{\gamma_k}(x^k) -\nabla F(x^k),x^k-\gamma_k T_{\gamma_k}(x^k)-z}_X\\
            &\leq  F(z) +\inner{\nabla F(x^k),x^k-z}_X - \gamma_k\inner{\nabla F(x^k),T_{\gamma_k}(x^k)}_X + \frac{\gamma_k}{2} \norm{T_{\gamma_k}(x^k)}_X^2\\
            \MoveEqLeft[-1] + G(z) + \inner{T_{\gamma_k}(x^k) -\nabla F(x^k),x^k-z-\gamma_k T_{\gamma_k}(x^k)}_X\\
            &= J(z) + \inner{T_{\gamma_k}(x^k),x^k-z}_X - \frac{\gamma_k}{2} \norm{T_{\gamma_k}(x^k)}_X^2.
        \end{aligned}
    \end{equation}

    For $z=x^k$ this implies that
    \begin{equation*}
        J(x^{k+1}) \leq J(x^k) - \frac{\gamma_k}{2} \norm{T_{\gamma_k}(x^k)}_X^2,
    \end{equation*}
    i.e., $\{J(x^k)\}_{k\in\N}$ is decreasing. (The proximal gradient method is thus a \emph{descent method}.)
    Furthermore, by inserting $z=x^*$ with $J(x^*)=\min_{x\in X} J(x)$ in \eqref{eq:splitting:conv_f2} and completing the square, we deduce that
    \begin{equation}\label{eq:splitting:conv_x1}
        \begin{aligned}[t]
            0 \leq  J(x^{k+1}) - J(x^*) &\leq \inner{T_{\gamma_k}(x^k),x^k-x^*}_X - \frac{\gamma_k}{2}\norm{T_{\gamma_k}(x^k)}_X^2\\
            &= \frac{1}{2\gamma_k}\left(\norm{x^k-x^*}_X^2 - \norm{x^k - x^* - \gamma_k T_{\gamma_k}(x^k)}_X^2\right)\\
            &= \frac{1}{2\gamma_k}\left(\norm{x^k-x^*}_X^2 - \norm{x^{k+1} - x^*}_X^2\right).
        \end{aligned}
    \end{equation}
    In particular, $\{\norm{x^{k}-x^*}_X\}_{k\in\N}$ is decreasing, and hence  $\{x^k\}_{k\in\N}$ is Féjer monotone and therefore bounded. We can thus extract a weakly convergent subsequence $\{x^{k_l}\}_{l\in\N}$ with $x^{k_l}\weakto \bar x$.

    We now sum \eqref{eq:splitting:conv_x1} over $k=1,\dots,n$ for arbitrary $n\in\N$ and obtain that
    \begin{equation*}
        \begin{aligned}
            \sum_{k=1}^n ( J(x^{k}) - J(x^*) ) &\leq \frac1{2\gamma_{\min}} \sum_{k=1}^n \left(\norm{x^{k-1}-x^*}_X^2 - \norm{x^{k} - x^*}_X^2\right)\\
            &= \frac1{2\gamma_{\min}} \left(\norm{x^{0}-x^*}_X^2 - \norm{x^{n} - x^*}_X^2\right)\\
            &\leq \frac1{2\gamma_{\min}} \norm{x^{0}-x^*}_X^2.
        \end{aligned}
    \end{equation*}
    Since $\{J(x^k)\}_{k\in\N}$ is decreasing, this implies that
    \begin{equation}\label{eq:splitting:conv_f_o}
        J(x^n) - J(x^*) \leq \frac1n  \sum_{k=1}^n ( J(x^{k}) - J(x^*) )  \leq  \frac1{2n\gamma_{\min}} \norm{x^{0}-x^*}_X^2
    \end{equation}
    and hence $J(x^n) \to J(x^*)$ for $n\to \infty$. The lower semicontinuity of $F$ and $G$ now yields that
    \begin{equation*}
        J(\bar x) \leq \liminf_{l\to\infty} J(x^{k_l}) = J(x^*).
    \end{equation*}
    As in the proof of \cref{thm:proximal:fixedpoint}, we can use the Féjer monotonicity of $\{x^{k}\}_{k\in\N}$ to show that $x^k\weakto \bar x$ for the full sequence.
\end{proof}

In particular, we obtain from \eqref{eq:splitting:conv_f_o} that $J(x^k) = J(x^*)+ \mathcal{O}(k^{-1})$. Ensuring $J(x^k) \leq J(x^*) + \eps$ thus requires $\mathcal{O}(\eps^{-1})$ iterations.
By introducing a clever extrapolation, this can be reduced to $\mathcal{O}(\eps^{-1/2})$ which is provably optimal; see \cite{Nesterov}, \cite[Theorem 2.1.7]{Nesterov:2004}. (However, the sequence of iterates is then no longer monotonically decreasing.) The corresponding iteration is given by
\begin{equation*}
    \left\{\begin{aligned}
            x^{k+1} &= \prox_{\gamma_k G}(\bar x^k - \gamma_k \nabla F(\bar x^k)),\\
            \bar x^{k+1} &= x^{k+1} + \frac{1-\tau_k}{\tau_{k+1}}\left(x^{k} - x^{k+1}\right),
    \end{aligned}\right.
\end{equation*}
for the (hardly intuitive) choice\footnote{This choice satisfies the quadratic recursion $\tau_{k+1}^2-\tau_{k+1}=\tau_k$, which cancels the $\mathcal{O}(k^{-1})$ terms in a key estimate.}
\begin{equation*}
    \tau_0 = 1,\qquad \tau_k = \frac{1+\sqrt{1+4\tau_{k-1}^2}}{2} \ (\to \infty),
\end{equation*}
see \cite[\S{}\,4]{BeckTeboulle}.

One drawback of the explicit splitting is needing to know the Lipschitz constant $L$ of $\nabla F$ in order to choose admissible step sizes $\gamma_k$. Looking at the proof of \cref{thm:splitting:imp_conv}, we can see that this is only used to obtain the estimate \eqref{eq:splitting:conv_f1}.
Hence, if the Lipschitz constant is unknown, we can try to satisfy \eqref{eq:splitting:conv_f1} by a line search in each iteration: Start with $\gamma^0>0$ and reduce $\gamma_k$ (e.g., by halving) until
\begin{equation*}
    F(x^k - \gamma_k T_{\gamma_k}(x^k)) \leq F(x^k) - \gamma_k\inner{\nabla F(x^k),T_{\gamma_k}(x^k)}_X + \frac{\gamma_k}{2} \norm{T_{\gamma_k}(x^k)}_X^2
\end{equation*}
(which will be the case for $\gamma_k<L^{-1}$ at the latest).
Of course, there's no free lunch: each step of the line search requires evaluating both $F$ and $\prox_{\gamma_k G}$ (although the latter can be avoided by exchanging gradient and proximal steps, i.e., \emph{backward--forward splitting}).

\section{Implicit splitting}

Even with a line search, the restriction on the step sizes $\gamma_k$ in explicit splitting remains unsatisfactory. Such restrictions are not needed in implicit splitting methods (compare the properties of explicit vs.~implicit Euler methods for differential equations). Here, the proximal point formulation is applied to both subdifferential inclusions in \eqref{eq:splitting:optsys}, which yields the optimality system
\begin{equation*}
    \left\{\begin{aligned}
            \bar x &= \prox_{\gamma F}(\bar x + \gamma \bar p),\\
            \bar x &= \prox_{\gamma G}(\bar x - \gamma \bar p).
    \end{aligned}\right.
\end{equation*}
To eliminate $\bar p$ from these equations, we set $\bar z  := \bar x+\gamma \bar p$ and $\bar w := \bar x - \gamma \bar p$. This implies that $\bar z + \bar w = 2\bar x$, i.e.,
\begin{equation*}
    \bar w  = 2\bar x - \bar z.
\end{equation*}
It remains to derive a recursion for $\bar z$, which we obtain from the tautology
\begin{equation*}
    \bar z = \bar z + (\bar x - \bar x).
\end{equation*}
Replacing two (suitable) occurences of $\bar x$ by a new $\bar y$ in these four equations and then applying a fixed-point iteration yields the \emph{Douglas--Rachford method}
\begin{equation}\label{eq:splitting:dr}
    \left\{\begin{aligned}
            x^{k+1} &= \prox_{\gamma F}(z^k),\\
            y^{k+1} &= \prox_{\gamma G}(2x^{k+1} - z^k),\\
            z^{k+1} &= z^k + y^{k+1} - x^{k+1}.
    \end{aligned}\right.
\end{equation}

This iteration can be written as a proximal point iteration by introducing suitable block operators, which with some effort (in showing that these operators are maximally monotone) allows deducing the convergence from \cref{thm:proximal:fixedpoint}; see, e.g., \cite{Eckstein:1992}. Here we will instead consider a variant which has proved extremely successful, in particular in mathematical imaging, inverse problems, and optimal control.

\section{Primal-dual splitting}\label{sec:proximal:pd}

Methods of this class were specifically developed to solve problems of the form
\begin{equation*}
    \min_{x\in X} F(x) + G(Ax)
\end{equation*}
for $F:X\to\Rbar$ and $G:Y\to\Rbar$ proper, convex, and lower semicontinuous, and $A\in L(X,Y)$. Applying \cref{thm:convex:fenchel,lem:convex:fenchel-young} to such a problem yields the Fenchel extremality conditions
\begin{equation}\label{eq:splitting:fenchel}
    \left\{    \begin{aligned}
            -A^*\bar{y} &\in \partial F(\bar x),\\
            \bar{y} &\in \partial G(A\bar x) ,
    \end{aligned}\right.
    \quad\equivalent \quad
    \left\{    \begin{aligned}
            -A^*\bar{y} &\in \partial F(\bar x),\\
            A\bar x &\in \partial G^*(\bar y),
    \end{aligned}\right.
\end{equation}
which can be reformulated using \cref{lem:proximal:subdiff} as
\begin{equation*}
    \left\{\begin{aligned}
            \bar x &= \prox_{\tau F}(\bar x - \tau A^*\bar y),\\
            \bar y &= \prox_{\sigma G^*}(\bar y + \sigma A\bar x),
    \end{aligned}        \right.
\end{equation*}
for arbitrary $\sigma,\tau>0$. This suggests the fixed-point iteration
\begin{equation}\label{eq:splitting:pd0}
    \left\{\begin{aligned}
            x^{k+1} &= \prox_{\tau F}(x^k - \tau A^*y^k),\\
            y^{k+1} &= \prox_{\sigma G^*}(y^{k} + \sigma A x^{k+1}),
    \end{aligned}\right.
\end{equation}
(where we have left the step sizes constant for simplicity). We now try to show convergence by interpreting it as a proximal point method. To that end, we rewrite \eqref{eq:splitting:pd0} in fully explicit form to have $(x^{k+1},y^{k+1})$ and $(x^{k},y^{k})$ on different sides. For the first equation, we use $\prox_{\tau F} = (\Id +\tau\partial F)^{-1}$ to obtain that
\begin{equation*}
    \begin{aligned}
        x^{k+1}  = \prox_{\tau F}(x^k - \tau A^*y^k) &\equivalent x^k - \tau A^* y^k \in  \{x^{k+1}\} + \tau \partial F(x^{k+1}) \\
        &\equivalent \tau^{-1} x^k -  A^* y^k \in \{\tau^{-1} x^{k+1}\} + \partial F(x^{k+1}).
    \end{aligned}
\end{equation*}
Similarly, for the second equation we have
\begin{equation*}
    y^{k+1}  = \prox_{\sigma G^*}(y^{k} + \sigma Ax^{k+1}) \equivalent
    \sigma^{-1}y^{k} \in \{\sigma^{-1} y^{k+1}-Ax^{k+1}\} + \partial G^*(y^{k+1}).
\end{equation*}
Setting $Z = X\times Y$, $z=(x,y)$, as well as
\begin{equation*}
    M = \begin{pmatrix} \tau^{-1}\Id &-A^* \\ 0 & \sigma^{-1}\Id \end{pmatrix},\qquad
    T = \begin{pmatrix} \partial F & A^* \\ -A & \partial G^* \end{pmatrix},
\end{equation*}
we see that \eqref{eq:splitting:pd0} is equivalent to
\begin{equation*}
    Mz^k \in (M+T) z^{k+1} \qquad \equivalent \qquad z^{k+1} \in (M+T)^{-1}Mz^k.
\end{equation*}
If $M$ were invertible, we could use that $M=(M^{-1})^{-1}$ to obtain that $(M+T)^{-1}Mz^{k}=(\Id + M^{-1} T)^{-1}z^k$; the iteration would indeed amount to a proximal point method for the operator $M^{-1}T$ (which hopefully is maximally monotone).

Unfortunately, we cannot show the desired invertibility in this form. We therefore replace $M$ by a self-adjoint operator for which we can show positive definiteness; i.e., we consider
\begin{equation*}
    M = \begin{pmatrix} \tau^{-1}\Id &-A^* \\ -A & \sigma^{-1}\Id \end{pmatrix},
\end{equation*}
so that the second step in the iteration becomes
\begin{equation*}
    \sigma^{-1}y^{k} - Ax^{k} \in \{\sigma^{-1} y^{k+1}-2Ax^{k+1}\} + \partial G^*(y^{k+1})
    \equivalent
    y^{k+1}  = \prox_{\sigma G^*}(y^{k} + \sigma A(2x^{k+1}-x^k)).
\end{equation*}
This yields the \emph{primal-dual extragradient method}\footnote{This method was introduced in \cite{Pock_PD_2010}, which is why it is frequently referred to as the \emph{Chambolle--Pock method}. The relation to proximal point methods was first pointed out in \cite{He:2012}.}
\begin{equation}\label{eq:splitting:pd}
    \left\{\begin{aligned}
            x^{k+1} &= \prox_{\tau F}(x^k - \tau A^*y^k),\\
            \bar x^{k+1} &= 2x^{k+1}-x^k,\\
            y^{k+1} &= \prox_{\sigma G^*}(y^{k} + \sigma A\bar x^{k+1}).
    \end{aligned}        \right.
\end{equation}

We now show that -- under suitable conditions on $\sigma$ and $\tau$ -- the operator $M$ is self-adjoint and positive definite with respect to the inner product
\begin{equation*}
    \inner{z_1,z_2}_Z = \inner{x_1,x_2}_X  + \inner{y_1,y_2}_Y\quad\text{for all }z_1=(x_1,y_1)\in Z, z_2=(x_2,y_2)\in Z.
\end{equation*}
\begin{lemma}\label{lem:splitting:pd_spd}
    The operator $M:Z\to Z$ is bounded and self-adjoint. If $\sigma\tau\norm{A}_{L(X,Y)}^2 < 1$, then $M$ is uniformly positive definite.
\end{lemma}
\begin{proof}
    The definition of $M$ directly implies boundedness (since $A\in L(X,Y)$ is bounded) and self-adjointness. Let now $z=(x,y)\in Z\setminus\{0\}$ be given. Then,
    \begin{equation*}
        \begin{aligned}
            \inner{Mz,z}_Z &= \inner{\tau^{-1}x - A^*y,x}_X + \inner{\sigma^{-1}y-Ax,y}_Y\\
            &= \tau^{-1}\norm{x}_X^2 - 2 \inner{x,A^*y}_X + \sigma^{-1}\norm{y}_Y^2\\
            &\geq \tau^{-1}\norm{x}_X^2 - 2 \norm{A}_{L(X,Y)}\norm{x}_X\norm{y}_Y + \sigma^{-1}\norm{y}_Y^2\\
            &\geq \tau^{-1}\norm{x}_X^2 - \norm{A}_{L(X,Y)}\sqrt{\sigma\tau}(\tau^{-1}\norm{x}_X^2+\sigma^{-1}\norm{y}_Y^2)+ \sigma^{-1}\norm{y}_Y^2\\
            &= (1-\norm{A}_{L(X,Y)}\sqrt{\sigma\tau})(\sqrt{\tau}^{-1}\norm{x}_X^2 + \sqrt{\sigma}^{-1}\norm{y}_Y^2) \\
            &\geq C(\norm{x}_X^2 + \norm{y}_Y^2)
        \end{aligned}
    \end{equation*}
    for $C:=(1-\norm{A}_{L(X,Y)}\sqrt{\sigma\tau})\min\{\tau^{-1},\sigma^{-1}\}>0$.
    Hence, $\inner{Mz,z}_Z>C\norm{z}^2_Z$ for all $z\neq 0$, and therefore $M$ is positive definite.
\end{proof}
Under these conditions, the operator $M$ induces an inner product $\inner{z_1,z_2}_M := \inner{Mz_1,z_2}_Z$ and, through it, a norm $\norm{z}_M^2 = \inner{z,z}_M$ that satisfies
\begin{equation}
    \label{eq:splitting:pd_M_eq}
    c_1 \norm{z}_Z  \leq \norm{z}_M \leq c_2 \norm{z}_Z \qquad\text{for all }z\in Z,
\end{equation}
where $c_1 = \sqrt{C}>0$ from the proof of \cref{lem:splitting:pd_spd} and $c_2:=\norm{M}_{L(Z,Z)}>0$.
From this, we can deduce continuous invertibility of $M$ by a standard functional-analytic argument (a special case of the \emph{Lax--Milgram Theorem}).
\begin{cor}\label{cor:splitting:precond-inv}
    If $\sigma\tau\norm{A}_{L(X,Y)}^2 < 1$, then $M$ is continuously invertible, i.e., $M^{-1}\in L(Y,X)$.
\end{cor}
\begin{proof}
    Let $z\in Z$ be given. Then \eqref{eq:splitting:pd_M_eq} implies that the mapping $v\mapsto \inner{z,v}_Z$ is a bounded  (with respect to $\norm{\cdot}_M$) linear functional. The \nameref{thm:frechetriesz} \cref{thm:frechetriesz} applied to the Hilbert space $(Z,\inner{\cdot,\cdot}_M)$ thus yields a unique preimage $z^*\in Z$ with
    \begin{equation*}
        \inner{Mz^*,v}_Z = \inner{z^*,v}_M=\inner{z,v}_Z \qquad\text{for all } v\in Z.
    \end{equation*}
    Furthermore, the Riesz mapping $M^{-1}:z\mapsto z^*$ is linear.
    Hence,
    \begin{equation*}
        c_1^2 \norm{z^*}_Z^2 \leq  \norm{z^*}_M^2 = \inner{Mz^*,z^*}_Z = \inner{z,z^*}_Z \leq \norm{z}_{Z}\norm{z^*}_Z,
    \end{equation*}
    and dividing by $c_1^{2}\norm{z^*}_Z$ yields the claimed boundedness of $M^{-1}$.
\end{proof}
Hence $M^{-1}T$ is well-defined, i.e., $\graph M^{-1}T\neq \emptyset$. We now show maximal monotonicity with respect to the inner product $\inner{\cdot,\cdot}_M$.
\begin{lemma}\label{lem:splitting:pd_maxmon}
    If $\sigma\tau\norm{A}_{L(X,Y)}^2 < 1$, then $M^{-1}T$ is maximally monotone on $(Z,\inner{\cdot,\cdot}_M)$.
\end{lemma}
\begin{proof}
    We first show the monotonicity of $M^{-1}T$. Let $z\in Z$ and $z^*\in M^{-1}Tz$, i.e., $Mz^*\in Tz$.
    By definition of $T$, we can thus find for any $z=(x,y)$ a $\xi\in \partial F(x)$ and an $\eta\in \partial G^*(y)$ with $Mz^* = (\xi+A^*y,\eta-Ax)$. Similarly, for given $\bar z =(\bar x,\bar y)\in Z$ and $\bar z^*\in M^{-1} T\bar z$ we can write $M\bar z^* = (\bar \xi +A^*\bar y,\bar \eta - A\bar x)$ for a $\bar\xi\in\partial  F(\bar x)$ and an $\bar\eta \in \partial G^*(\bar y)$. Hence
    \begin{equation*}
        \begin{aligned}
            \inner{\bar z^*-z^*,\bar z-z}_M=   \inner{M\bar z^*-Mz^*,\bar z - z}_Z &=
            \inner{(\bar \xi + A^*\bar y) -(\xi + A^* y),\bar x-x}_X \\
            \MoveEqLeft[-1] + \inner{(\bar \eta - A\bar x) -(\eta - Ax),\bar y-y}_Y\\
            &=\inner{\bar \xi-\xi,\bar x-x}_X +\inner{A^*(\bar y -y),\bar x-x}_X \\
            \MoveEqLeft[-1]-\inner{A(\bar x -x),\bar y -y}_Y + \inner{\bar \eta-\eta,\bar y-y}_Y\\
            &=\inner{\bar \xi-\xi,\bar x-x}_X + \inner{\bar \eta-\eta,\bar y-y}_Y \geq 0
        \end{aligned}
    \end{equation*}
    by the monotonicity of subdifferentials.

    To show maximal monotonicity, let $\bar z^*,\bar z\in Z$ with
    \begin{equation}\label{eq:splitting:maxmonoton1}
        \inner{M\bar z^*-Mz^*,\bar z - z}_Z =\inner{\bar z^*-z^*,\bar z - z}_M \geq 0 \quad\text{for all } (z,z^*)\in\graph M^{-1}T,
    \end{equation}
    i.e., for all $z\in Z$ and $Mz^*\in Tz$.
    As above, we can write $Mz^* = (\xi+A^*y,\eta-Ax)$ for some $\xi \in \partial F(x)$ and $\eta\in \partial G^*(y)$. We now set $\bar \xi := \bar x^* - A^*\bar y$ and $\bar \eta  := \bar y^* + A\bar x$ for $M\bar z^* = (\bar x^*,\bar y^*)$ and $\bar z=(\bar x,\bar y)$. Then $M\bar z^*=(\bar\xi+A^*\bar y,\bar \eta - A\bar x)$, and \eqref{eq:splitting:maxmonoton1} implies for all $(x,y)\in Z$ that
    \begin{equation*}
        \begin{aligned}
            0 &\leq \inner{(\bar \xi + A^*\bar y) -(\xi + A^* y),\bar x-x}_X +  \inner{(\bar \eta - A\bar x) -(\eta - Ax),\bar y-y}_Y\\
            &=\inner{\bar \xi-\xi,\bar x-x}_X + \inner{\bar \eta-\eta,\bar y-y}_Y.
        \end{aligned}
    \end{equation*}
    In particular, this holds for pairs $(x,y)$ of the form $(x,\bar y)$ for arbitrary $x\in X$ or $(\bar x,y)$ for arbitrary $y\in Y$, which shows that each inner product on the right-hand side is nonnegative.
    The maximal monotonicity of subdifferentials now implies that $\bar \xi\in\partial F(\bar x)$ and $\bar \eta\in\partial G^*(\bar y)$.
    Hence
    \begin{equation*}
        M\bar z^* =  (\bar\xi + A^*\bar y,\bar \eta - A \bar x)\in T\bar z,
    \end{equation*}
    i.e., $\bar z^*\in M^{-1}T\bar z$. We conclude that $M^{-1}T$ is maximally monotone as claimed.
\end{proof}

In sum, we have shown that the primal-dual extragradient method \eqref{eq:splitting:pd} is equivalent to the proximal point method $z^{k+1} = \calR_{M^{-1}T}z^k$ for the maximally monotone operator $M^{-1}T$, and hence its convergence follows from \cref{thm:proximal:fixedpoint} together with the invertibility of $M$.
\begin{theorem}\label{thm:splitting:pd_conv}
    Let $F:X\to\Rbar$, $G:Y\to\Rbar$, and $A\in L(X,Y)$ satisfy the assumptions of \cref{thm:convex:fenchel}. If $\sigma\tau \norm{A}_{L(X,Y)}^2 <1$, then the sequence $\{(x^k,y^k)\}_{k\in\N}$ generated by \eqref{eq:splitting:pd} converges weakly to some $(\bar x,\bar y)\in X\times Y$ satisfying \eqref{eq:splitting:fenchel}.
\end{theorem}
\begin{proof}
    First, \cref{thm:convex:fenchel} yields the existence of a $\bar z:=(\bar x,\bar y)$ satisfying the Fenchel extremality relations \eqref{eq:splitting:fenchel}. By definition of $T$, this is equivalent to $0\in T\bar z$, which by the invertibility from $M$ due to \cref{cor:splitting:precond-inv} holds if and only if $0\in M^{-1}T\bar z$. Hence there exists a root of $M^{-1}T$. By \cref{lem:splitting:pd_maxmon}, $M^{-1}T$ is maximally monotone (with respect to $\inner{\cdot,\cdot}_M$) and hence we can apply \cref{thm:proximal:fixedpoint} to obtain that
    \begin{equation*}
        \inner{z^k,Mw}_Z = \inner{z_k,w}_M \to \inner{\bar z,w}_M = \inner{\bar z,Mw}_Z \quad\text{for all } w\in Z,
    \end{equation*}
    where $\bar z\in Z$ satisfies $0\in M^{-1}T\bar z$ and hence $0\in T\bar z$. Since $M$ is invertible and hence in particular surjective, this implies that $\inner{z_k,\tilde w}_Z\to \inner{\bar z,\tilde w}_Z$ for all $\tilde w:=Mw\in Z$, which is the claimed weak convergence.
\end{proof}
Note that although the iteration is implicit in $F$ and $G$, it is still explicit in $A$; it is therefore not surprising that step size restrictions based on $A$ remain.\footnote{Using a proximal point mapping for $G\circ A$ would lead to a fully implicit method but involve the inverse $A^{-1}$ in the corresponding proximal point mapping. It is precisely the point of the primal-dual extragradient method to avoid having to invert $A$, which is often prohibitively expensive if not impossible (e.g., if $A$ does not have closed range as in many inverse problems).}

Finally, we remark that by setting $A=\Id$, $\tau = \gamma$, $\sigma = \gamma^{-1}$ and $z^k = x^k - \gamma y^k$ in \eqref{eq:splitting:pd} and applying \cref{lem:proximal:calculus}\,(ii), we recover the Douglas--Rachford method \eqref{eq:splitting:dr}; however, since in this case $\sigma\tau\norm{A}_{L(X,Y)}^2=1$, we cannot obtain its convergence from \cref{thm:splitting:pd_conv}.

\section{Strong convergence and rates}

The central idea of the convergence proofs we have seen so far is to use the iteration step together with the monotonicity of the set-valued mapping to obtain the boundedness of the error $\norm{x^k-\bar x}_X$ and thus the weak convergence (at first of a subsequence) of the iterates. For \emph{strong} convergence, however, we require additional properties that yield a more direct relation between the iteration step and the error, which will also allow deriving \emph{convergence rates}.

One possibility is the following: A set-valued mapping $H:X\setto X$ is called \emph{strongly monotone} if there exists a $\gamma>0$ such that
\begin{equation}\label{eq:splitting:strong-monotone}
    \inner{x_1^*-x_2^*,x_1-x_2}_X \geq \gamma \norm{x_1-x_2}_X^2\quad\text{for all }(x_1,x_1^*),(x_2,x_2^*)\in\graph H.
\end{equation}
For example, $H=\partial F$ for $F(x)=\frac12\norm{x}_X^2$ is clearly strongly monotone with $\gamma=1$; more generally, $\partial F$ is strongly monotone if $F-\frac\gamma2\norm{\cdot}_X^2$ is convex.

To illustrate the general approach, we show strong convergence for the proximal point method.
\begin{theorem}
    Let $H:X\setto X$ be strongly monotone with $\gamma>0$ and let $\bar x\in X$ satisfy $0\in H(\bar x)$. Furthermore, let $\{x^k\}_{k\in\N}$ be generated via
    \begin{equation*}
        x^{k+1}=\calR_{\tau_k H}(x^k)
    \end{equation*}
    for some $x^0\in X$ and $\{\tau_k\}_{k\in\N}\subset (0,\infty)$.
    \begin{enumerate}[(i)]
        \item If $\tau_k\equiv \tau$ is constant, then $x^k\to \bar x$ linearly, i.e., $\lim_{k\to\infty} \frac{\norm{x^{k+1}-\bar x}_X}{\norm{x^{k}-\bar x}_X}=\mu <1$.
        \item If $\tau_k\to\infty$, then $x^k\to \bar x$ superlinearly, i.e., $\lim_{k\to\infty} \frac{\norm{x^{k+1}-\bar x}_X}{\norm{x^{k}-\bar x}_X}=0$.
    \end{enumerate}
\end{theorem}
\begin{proof}
    By definition of the resolvent, the iteration step is equivalent to
    \begin{equation*}
        -\frac1{\tau_k}(x^{k+1}-x^k) \in H(x^{k+1}).
    \end{equation*}
    Together with $0\in H(\bar x)$, it thus follows from \eqref{eq:splitting:strong-monotone} that
    \begin{equation*}
        -\inner{x^{k+1}-x^k,x^{k+1}-\bar x}_X \geq \tau_k\gamma \norm{x^{k+1}-\bar x}_X^2.
    \end{equation*}
    We now apply to the left-hand side the (easily verified) \emph{three-point identity}
    \begin{equation*}
        \inner{x-y,x-z}_X = \frac12\norm{x-y}_X^2 - \frac12\norm{y-z}_X^2 + \frac12\norm{x-z}_X^2\quad\text{for all }x,y,z\in X
    \end{equation*}
    for $x=x^{k+1}$, $y=x^k$, and $z=\bar x$. After rearranging, we obtain
    \begin{equation*}
        \frac{1+2\tau_k\gamma}{2}\norm{x^{k+1}-\bar x}_X^2 + \frac12\norm{x^{k+1}-x^k}_X^2 \leq \frac12\norm{x^k-\bar x}_X^2.
    \end{equation*}
    In particular, it follows that
    \begin{equation*}
        \frac{\norm{x^{k+1}-\bar x}_X^2}{\norm{x^k-\bar x}_X^2} \leq \frac{1}{1+2\tau_k\gamma}.
    \end{equation*}
    We now make the case distinction:
    \begin{enumerate}[(i)]
        \item if $\tau_k\equiv \tau$, then $\mu:=(1+2\tau \gamma)^{-1/2} < 1$ and hence $x^k\to \bar x$ linearly;
        \item if $\tau_k\to \infty$, then $(1+2\tau_k \gamma)^{-1/2} \to 0$  and hence $x^k\to \bar x$ superlinearly.
            \qedhere
    \end{enumerate}
\end{proof}

Similarly, one can show with a bit more effort that explicit splitting for strongly convex $G$ converges linearly (but not superlinearly, since $\tau_k\leq L^{-1}$ has to remain bounded).
For the primal-dual extragradient method, this is possible as well (with significantly more effort) if the definition of strong monotonicity and the three-point identity are adapted to norms and inner products weighted with a \enquote{testing operator} that depends on the step sizes and the desired convergence rate; see \cite{tuomov-proxtest}.

\part{Lipschitz analysis}

\chapter{Clarke subdifferentials}\label{chap:clarke}

We now turn to a concept of generalized derivatives that covers, among others, both Fréchet derivatives and convex subdifferentials. Again, we start with the general class of functionals that admit such a derivative; these are the locally Lipschitz continuous functionals. Recall that $F:X\to\R$ is locally Lipschitz continuous in $x\in X$ if there exist a $\delta >0$ and an $L>0$ (which in the following will always denote the local Lipschitz constant of $F$) such that
\begin{equation*}
    |F(x_1)-F(x_2)|\leq L \norm{x_1-x_2}_X \qquad\text{for all }x_1,x_2\in O_\delta (x).
\end{equation*}
We will refer to the $O_\delta(x)$ from the definition as the \emph{Lipschitz neighborhood} of $x$.
Note that in contrast to convexity, this is a purely local condition; on the other hand, we have to require that $F$ is (locally) finite-valued.\footnote{For $F:X\to\Rbar$, this is always the case in the interior of the effective domain. It is also possible to extend the generalized derivative introduced below to points on the boundary of the effective domain where $F$ is finite using an equivalent, more geometrical, construction involving generalized normal cones to epigraphs; see \cite[Definition 2.4.10]{Clarke:1990a}.}

\section{Definition and basic properties}\label{sec:clarke:definition}

We proceed as for the convex subdifferential and first define for $F:X\to\R$ the
\emph{generalized directional derivative} in $x\in X$ in direction $h\in X$ as
\begin{equation*}
    F^\circ(x;h) := \limsup_{\substack{y\to x\\t\to 0^+}} \frac{F(y+th)-F(y)}{t}.
\end{equation*}
Note the difference to the classical directional derivative: We no longer require the existence of a limit but merely of accumulation points.
We will need the following properties.
\begin{lemma}\label{lem:clarke:dir}
    Let $F:X\to\R$ be locally Lipschitz continuous in $x\in X$. Then the mapping $h\mapsto F^\circ (x;h)$ is
    \begin{enumerate}[(i)]
        \item Lipschitz continuous with constant $L$ and satisfies $|F^\circ(x;h)|\leq L\norm{h}_X<\infty$;
        \item \emph{subadditive}, i.e., $F^\circ(x;h+g)\leq F^\circ(x;h)+F^\circ(x;g)$ for all $h,g\in X$;
        \item \emph{positively homogeneous}, i.e., $F^\circ(x;\alpha h) = (\alpha F)^\circ(x;h)$ for all $\alpha \geq 0$ and $h\in X$;
        \item \emph{reflective}, i.e., $F^\circ(x;-h) = (-F)^\circ(x;h)$ for all $h\in X$.
    \end{enumerate}
\end{lemma}
\begin{proof}
    \emph{(i):} Let $h,g\in X$ be arbitrary. The local Lipschitz continuity of $F$ implies that
    \begin{equation*}
        F(y+th)-F(y) \leq F(y+tg) - F(y) + t L\norm{h-g}_X
    \end{equation*}
    for all $y$ sufficiently close to $x$ and $t$ sufficiently small. Dividing by $t>0$ and taking the $\limsup$ then yields that
    \begin{equation*}
        F^\circ(x;h) \leq F^\circ(x;g) + L\norm{h-g}_X.
    \end{equation*}
    Exchanging the roles of $h$ and $g$ shows the Lipschitz continuity of $F^\circ(x;\cdot)$, which also yields the claimed boundedness since $F^\circ(x;g) = 0$ for $g=0$ from the definition.

    \emph{(ii):} The definition of the $\limsup$ and the productive zero immediately yield
    \begin{equation*}
        \begin{aligned}
            F^\circ(x;h+g) &= \limsup_{\substack{y\to x\\t\to 0^+}} \frac{F(y+th+tg)-F(y)}{t} \\
            &\leq
            \limsup_{\substack{y\to x\\t\to 0^+}} \frac{F(y+th+tg)-F(y+tg)}{t}+ \limsup_{\substack{y\to x\\t\to 0^+}} \frac{F(y+tg)-F(y)}{t}\\
            &= F^\circ(x;h)+ F^\circ(x;g),
        \end{aligned}
    \end{equation*}
    since $y\to x$ and $t\to 0$ implies that $y+tg\to x$ as well.

    \emph{(iii):} The claim is clear for $\alpha=0$. For $\alpha>0$, we obain again from the definition that
    \begin{equation*}
        \begin{aligned}
            F^\circ(x;\alpha h) &= \limsup_{\substack{y\to x\\t\to 0^+}} \frac{F(y-t(\alpha h))-F(y)}{t}\\
            &= \limsup_{\substack{y\to x\\\alpha t\to 0^+}} \alpha \frac{F(y+(\alpha t)h)-F(y)}{\alpha t} =  (\alpha F)^\circ(x;h).
        \end{aligned}
    \end{equation*}

    \emph{(iv):} Similarly, we have that
    \begin{equation*}
        \begin{aligned}
            F^\circ(x;-h) &= \limsup_{\substack{y\to x\\t\to 0^+}} \frac{F(y-th)-F(y)}{t}\\
            &= \limsup_{\substack{w\to x\\t\to 0^+}} \frac{-F(w+th)-(-F(w))}{t} =  (-F)^\circ(x;h),
        \end{aligned}
    \end{equation*}
    since $y\to x$ and $t\to 0$ implies that $w:=y-th \to x$ as well.
\end{proof}
In particular, \cref{lem:clarke:dir}\,(i--iii) implies that the mapping  $h\mapsto F^\circ(x;h)$ is proper, convex, and lower semicontinuous.

We now define for a locally Lipschitz continuous functional $F:X\to \R$ the \emph{Clarke subdifferential} in $x\in X$ as
\begin{equation}\label{eq:clarke:def}
    \partial_C F(x) := \setof{x^*\in X^*}{\dual{x^*,h}_X \leq F^\circ (x;h)\quad\text{for all }h\in X}.
\end{equation}
The definition together with \cref{lem:clarke:dir}\,(i) directly implies the following properties.
\begin{cor}\label{lem:clarke:properties}
    Let $F:X\to\R$ be locally Lipschitz continuous and $x\in X$. Then $\partial_C F(x)$ is convex, weakly-$*$ closed, and bounded. Specifically, if $F$ is Lipschitz near $x$ with constant $L$, then $\partial _C F(x)\subset K_L(0)$.
\end{cor}
Furthermore, we have the following useful closedness property.
\begin{lemma}\label{lem:clarke:closed}
    Let $F:X\to\R$ be locally Lipschitz continuous in $x\in X$. If $\{x_n\}_{n\in\N}\subset X$ is a sequence with $x_n\to x$ and if $x_n^*\in \partial_C F(x_n)$ for all $n\in\N$ with $x_n^*\weakto^* x^*$ in $X^*$, then $x^*\in \partial_C F(x)$.
\end{lemma}
\begin{proof}
    Let $h\in X$ be arbitrary.
    By assumption, we then have that $\dual{x_n^*,h}_X \leq F^\circ(x_n;h)$ for all $n\in \N$.
    The weak-$*$ convergence of $\{x_n^*\}_{n\in\N}$ then implies that
    \begin{equation*}
        \dual{x^*,h}_X = \lim_{n\to\infty} \dual{x_n^*,h}_X \leq \limsup_{n\to\infty} F^\circ(x_n;h).
    \end{equation*}
    Hence we are finished if we can show that $\limsup_{n\to\infty} F^\circ(x_n;h)\leq F^\circ(x;h)$ (since then $x^*\in \partial_C F(x)$ by definition).

    For this, we use that by definition of $F^\circ(x_n;h)$, there exist sequences $\{y_{n,m}\}_{m\in\N}$ and $\{t_{n,m}\}_{m\in\N}$ with $y_{n,m}\to x_n$ and $t_{n,m}\to 0$ for $m\to \infty$ realizing the $\limsup$ for each $x_n$. Hence, for all $n\in \N$ we can find a $y_n:=y_{n,m(n)}$ and a $t_n:=t_{n,m(n)}$ such that $\norm{y_n-x_n}_X+t_n < n^{-1}$ (and hence in particular $y_n\to x$ and $t_n\to 0$) as well as
    \begin{equation*}
        F^\circ(x_n;h) - \tfrac1n \leq \frac{F(y_n+t_nh)-F(y_n)}{t_n}
    \end{equation*}
    for $n$ sufficiently large. Taking the $\limsup$ for $n\to\infty$ on both sides yields the desired inequality.
\end{proof}

Again, the construction immediately yields a Fermat principle.\footnote{Similarly to \cref{thm:convex:fermat}, we do not need to require Lipschitz continuity of $F$ -- the Fermat principle for the Clarke subdifferential characterizes (among others) \emph{any} local minimizer. However, if we want to use this principle to verify that a given $\bar x\in X$ is indeed a (candidate for) a minimizer, we need a suitable characterization of the subdifferential -- and this is only possible for (certain) locally Lipschitz continuous functionals.}
\begin{theorem}\label{thm:clarke:fermat}
    If $F:X\to\R$ has a local minimum in $\bar x$, then $0\in \partial_C F(\bar x)$.
\end{theorem}
\begin{proof}
    If $\bar x\in X$ is a local minimizer of $F$, then $F(\bar x) \leq F(\bar x + th)$ for all $h\in X$ and $t>0$ sufficiently small (since the topological interior is always included in the algebraic interior). But this implies that
    \begin{equation*}
        \dual{0,h}_X = 0  \leq \liminf_{t\to 0^+} \frac{F(\bar x+th)-F(\bar x)}{t}\leq \limsup_{t\to 0^+} \frac{F(\bar x+th)-F(\bar x)}{t}\leq F^\circ(x;h)
    \end{equation*}
    and hence $0\in \partial_C F(\bar x)$ by definition.
\end{proof}
Note that $F$ is not assumed to be convex, and hence the condition is in general not sufficient (consider, e.g., $f(t) = -|t|$).

Next, we show that the Clarke subdifferential is indeed a generalization of the derivative concepts we've studied so far.
\begin{theorem}\label{thm:clarke:frechet}
    Let $F:X\to\R$ be continuously Fréchet differentiable in a neighborhood $U$ of $x\in X$. Then, $\partial_C F(x) = \{F'(x)\}$.
\end{theorem}
\begin{proof}
    First we note that the assumption implies local Lipschitz continuity of $F$: Since $F':X\to X^*$ is continuous in $U$, there exists a $\delta>0$ with $\norm{F'(z)-F'(x)}_{X^*} \leq 1$ and hence $\norm{F'(z)}_{X^*} \leq 1 +\norm{F'(x)}_{X^*}$ for all $z\in K_\delta(x)\subset U$.
    For any $x_1,x_2\in K_\delta(x)$ we also have $x_2 + t(x_1-x_2)\in K_\delta(x)$ for all $t\in [0,1]$ (since balls in normed vector spaces are convex), and hence \cref{thm:frechet:mean} implies that
    \begin{equation*}
        \begin{aligned}
            |F(x_1)-F(x_2)| &\leq \int_0^1 \norm{F'(x_2 + t(x_1-x_2))}_{X^*}t\norm{x_1-x_2}_X \,dt  \\
            &\leq \frac{1 +\norm{F'(x)}_{X^*}}2 \norm{x_1-x_2}_X.
        \end{aligned}
    \end{equation*}

    We now show that $F^\circ(x;h)=F'(x)h$ for all $h\in X$. Take again sequences $\{x_n\}_{n\in\N}$ and $\{t_n\}_{n\in\N}$ with $x_n\to x$ and $t_n\to 0^+$ realizing the $\limsup$. Applying again the mean value \cref{thm:frechet:mean} and using the continuity of $F'$ yields for any $h\in X$ that
    \begin{equation*}
        \begin{aligned}
            F^\circ(x;h) &= \lim_{n\to\infty} \frac{F(x_n+t_nh)-F(x_n)}{t_n}\\
            &=\lim_{n\to\infty} \int_{0}^1 \frac1{t_n}\dual{F'(x_n+t(t_nh)),t_nh}_X\,dt\\
            &=\dual{F'(x),h}_X
        \end{aligned}
    \end{equation*}
    since the integrand converges uniformly in $t\in [0,1]$ to $\dual{F'(x),h}_X$.
    Hence by definition, $x^*\in \partial_C F(x)$ if and only if $\dual{x^*,h}_X \leq \dual{F'(x),h}_X$ for all $h\in X$, which is only possible for $x^*=F'(x)$.
\end{proof}
However, if $F$ is merely Fréchet differentiable, we only have that $F'(x)\in \partial_C F(x)$.

\begin{theorem}\label{thm:clarke:convex}
    Let $F:X\to\R$ be convex and lower semicontinuous. Then, $\partial_C F(x) = \partial F(x)$ for all $x\in X$.
\end{theorem}
\begin{proof}
    Since $F$ is finite-valued, $(\dom F)^o=X$, and hence $F$ is local Lipschitz continuous in every $x\in X$ by \cref{thm:convex:cont}.
    We now show that $F^\circ(x;h)=F'(x;h)$ for all $h\in X$, which together with the definition \eqref{eq:convex:subdiff_dir} of the convex subdifferential (which is equivalent to definition \eqref{eq:convex:def} by \cref{lem:convex:equiv}) yields the claim. First, we always have that
    \begin{equation*}
        F'(x;h) = \lim_{t\to 0^+} \frac{F(x+th)-F(x)}{t} \leq  \limsup_{\substack{y\to x\\t\to 0^+}} \frac{F(y+th)-F(y)}{t}
        = F^\circ(x;h).
    \end{equation*}
    To show the reverse inequality, let $\delta>0$ be arbitrary.
    Since the difference quotient of convex functionals is increasing by \cref{lem:convex:direct}\,(i), we obtain that
    \begin{equation*}
        \begin{aligned}
            F^\circ(x;h) &= \lim_{\eps\to 0^+} \sup_{y\in K_{\delta\eps}(x)}\sup_{0<t<\eps} \frac{F(y+th)-F(y)}{t}\\
            &\leq \lim_{\eps\to 0^+} \sup_{y\in K_{\delta\eps}(x)} \frac{F(y+\eps h)-F(y)}{\eps}\\
            &\leq \lim_{\eps\to 0^+} \frac{F(x+\eps h)-F(x)}{\eps} + 2L \delta\\
            &= F'(x;h) + 2L \delta,
        \end{aligned}
    \end{equation*}
    where the last inequality follows by adding two productive zeros and using the local Lipschitz continuity in $x$. Since $\delta>0$ was arbitrary, this implies that $F^\circ(x;h) \leq F'(x;h)$, and the claim follows.
\end{proof}
A locally Lipschitz continuous functional $F:X\to\R$ with $F^\circ(x;h)=F'(x;h)$ for all $h\in X$ is called \emph{regular} in $x\in X$. We have just shown that every continuously differentiable and every convex and lower semicontinuous functional is regular; intuitively, a function is thus regular in any points in which it is either differentiable or has at most a \enquote{convex kink}.

Finally, similarly to \cref{lem:lebesgue:subdiff} one can show the following pointwise characterization of the Clarke subdifferential of integral functionals with Lipschitz continuous integrands. We again assume that $\Omega\subset\R^d$ is open and bounded.
\begin{theorem}\label{thm:clarke:pointwise}
    Let $f:\R\to\R$ be Lipschitz continuous and $F:L^p(\Omega)\to\Rbar$ with $1\leq p <\infty$ as in \cref{lem:lebesgue:lsc}. Then we have for all $u\in L^p(\Omega)$ with $q=\frac{p}{p-1}$ (where $q=\infty$ for $p=1$) that
    \begin{equation*}
        \partial_C F(u)  \subset \setof{u^*\in L^q(\Omega)}{u^*(x) \in\partial_C f(u(x))\text{ for almost every } x\in\Omega}.
    \end{equation*}
    If $f$ is regular at $u(x)$ for almost every $x\in \Omega$, then $F$ is regular at $u$, and equality holds.
\end{theorem}
\begin{proof}
    First, by the properties of the Lebesgue integral and the Lipschitz continuity of $f$, we have for any $u,v\in L^p(\Omega)$ that
    \begin{equation*}
        |F(u)-F(v)| \leq \int_\Omega |f(u(x))-f(v(x))|\,dx \leq L \int_\Omega |u(x)-v(x)|\,dx \leq L C_p \norm{u-v}_{L^p},
    \end{equation*}
    where $L$ is the Lipschitz constant of $f$ and $C_p$ the constant from the continuous embedding $L^p(\Omega)\hookrightarrow L^1(\Omega)$ for $1\leq p\leq \infty$. Hence $F:L^p(\Omega)\to \R$ is Lipschitz continuous and therefore finite-valued as well.

    Let now $\xi\in \partial_C F(u)\subset L^p(\Omega)^*$ be given and $h\in L^p(\Omega)$ be arbitrary. By definition, we thus have
    \begin{equation}
        \label{eq:clarke:pointwise:integral-ineq}
        \begin{aligned}[t]
            \dual{\xi,h}_{L^p} \leq F^\circ(u; h) &=\limsup_{\substack{v\to u\\t\to 0}} \frac{F(v+th)-F(v)}{t}\\
            &\leq \int_\Omega \limsup_{\substack{v\to u\\t\to 0}} \frac{f(v(x)+th(x))-f(v(x))}{t} \,dx\\
            &\leq \int_\Omega \limsup_{\substack{v_x\to u(x)\\t_x\to 0}} \frac{f(v_x+t_xh(x))-f(v_x)}{t_x} \,dx\\
            &= \int_\Omega f^\circ(u(x); h(x))\,dx,
        \end{aligned}
    \end{equation}
    where we were able to use the Reverse Fatou Lemma to exchange the $\limsup$ with the integral in the first inequality since the integrand is bounded from above by the integrable function $L |h|$ due to \cref{lem:clarke:dir}\,(i); the second inequality follows by bounding for almost every $x\in \Omega$ the (pointwise) limit over the sequences realizing the $\limsup$ in the second line by the $\limsup$ over all admissible sequences.

    To interpret \eqref{eq:clarke:pointwise:integral-ineq} pointwise, we define for $x\in \Omega$
    \begin{equation*}
        g_x:\R\to\R,\qquad g_x(t) := f^\circ(u(x);t).
    \end{equation*}
    From \cref{lem:clarke:dir}\,(ii)--(iii), it follows that $g_x$ is convex; \cref{lem:clarke:dir}\,(i) further implies that the function $x\mapsto g_x(h(x))$ is measurable for any $h\in L^p(\Omega)$. Since $g_x(0) = 0$, \eqref{eq:clarke:pointwise:integral-ineq} implies that
    \begin{equation*}
        \dual{\xi,h-0}_{L^p} \leq \int_\Omega g_x(h(x))\,dx - \int_\Omega g_x(0)\,dx,
    \end{equation*}
    i.e., $\xi\in \partial G(0)$ for the superposition operator $G(h) :=\int_\Omega g_x(h(x))\,dx$. Arguing exactly as in the proof of \cref{lem:lebesgue:subdiff} (using that the spatially-varying(!) integrand $g_x(t)$ is measurable in $x$), this implies that $\xi = u^*\in L^q(\Omega)$ with $u^*(x)\in \partial g_x(0)$ for almost every $x\in \Omega$, i.e.,
    \begin{equation*}
        u^*(x)h(x) = u^*(x)(h(x)-0) \leq g_x(h(x)) - g_x(0) = f^\circ(u(x); h(x))
    \end{equation*}
    for almost every $x\in \Omega$.
    Since $h\in L^p(\Omega)$ was arbitrary, this implies that $u^*(x)\in \partial_C f(u(x))$ almost everywhere as claimed.

    It remains to show the remaining assertions when $f$ is regular. In this case, it follows from \eqref{eq:clarke:pointwise:integral-ineq} that for any $h\in L^p(\Omega)$,
    \begin{equation}
        \label{eq:clarke:pointwise:regular-est}
        \begin{aligned}[t]
            F^\circ(u; h) & \leq \int_\Omega f^\circ(u(x);h(x))\,dx = \int_\Omega f'(u(x); h(x))\,dx
            \\
            &
            \leq \lim_{t\to 0} \frac {F(u+th)-F(u)}{t} = F'(u; h) \leq F^\circ(u; h),
        \end{aligned}
    \end{equation}
    where the second inequality is obtained by applying Fatou's Lemma, this time appealing to the integrable lower bound $- L|h(x)|$. This shows that $F'(u; h)= F^\circ(u; h)$ and hence that $F$ is regular. We further obtain for any $u^* \in L^q(\Omega)$ with $u^*(x) \in \partial_C f(u(x))$ almost everywhere and any $h\in L^p(\Omega)$, that
    \begin{equation*}
        \dual{u^*,h}_{L^p} = \int_\Omega u^*(x) h(x) \, dx \leq \int_\Omega f^\circ(u(x); h(x))\,dx
        \leq F^\circ(u, h),
    \end{equation*}
    where we have used \eqref{eq:clarke:pointwise:regular-est} in the last inequality.
    Since $h\in L^p(\Omega)$ was arbitrary, this implies that $u^*\in \partial_C F(u)$.
\end{proof}
Under additional assumptions similar to those of \cref{thm:superpos:continuous} and with more technical arguments, this result can be extended to spatially varying integrands $f:\Omega\times \R\to \R$; see, e.g., \cite[Theorem 2.7.5]{Clarke:1990a}.

\section{Calculus rules}\label{sec:clarke:calculus}

We now turn to calculus rules. The first one still follows directly from the definition.
\begin{theorem}\label{thm:clarke:scalar}
    Let $F:X\to\R$ be locally Lipschitz continuous in $x\in X$ and $\alpha\in\R$. Then,
    \begin{equation*}
        \partial_C(\alpha F)(x) = \alpha \partial_C(F)(x).
    \end{equation*}
\end{theorem}
\begin{proof}
    First, $\alpha F$ is clearly locally Lipschitz continuous for any $\alpha\in\R$. If $\alpha = 0$, both sides of the claimed equality are zero (which is easiest seen from \cref{thm:clarke:frechet}).
    If $\alpha > 0$, we have that $(\alpha F)^\circ(x;h) = \alpha F^\circ (x;h)$ for all $h\in X$ from the definition. Hence,
    \begin{equation*}
        \begin{aligned}
            \alpha \partial_C F(x)
            &= \setof{\alpha x^*\in X^*}{\dual{x^*,h}_X \leq F^\circ(x;h)\quad\text{for all }h\in X}\\
            &= \setof{\alpha x^*\in X^*}{\dual{\alpha x^*,h}_X \leq \alpha F^\circ(x;h)\quad\text{for all }h\in X}\\
            &= \setof{y^*\in X^*}{\dual{y^*,h}_X \leq (\alpha F)^\circ(x;h)\quad\text{for all }h\in X}\\
            &= \partial_C(\alpha F)(x).
        \end{aligned}
    \end{equation*}
    To conclude the proof, it suffices to show the claim for $\alpha=-1$.
    For that, we use \cref{lem:clarke:dir}\,(iv) to obtain that
    \begin{equation*}
        \begin{split}
            \begin{aligned}[b]
                \partial_C (-F)(x)
                &= \setof{x^*\in X^*}{\dual{x^*,h}_X \leq (-F)^\circ(x;h)\quad\text{for all }h\in X}\\
                &= \setof{x^*\in X^*}{\dual{-x^*,-h}_X \leq F^\circ(x;-h)\quad\text{for all }h\in X}\\
                &= \setof{-y^*\in X^*}{\dual{y^*,g}_X \leq F^\circ(x;g)\quad\text{for all }g\in X}\\
                &= -\partial_C(F)(x).
            \end{aligned}
            \qedhere
        \end{split}
    \end{equation*}
\end{proof}
\begin{cor}\label{lem:clarke:fermat2}
    Let $F:X\to\R$ be locally Lipschitz continuous in $\bar x\in X$. If $F$ has a local maximum in $\bar x$, then $0\in \partial_C F(\bar x)$.
\end{cor}
\begin{proof}
    If $\bar x$ is a local maximizer of $F$, it is a local minimizer of $-F$. Hence, \cref{thm:clarke:fermat} implies that
    \begin{equation*}
        0\in \partial_C(-F)(\bar x) = -\partial_C F(\bar x),
    \end{equation*}
    i.e., $0=-0\in \partial_C F(\bar x)$.
\end{proof}

\subsection*{Support functionals}

The remaining rules are again significantly more involved. As in the previous proofs, a key step is to relate different sets of the form \eqref{eq:clarke:def}, for which the following lemmas will be helpful.
\begin{lemma}\label{lem:clarke:support1}
    Let $S:X\to\R$ be positively homogeneous, subadditive, and lower semicontinuous, and let
    \begin{equation*}
        A=\setof{x^*\in X^*}{\dual{x^*,x}_X\leq S(x)\quad\text{for all }x\in X}.
    \end{equation*}
    Then
    \begin{equation}\label{eq:clarke:support1}
        S(x) =  \sup_{x^*\in A} \dual{x^*,x}_X\qquad\text{for all }x\in X.
    \end{equation}
\end{lemma}
\begin{proof}
    By definition of $A$, the inequality $\dual{x^*,x}_X-S(x)\leq 0$ holds for all $x\in X$ if and only if $x^*\in A$. Thus a case distinction as in \cref{ex:convex:fenchel}\,(iii) using the positive homogeneity of $S$ (which in particular implies that $S(0)=0$) shows that
    \begin{equation*}
        S^*(x^*) = \sup_{x\in X}\, \dual{x^*,x}_X - S(x) = \begin{cases} 0 & x^* \in A,\\\infty & x^*\notin A,\end{cases}
    \end{equation*}
    i.e., $S^*=\delta_A$.
    Furthermore, by assumption $S$ is also subadditive and hence convex as well as lower semicontinuous.
    \Cref{thm:convex:moreau} thus yields
    \begin{equation*}
        \begin{split}
            S(x) = S^{**}(x) = (\delta_A)^*(x) = \sup_{x^*\in A} \dual{x^*,x}_X.
            \qedhere
        \end{split}
    \end{equation*}
\end{proof}
The right-hand side of \eqref{eq:clarke:support1} is called the \emph{support functional} of $A$.

\begin{lemma}\label{lem:clarke:support3}
    Let $A,B\subset X^*$ be nonempty, convex, and weakly-$*$ closed. Then $A\subset B$ if and only if
    \begin{equation}\label{eq:clarke:support3}
        \sup_{x^*\in A}\,  \dual{x^*,x}_X \leq \sup_{x^*\in B}\, \dual{x^*,x}_X \qquad\text{for all }x\in X.
    \end{equation}
\end{lemma}
\begin{proof}
    If $A\subset B$, then the right-hand side of \eqref{eq:clarke:support3} is obviously not less than the left-hand side.
    Conversely, assume that there exists an $x^*\in A$ with $x^*\notin B$.
    By the assumptions on $A$ and $B$, we then obtain from  \cref{thm:clarke:hb} an $x\in X$ and a $\lambda\in \R$ with
    \begin{equation*}\label{eq:clarke:hb1}
        \dual{z^*,x}_X \leq \lambda < \dual{x^*,x}_X  \qquad\text{for all }z^*\in B.
    \end{equation*}
    Taking the supremum over all $z^*\in B$ and estimating the right-hand side by the supremum over all $x^*\in A$ then yields that
    \begin{equation*}\label{eq:clarke:hb2}
        \sup_{z^*\in B}\, \dual{z^*,x}_X < \sup_{x^*\in A}\,\dual{x^*,x}_X.
    \end{equation*}
    Hence \eqref{eq:clarke:support3} is violated, and the claim follows by contraposition.
\end{proof}
\begin{cor}\label{lem:clarke:support2}
    Let $A,B\subset X^*$ be nonempty, convex, and weakly-$*$ closed. Then $A=B$ if and only if
    \begin{equation}\label{eq:clarke:support2}
        \sup_{x^*\in A}\, \dual{x^*,x}_X = \sup_{x^*\in B}\, \dual{x^*,x}_X \qquad\text{for all }x\in X.
    \end{equation}
\end{cor}
\begin{proof}
    Again, the claim is obvious if $A=B$. Conversely, if \eqref{eq:clarke:support2} holds, then in particular \eqref{eq:clarke:support3} holds, and we obtain from \cref{lem:clarke:support3} that $A\subset B$. Exchanging the roles of $A$ and $B$ now yields the claim.
\end{proof}

\bigskip

\Cref{lem:clarke:support1} together with \cref{lem:clarke:dir} directly yields the following useful representation.
\begin{cor}\label{cor:clarke:support-dir}
    Let $F:X\to\R$ be locally Lipschitz continuous and $x\in X$. Then
    \begin{equation*}
        F^\circ(x;h) = \sup_{x^*\in\partial_C F(x)}\dual{x^*,h}_X\quad\text{for all }h\in X.
    \end{equation*}
\end{cor}
With its help, we can finally show the promised nonemptiness of the convex subdifferential.
\begin{cor}\label{cor:convex:nonempty}
    Let $F:X\to\Rbar$ be proper, convex, and lower semicontinuous, and $x\in(\dom F)^o$. Then, $\partial F(x)$ is nonempty, convex, weakly-$*$ closed, and bounded.
\end{cor}
\begin{proof}
    Since $x\in (\dom F)^o$, \cref{thm:clarke:convex} shows that $\partial_C F(x) = \partial F(x)$ and that $F$ is regular in $x$. It thus follows from \cref{cor:clarke:support-dir} and \cref{lem:convex:direct}\,(iii) that $\sup_{x^*\in\partial F(x)} \dual{x^*,h}_X = F'(x;h)\in \R$ for $x\in(\dom F)^o$, and hence the supremum cannot be over the empty set (for which any supremum is $-\infty$ by convention). The remaining properties follow from \cref{lem:clarke:properties}.
\end{proof}

\subsection*{Sum rule}

We now use these results to prove a sum rule.
\begin{theorem}\label{thm:clarke:sum}
    Let $F,G:X\to\R$ be locally Lipschitz continuous in $x\in X$. Then
    \begin{equation*}
        \partial_C (F+G)(x) \subset \partial_C F(x) + \partial_C G(x).
    \end{equation*}
    If $F$ and $G$ are regular in $x$, then $F+G$ is regular in $x$ and equality holds.
\end{theorem}
\begin{proof}
    It is clear that $F+G$ is locally Lipschitz continuous in $x$. Furthermore, from the properties of the $\limsup$ we always have for all $h\in X$ that
    \begin{equation*}
        (F+G)^\circ(x;h) \leq F^\circ (x;h) + G^\circ (x;h).
    \end{equation*}
    If $F$ and $G$ are regular in $x$, the calculus of limits yields that
    \begin{equation*}
        F^\circ (x;h) + G^\circ (x;h) = F'(x;h) + G'(x;h) = (F+G)'(x;h) \leq (F+G)^\circ(x;h),
    \end{equation*}
    which implies that $(F+G)^\circ(x;h) = (F+G)'(x;h)$, i.e., $F+G$ is regular.

    By \cref{lem:clarke:support3} we are thus finished if we can show that
    \begin{equation*}
        \partial_C F(x) + \partial_C G(x) = \setof{x^*\in X^*}{\dual{x^*,h}_X \leq F^\circ (x;h) + G^\circ (x;h)\text{ for all }h\in X}=:A.
    \end{equation*}
    For this, we use that $\partial_C F(x)$ and $\partial_C G(x)$ are convex and weakly-$*$ closed by \cref{lem:clarke:properties}, and hence so is their sum since both sets are bounded.
    Furthermore, as shown in \cref{lem:clarke:dir}, generalized directional derivatives and hence their sums are positively homogeneous, convex, and lower semicontinuous. We thus obtain from \cref{lem:clarke:support1} for all $h\in X$ that
    \begin{equation*}
        \begin{aligned}[b]
            \sup_{x^* \in \partial_C F(x) + \partial_C G(x)}\dual{x^*,h}_X
            &= \sup_{x_1^*\in \partial_C F(x)} \dual{x_1^*,h}_X + \sup_{x_2^*\in \partial_C G(x)} \dual{x_2^*,h}_X \\
            &= F^\circ(x;h) + G^\circ(x;h)
            = \sup_{x^*\in A}\,\dual{x^*,h}_X.
        \end{aligned}
    \end{equation*}
    The claimed equality of $A$ and the sum of the subdifferentials now follows from \cref{lem:clarke:support2}.
\end{proof}
Note the differences to the convex sum rule: The generic inclusion is now in the other direction; furthermore, \emph{both} functionals have to be regular, and in exactly the point where the sum rule is applied.
By induction, one obtains from this sum rule for an arbitrary number of functionals (which all have to be regular).

\subsection*{Chain rule}

To prove a chain rule, we need the following \enquote{nonsmooth} mean value theorem due to Lebourg.
\begin{theorem}\label{thm:clarke:mean}
    Let $F:X\to\R$ be locally Lipschitz continuous near $x\in X$ and $\tilde x$ be in the Lipschitz neighborhood of $x$. Then there exists an $x^*\in \partial_C F(x+\lambda(\tilde x-x))$ for some $\lambda\in (0,1)$ such that
    \begin{equation*}
        F(\tilde x)-F(x) = \dual{x^*,\tilde x-x}_X.
    \end{equation*}
\end{theorem}
\begin{proof}
    Define $\psi,\phi:[0,1]\to\R$ as
    \begin{equation*}
        \psi(\lambda):= F(x+\lambda(\tilde x-x)),\qquad \phi(\lambda):=\psi(\lambda)+\lambda(F(x)-F(\tilde x)).
    \end{equation*}
    By the assumptions on $F$ and $\tilde x$, both $\psi$ and $\phi$ are Lipschitz continuous. In addition, $\phi(0) = F(x) = \phi(1)$, and hence $\phi$ has a local minimum or maximum in an interior point $\bar \lambda\in (0,1)$. From the Fermat principle \cref{thm:clarke:fermat} or \cref{lem:clarke:fermat2}, respectively, together with the sum rule from \cref{thm:clarke:sum} and the characterization of the subdifferential of the second term from \cref{thm:clarke:frechet}, we thus obtain that
    \begin{equation*}
        0\in \partial_C \phi(\bar \lambda) \subset \partial_C\psi(\bar \lambda) + \{F(x)-F(\tilde x)\}.
    \end{equation*}
    Hence we are finished if we can show for $x_{\bar \lambda} := x +\bar \lambda(\tilde x-x)$ that
    \begin{equation}\label{eq:clarke:mean1}
        \partial_C\psi(\bar \lambda) \subset\setof{\dual{x^*,\tilde x-x}_X}{x^*\in\partial_C F(x_{\bar\lambda})} =:A.
    \end{equation}
    For this purpose, consider for arbitrary $s\in \R$ the generalized directional derivative
    \begin{equation*}
        \begin{aligned}
            \psi^\circ(\bar\lambda;s) &= \limsup_{\substack{\lambda\to \bar\lambda\\t\to 0}} \frac{\psi(\lambda + ts)-\psi(\lambda)}{t}\\
            &=\limsup_{\substack{\lambda\to \bar\lambda\\t\to 0}} \frac{F(x+(\lambda + ts)(\tilde x-x))-F(x+\lambda(\tilde x-x))}{t}\\
            &\leq \limsup_{\substack{z\to x_{\bar\lambda}\\t\to 0}} \frac{F(z +ts(\tilde x-x))-F(z)}{t} = F^\circ(x_{\bar\lambda};s(\tilde x-x)),
        \end{aligned}
    \end{equation*}
    where the inequality follows from considering arbitrary sequences $z\to x_{\bar\lambda}$ (instead of special sequences of the form $z_n = x+\lambda_n(\tilde x-x)$) in the last $\limsup$. \Cref{lem:clarke:support3} thus implies that
    \begin{equation}\label{eq:clarke:mean2}
        \partial_C\psi(\bar \lambda) \subset\setof{t^*\in\R}{t^*s \leq F^\circ(x_{\bar\lambda};s(\tilde x-x))\text{ for all }s\in \R} =: B.
    \end{equation}
    It remains to show that the sets $A$ and $B$ from \eqref{eq:clarke:mean1} and \eqref{eq:clarke:mean2} coincide. But this follows again from \cref{lem:clarke:support1,lem:clarke:support2}, since for all $s\in\R$ we have that
    \begin{equation*}
        \sup_{t^*\in A}\,t^*s = \sup_{x^*\in \partial_C F(x_{\bar\lambda})} \dual{x^*,s(\tilde x-x)}_X = F^\circ (x_{\bar\lambda};s(\tilde x-x)) = \sup_{t^*\in B}\,t^*s.
        \qedhere
    \end{equation*}
\end{proof}
We also need the following generalization of the argument in \cref{thm:clarke:frechet}.
\begin{lemma}\label{lem:frechet:diffquot}
    Let $X,Y$ be Banach spaces and $F:X\to Y$ be continuously Fréchet differentiable at $x\in X$. Let $\{x_n\}_{n\in\N}\subset X$ be a sequence with $x_n\to x$ and $\{t_n\}_{n\in\N}\subset(0,\infty)$ be a sequence with $t_n\to 0$. Then for any $h\in X$,
    \begin{equation*}
        \lim_{n\to\infty} \frac{F(x_n+t_n h)-F(x_n)}{t_n} = F'(x)h.
    \end{equation*}
\end{lemma}
\begin{proof}
    Let $h\in X$ be arbitrary. By the Hahn--Banach extension \cref{thm:hb_extension}, for every $n\in \N$ there exists a $y_n^*\in Y^*$ with $\norm{y_n^*}_{Y^*}=1$ and
    \begin{equation*}
        \norm{t_n^{-1}(F(x_n+t_nh)-F(x_n)) - F'(x)h}_Y = \dual{y_n^*,t_n^{-1}(F(x_n+t_nh)-F(x_n)) - F'(x)h}_Y.
    \end{equation*}
    Applying now the classical mean value theorem to the scalar functions
    \begin{equation*}
        f_n :[0,1]\to\R,\qquad f_n(s) = \dual{y_n^*,F(x_n+st_nh)}_Y,
    \end{equation*}
    we obtain similarly to the proof of \cref{thm:frechet:mean} for all $n\in\N$ that
    \begin{equation*}
        \begin{aligned}
            \norm{t_n^{-1}(F(x_n+t_nh)-F(x_n)) - F'(x)h}_Y &= t_n^{-1}\int_0^1\dual{y_n^*,F'(x_n+st_nh)t_nh}_Y\,ds - \dual{y_n^*,F'(x)h}_Y\\
            &= \int_0^1 \dual{y_n^*,(F'(x_n+st_nh)-F'(x))h}_Y\,ds\\
            &\leq \int_0^1 \norm{F'(x_n+st_nh)-F'(x))}_{L(X; Y)}\,ds \,\norm{h}_X,
        \end{aligned}
    \end{equation*}
    where we have used \eqref{eq:functan:cs_banach} together with $\norm{y_n^*}_{Y^*}=1$ in the last step. Since $F'$ is continuous by assumption, the integrand goes to zero as $n\to\infty$ uniformly in $s\in[0,1]$, and the claim follows.
\end{proof}

We now come to the chain rule, which in contrast to the convex case does not require the inner mapping to be linear; this is one of the main advantages of the Clarke subdifferential in the context of nonsmooth optimization.
\begin{theorem}\label{thm:clarke:chain}
    Let $Y$ be a separable Banach space, $F:X\to Y$ be continuously Fréchet differentiable at $x\in X$, and $G:Y\to \R$ be locally Lipschitz continuous near $F(x)$. Then
    \begin{equation*}
        \partial_C (G\circ F)(x) \subset F'(x)^*\partial_C G(F(x)) := \setof{F'(x)^*y^*}{y^*\in\partial_C G(F(x))}.
    \end{equation*}
    If $G$ is regular at $F(x)$, then $G\circ F$ is regular at $x$, and equality holds.
\end{theorem}
\begin{proof}
    The local Lipschitz continuity of $G\circ F$ follows from that of $G$ and $F$ (which follows from the assumption as in the proof of \cref{thm:clarke:frechet}). For the claimed inclusion (or equality), we argue as before using the support functional calculus.
    First we show that for every $h\in X$ there exists a $y^*\in \partial_C G(F(x))$ with
    \begin{equation}\label{eq:clarke:chain1}
        (G\circ F)^\circ(x;h) = \dual{y^*,F'(x)h}_Y.
    \end{equation}
    To this end, consider for given $h\in X$ sequences $\{x_n\}_{n\in\N}\subset X$ and $\{t_n\}_{n\in\N}\subset (0,\infty)$ with $x_n\to x$, $t_n \to 0$, and
    \begin{equation*}
        (G\circ F)^\circ(x;h) = \lim_{n\to\infty} \frac{G(F(x_n+t_nh))-G(F(x_n))}{t_n}.
    \end{equation*}
    Furthermore, by continuity of $F$, we can find $n_0\in\N$ such that $F(x_n), F(x_n+t_nh)$ lie in the Lipschitz neighborhood of $F(x)$ for all $n\geq n_0$. \Cref{thm:clarke:mean} thus yields for all $n\geq n_0$ a $y_n^*\in \partial_C G(y_n)$ with $y_n: = F(x_n)+\lambda_n(F(x_n+t_nh)-F(x_n))$ for some $\lambda_n\in(0,1)$ such that
    \begin{equation}\label{eq:clarke:chain1a}
        \frac{G(F(x_n+t_nh))-G(F(x_n))}{t_n} = \dual{y_n^*,q_n}_Y
        \quad\text{with}\quad
        q_n := \frac{F(x_n+t_nh)-F(x_n)}{t_n}
    \end{equation}
    Since $\lambda_n\in(0,1)$ is uniformly bounded, we also have that $y_n\to F(x)$ for $n\to \infty$.
    Hence $y_n$ is in the Lipschitz neighborhood of $F(x)$ for $n\in\N$ large enough,
    and \cref{lem:clarke:properties} yields that $y_n^*\in \partial_C G(y_n) \subset K_L(0)$ for $n\in \N$ sufficiently large.
    This implies that $\{y_n^*\}_{n\in\N}\subset Y^*$ is bounded, and the
    \nameref{thm:banachal} \cref{thm:banachal} yields a weakly-$*$ convergent subsequence with limit $y^*\in \partial_C G(F(x))$ by \cref{lem:clarke:closed}.
    Finally, since $F$ is continuously Fréchet differentiable, $q_n \to F'(x)h$ strongly in $Y$ by \cref{lem:frechet:diffquot}. Hence, $\dual{y_n^*,q_n}_Y \to \dual{y^*,F'(x)h}$ as the duality pairing of weakly-$*$ and strongly converging sequences. Passing to the limit in \eqref{eq:clarke:chain1a} therefore yields \eqref{eq:clarke:chain1} (first along the subsequence chosen above; by convergence of the left-hand side of \eqref{eq:clarke:chain1a} and the uniqueness of limits then for the full sequence as well). By definition of the Clarke subdifferential, we thus have for $y^*\in\partial_C G(F(x))$ that
    \begin{equation}\label{eq:clarke:chain2}
        (G\circ F)^\circ(x;h) = \dual{y^*,F'(x)h}_Y\leq G^\circ(F(x);F'(x)h).
    \end{equation}

    If $G$ is now regular at $x$, we have that $G^\circ(F(x);F'(x)h) = G'(F(x);F'(x)h)$ and hence by the local Lipschitz continuity of $G$ and the Fréchet differentiability of $F$ that
    \begin{multline*}
        G^\circ(F(x);F'(x)h) \\
        \begin{aligned}[t]
            &= \lim_{t\to 0} \frac{G(F(x)+tF'(x)h) - G(F(x))}{t}\\
            &= \lim_{t\to 0} \frac{G(F(x)+tF'(x)h) -G(F(x+th))+G(F(x+th))- G(F(x))}{t}\\
            &\leq \lim_{t\to 0} \left(L\norm{h}_X\frac{\norm{F(x)+F'(x)th - F(x+th)}_Y}{\norm{th}_X} + \frac{G(F(x+th))- G(F(x))}{t} \right)\\
            &= (G\circ F)'(x;h) \leq (G\circ F)^\circ(x;h).
        \end{aligned}
    \end{multline*}
    Together with \eqref{eq:clarke:chain2}, this implies that $(G\circ F)'(x;h)  = (G\circ F)^\circ(x;h)$ (i.e., $G\circ F$ is regular at $x$) and that
    \begin{equation}
        \label{eq:clarke:chain3}
        (G\circ F)^\circ(x;h) = G^\circ(F(x);F'(x)h).
    \end{equation}

    As before, \cref{lem:clarke:support1} now implies for all $h\in X$ that
    \begin{equation*}
        \sup_{x^*\in F'(x)^*\partial_CG(F(x))}\dual{x^*,h}_X = \sup_{y^*\in \partial_C G(F(x))}\dual{y^*,F'(x)h}_Y = G^\circ (F(x); F'(x)h)
    \end{equation*}
    and hence by \cref{lem:clarke:support3} that
    \begin{equation*}
        F'(x)^*\partial_C G(F(x)) = \setof{x^*\in X^*}{\dual{x^*,h}_X\leq G^\circ(F(x);F'(x)h)\text{ for all }h\in X}.
    \end{equation*}
    Combined with \eqref{eq:clarke:chain2} or \eqref{eq:clarke:chain3} and the definition of the Clarke subdifferential in \eqref{eq:clarke:def}, this now yields the claimed inclusion or equality, respectively, for the Clarke subdifferential of the composition.
\end{proof}
Again, the generic inclusion is the reverse of the one in the convex chain rule.
Note that equality in the chain rule also holds if $-G$ is regular, since we can then apply \cref{thm:clarke:chain} to $-G\circ F$ and use that $\partial_C (-G)(F(x)) = -\partial_C G(F(x))$ by \cref{thm:clarke:scalar}.
Furthermore, if $G$ is not regular but $F'(x)$ is surjective, a similar proof shows that equality (but not the regularity of $G\circ F$) holds in the chain rule; see \cite[Theorem 10.19]{Clarke:2013}.

\section{Characterization in finite dimensions}\label{sec:clarke:finitedim}

A more explicit characterization of the Clarke subdifferential is possible in finite-dimensional spaces. The basis is the following theorem, which only holds in $\R^N$; a proof can be found in, e.g., \cite[Theorem 23.2]{DiBenedetto} or \cite[Theorem 3.1]{Heinonen}.
\begin{theorem}[Rademacher]\label{thm:rademacher}
    Let $U\subset \R^N$ be open and $F:U\to\R$ be Lipschitz continuous. Then $F$ is Fréchet differentiable in almost every $x\in U$.
\end{theorem}
This result allows replacing the $\limsup$ in the definition of the Clarke subdifferential (now considered as a subset of $\R^N$, i.e., identifying the dual of $\R^N$ with $\R^N$ itself) with a proper limit.
\begin{theorem}\label{thm:clarke:gradient}
    Let $F:\R^N\to\R$ be locally Lipschitz continuous in $x\in\R^N$ and Fréchet differentiable on $\R^N\setminus E_F$ for a set $E_F\subset \R^N$ of Lebesgue measure $0$. Then
    \begin{equation}\label{eq:clarke:gradient}
        \partial_C F(x) = \co \setof{\lim_{n\to\infty} \nabla F(x_n)}{x_n\to x,\ x_n\notin E_F},
    \end{equation}
    where $\co A$ denotes the convex hull of $A\subset \R^N$.
\end{theorem}
\begin{proof}
    We first note that the Rademacher Theorem ensures that such a set $E_F$ exists and has Lebesgue measure $0$. Hence there indeed exist sequences $\{x_n\}_{n\in\N}\in\R^N\setminus E_F$ with $x_n\to x$.
    Furthermore, the local Lipschitz continuity of $F$ yields that for any $x_n$ in the Lipschitz neighborhood of $x$ and any $h\in \R^N$, we have that
    \begin{equation*}
        |\inner{\nabla F(x_n),h}| = \left|\lim_{t\to 0^+}\frac{F(x_n+th)-F(x_n)}{t}\right| \leq L\norm{h}
    \end{equation*}
    and hence that $\norm{\nabla F(x_n)} \leq L$. This implies that $\{\nabla F(x_n)\}_{n\in\N}$ is bounded and thus contains a convergent subsequence. The set on the right-hand side of \eqref{eq:clarke:gradient} is therefore nonempty.

    Let now $\{x_n\}_{n\in\N}\subset \R^N\setminus E_F$ be an arbitrary sequence with $x_n\to x$ and $\{\nabla F(x_n)\}_{n\in\N}\to x^*$ for some $x^*\in\R^N$.
    Since $F$ is differentiable in every $x_n\notin E_F$, we have that
    \begin{equation*}
        \dual{\nabla F(x_n),h} = F'(x; h) \leq F^\circ(x; h)
    \end{equation*}
    and hence that $\nabla F(x_n) \in \partial_C F(x_n)$ by definition. \Cref{lem:clarke:closed} thus yields that $x^*\in \partial_C F(x)$.
    The convexity of $\partial_C F(x)$ from \cref{lem:clarke:properties} now implies that any convex combination of such limits $x^*$ is contained in $\partial_C F(x)$, which shows the inclusion ``$\supset$'' in \eqref{eq:clarke:gradient}.

    For the other inclusion, we first show for all $h\in \R^N$ and $\eps>0$ that
    \begin{equation}\label{eq:clarke:gradient1a}
        F^\circ(x;h) - \eps \leq \limsup_{E_F\not\ni y\to x}\, \inner{\nabla F(y),h}=:M(h).
    \end{equation}
    Indeed, by definition of $M(h)$ and of the $\limsup$, for every $\eps>0$ there exists a $\delta>0$ such that
    \begin{equation*}
        \inner{\nabla F(y),h} \leq  M(h)+\eps \qquad\text{for all }y\in O_\delta(x)\setminus E_F.
    \end{equation*}
    Here, $\delta>0$ can be chosen sufficiently small for $F$ to be Lipschitz continuous on $O_\delta(x)$. In particular, $E_F\cap O_\delta(x)$ is a set of zero measure. Hence, $F$ is differentiable in $y+th$ for almost all $y\in O_{\delta/2}(x)$ and almost all $t\in (0,\frac\delta{2\norm{h}})$ by Fubini's Theorem. The classical mean value theorem therefore yields for all such $y$ and $t$ that
    \begin{equation}\label{eq:clarke:gradient1}
        F(y+th)-F(y) = \int_0^t \inner{\nabla F(y+sh),h}\,ds \leq t(M(h)+\eps)
    \end{equation}
    since $y+sh\in O_{\delta}(x)$ for all $s\in (0,t)$ by the choice of $t$.
    The continuity of $F$ implies that the full inequality \eqref{eq:clarke:gradient1} even holds for \emph{all} $y\in O_{\delta/2}(x)$ and \emph{all} $t\in (0,\frac\delta{2\norm{h}})$. Dividing by $t>0$ and taking the $\limsup$ over all $y\to x$ and $t\to 0$ now yields \eqref{eq:clarke:gradient1a}. Since $\eps>0$ was arbitrary, we conclude that $F^\circ(x;h)\leq M(h)$ for all $h\in \R^N$.

    As in \cref{lem:clarke:dir}, one can show that the mapping $h\mapsto M(h)$ is positively homogeneous, subadditive, and lower semicontinuous. We are thus finished if we can show that the set on the right-hand side of \eqref{eq:clarke:gradient} -- hereafter denoted by $\co A$ --  can be written as
    \begin{equation*}
        \co A= \setof{x^*\in\R^N}{\inner{x^*,h} \leq M(h)\quad\text{for all }h\in\R^N}.
    \end{equation*}
    For this, we once again appeal to \cref{lem:clarke:support2} (since both sets are closed and convex). First, we note that the definition of the convex hull implies for all $h\in\R^N$ that
    \begin{equation*}
        \sup_{x^*\in \co A} \inner{x^*,h}
        = \sup_{\substack{x_i^*\in A\\\sum_i t_i = 1,t_i\geq 0}} \sum_i t_i\inner{x_i^*,h} = \sup_{\sum_i t_i = 1,t_i\geq 0} \sum_i t_i \sup_{x_i^*\in A}\inner{x_i^*,h} = \sup_{x^*\in A}\inner{x^*,h}
    \end{equation*}
    since the sum is maximal if and only if each summand is maximal. Now we have that
    \begin{equation*}
        M(h) = \limsup_{E_F\not\ni y\to x} \,\inner{\nabla F(y),h} = \sup_{E_F \not\ni x_n\to x}\inner{\lim\nolimits_{n\to\infty} \nabla F(x_n),h} = \sup_{x^*\in A}\inner{x^*,h},
    \end{equation*}
    and hence the claim follows from \cref{lem:clarke:support1}.
\end{proof}

\chapter{Semismooth Newton methods}\label{chap:semismooth}

The proximal point and splitting methods in \cref{chap:proximal} are generalizations of gradient methods and in general have the same only linear convergence. In this chapter, we will therefore consider a generalization of Newton methods which admit (locally) superlinear convergence.

\section{Convergence of generalized Newton methods}

As a motivation, we first consider the most general form of a Newton-type method. Let $X$ and $Y$ be Banach spaces and $F:X\to Y$ be given and suppose we are looking for an $\bar x \in X$ with $F(\bar x)=0$. A Newton-type method to find such an $\bar x$ then consists of repeating the following steps:
\begin{enumerate}
    \item choose an invertible $M_k := M(x^k)\in L(X,Y)$;
    \item solve the \emph{Newton step} $M_k s^k = -F(x^k)$;
    \item update $x^{k+1} = x^k + s^k$.
\end{enumerate}
We can now ask under which conditions this method converges to $\bar x$, and in particular, when the convergence is \emph{superlinear}, i.e.,
\begin{equation}\label{eq:newton:superlinear}
    \lim_{k\to\infty} \frac{\norm{x^{k+1}-\bar x}_X}{\norm{x^{k}-\bar x}_X} = 0.
\end{equation}
For this purpose, we set $e^k := x^k-\bar x$ and use the Newton step together with the fact that $F(\bar x)=0$ to obtain that
\begin{equation*}
    \begin{aligned}[t]
        \norm{x^{k+1}-\bar x}_X &=  \norm{x^k - M(x^k)^{-1}F(x^{k}) - \bar x}_X\\
        &= \norm{M(x^k)^{-1}[F(x^k) - F(\bar x) - M(x^k)(x^k-\bar x)]}_X\\
        &= \norm{M(\bar x+e^k)^{-1}[F(\bar x+e^k) - F(\bar x) - M(\bar x+e^k) e^k]}_X\\
        &\leq \norm{M(\bar x+e^k)^{-1}}_{L(Y,X)}\norm{F(\bar x+e^k) - F(\bar x) - M(\bar x+e^k)e^k}_Y.
    \end{aligned}
\end{equation*}
Hence, \eqref{eq:newton:superlinear} holds under
\begin{enumerate}[(i)]
    \item a \emph{regularity condition}: there exists a $C>0$ with
        \begin{equation*}
            \norm{M(x^k)^{-1}}_{L(Y,X)} \leq C \qquad\text{for all }k\in\N;
        \end{equation*}
    \item an \emph{approximation condition}:
        \begin{equation*}
            \lim_{k\to\infty} \frac{\norm{F(\bar x+ e^k) - F(\bar x) - M(\bar x +e^k) e^k}_Y}{\norm{e^k}_X}= 0.
        \end{equation*}
\end{enumerate}

This motivates the following definition: We call $F:X\to Y$ \emph{Newton differentiable} in $x\in X$ if there exists a neighborhood $U\subset X$ of $x$ and a mapping $D_N F: U \to L(X;Y)$ such that
\begin{equation}\label{eq:newton:semismooth}
    \lim_{\norm{h}_X\to 0} \frac{\norm{F(x+ h) - F(x) - D_N F(x +h)h }_Y}{\norm{h}_X}= 0.
\end{equation}
We then call $D_NF(x)$ a \emph{Newton derivative} of $F$ at $x$.
Note the differences to the Fréchet derivative: First, the Newton derivative is evaluated in $x+h$ instead of $x$. More importantly, we have not required \emph{any} connection between $D_N F$ with $F$, while the only possible candidate for the Fréchet derivative was the Gâteaux derivative (which itself was linked to $F$ via the directional derivative). A function thus can only be Newton differentiable (or not) with respect to a concrete choice of $D_N F$. In particular, Newton derivatives are not unique.\footnote{Here we follow \cite{Chen:2000a,Kunisch:2008a,Schiela:2008a} and only consider single-valued Newton derivatives (called \emph{slanting functions} in the first-named work).  Alternatively, one could fix for each $x\in X$ a set $\partial_N F(x)$, from which the linear operator $M(x)$ in the Newton step has to be taken. If the approximation condition together with a boundedness condition hold \emph{uniformly} for all $M\in \partial_N F(x)$, the function $F$ is called \emph{semismooth} (explaining the title of this chapter). This approach is followed in, e.g., \cite{Mifflin:1977,Kummer:1988,Ulbrich:2011}.}

If $F$ is Newton differentiable with Newton derivative $D_N F$, we can set $M(x^k) = D_N F(x^k)$ and obtain the \emph{semismooth Newton method}
\begin{equation}\label{eq:SSN}
    x^{k+1} = x^k - D_N F(x^k)^{-1} F(x^k).
\end{equation}
Its local superlinear convergence follows directly from the construction.
\begin{theorem}\label{thm:newton:superlinear}
    Let $X,Y$ be Banach spaces and let $F:X\to Y$ be Newton differentiable in $\bar x\in X$ with $F(\bar x)=0$ with Newton derivative $D_N F(\bar x)$. Assume further that there exist $\delta>0$ and $C>0$ with $\norm{D_NF(x)^{-1}}_{L(Y,X)}\leq C$ for all $x\in O_\delta(\bar x)$. Then the semismooth Newton method \eqref{eq:SSN} converges to $\bar x$ for all $x^0$ sufficiently close to $\bar x$.
\end{theorem}
\begin{proof}
    The proof is virtually identical to that for the classical Newton method. We have already shown that for any $x^0\in O_\delta(\bar x)$,
    \begin{equation}\label{eq:newton:superlinear1}
        \norm{e^1}_X \leq C \norm{F(\bar x+e^0) - F(\bar x) - D_NF(\bar x+e^0)e^0}_Y.
    \end{equation}
    Let now $\eps\in(0,1)$ be arbitrary. The Newton differentiability of $F$ then implies that there exists a $\rho>0$ such that
    \begin{equation*}
        \norm{F(\bar x+h) - F(\bar x) - D_NF(\bar x+h)h}_Y \leq \frac\eps{C}\norm{h}_X \qquad\text{for all }\norm{h}_X\leq \rho.
    \end{equation*}
    Hence, if we choose $x^0$ such that $\norm{\bar x- x^0}_X\leq \min\{\delta,\rho\}$, the estimate \eqref{eq:newton:superlinear1} implies that  $\norm{\bar x- x^1}_X\leq \eps\norm{\bar x- x^0}_X$. By induction, we obtain from this that $\norm{\bar x-x^k}_X\leq \eps^k \norm{\bar x- x^0}_X \to 0$. Since $\eps\in(0,1)$ was arbitrary, we can take in each step $k$ a different $\eps_k\to 0$, which shows that the convergence is in fact superlinear.
\end{proof}

\section{Newton derivatives}

The remainder of this chapter is dedicated to the construction of Newton derivatives (although it should be pointed out that the verification of the approximation condition is usually the much more involved step in practice).
We begin with the obvious connection with the Fréchet derivative.
\begin{theorem}\label{thm:newton:frechet}
    If $F:X\to Y$ is continuously differentiable in $x\in X$, then $F$ is also Newton differentiable in $x$ with Newton derivative $D_N F(x) = F'(x)$.
\end{theorem}
\begin{proof}
    We have for arbitrary $h\in X$ that
    \begin{equation*}
        \begin{aligned}
            \norm{F(x+h)-F(x)-F'(x+h)h}_Y &\leq \norm{F(x+h)-F(x)-F'(x)h}_Y\\
            \MoveEqLeft[-1]+ \norm{F'(x)-F'(x+h)}_{L(X,Y)}\norm{h}_X,
        \end{aligned}
    \end{equation*}
    where the first summand is $o(\norm{h}_X)$ by definition of the Fréchet derivative and the second by the continuity of $F'$.
\end{proof}
Calculus rules can be shown similarly to those for Fréchet derivatives. For the sum rule this is immediate; here we prove a chain rule by way of example.
\begin{theorem}\label{thm:newton:chain}
    Let $X$, $Y$, and $Z$ be Banach spaces, and let $F:X\to Y$ be Newton differentiable in $x\in X$ with Newton derivative $D_N F(x)$ and $G:Y\to Z$ be Newton differentiable in $y:=F(x)\in Y$ with Newton derivative $D_N G(y)$.
    If $D_N F$ and $D_N G$ are uniformly bounded in a neighborhood of $x$ and $y$, respectively, then $G\circ F$ is also Newton differentiable in $x$ with Newton derivative
    \begin{equation*}
        D_N (G\circ F)(x) = D_N G(F(x))\circ D_N F(x).
    \end{equation*}
\end{theorem}
\begin{proof}
    We proceed as in the proof of \cref{thm:frechet_chain}. For $h\in X$ and  $g := F(x+h)-F(x)$ we have that
    \begin{equation*}
        (G\circ F)(x+h) - (G\circ F)(x) = G(y+g)  - G(y).
    \end{equation*}
    The Newton differentiability of $G$ then implies that
    \begin{equation*}
        \norm{(G\circ F)(x+h) - (G\circ F)(x) - D_N G(y+g)g}_Z =  r_1(\norm{g}_Y)
    \end{equation*}
    with $r_1(t)/t \to 0$ for $t\to 0$.
    The Newton differentiability of $F$ further implies that
    \begin{equation*}
        \norm{g - D_N F(x+h)h}_Y = r_2(\norm{h}_X)
    \end{equation*}
    with $r_2(t)/t \to 0$ for $t\to 0$. In particular,
    \begin{equation*}
        \norm{g}_Y \leq \norm{D_N F(x+h)}_{L(X,Y)}\norm{h}_Y + r_2(\norm{h}_X).
    \end{equation*}
    The uniform boundedness of $D_N F$ now implies that $\norm{g}_Y\to 0$ for $\norm{h}_X\to 0$. Hence,
    \begin{multline*}
        \norm{(G\circ F)(x+h) - (G\circ F)(x) -  D_N G(F(x+h))D_NF(x+h)h}_Z \\
        \begin{aligned}[t]
            &\leq \norm{G(y+g)-G(y)-D_NG(y+g)g}_Z\\
            \MoveEqLeft[-1]+ \norm{D_N G(y+g)\left[g-D_NF(x+h)h\right]}_Z\\
            &\leq r_1(\norm{g}_Y) +  \norm{D_N G(y+g)}_{L(Y,Z)} r_2(\norm{h}_X),
        \end{aligned}
    \end{multline*}
    and the claim thus follows from the uniform boundedness of $D_N G$.
\end{proof}

Finally, it follows directly from the definition of the product norm and Newton differentiability that Newton derivatives of vector-valued functions can be computed componentwise.
\begin{theorem}\label{thm:newton:vector}
    Let $X,Y_i$ be Banach spaces and let $F_i:X\to Y_i$ be Newton differentiable with Newton derivative $D_N F_i$ for $1\leq i\leq m$. Then
    \begin{equation*}
        F:X\to (Y_1\times\cdots\times Y_m), \qquad x\mapsto (F_1(x),\dots,F_m(x))^T,
    \end{equation*}
    is also Newton differentiable with Newton derivative
    \begin{equation*}
        D_N F(x) = (D_N F_1(x),\dots, D_N F_m(x))^T.
    \end{equation*}
\end{theorem}

\bigskip

Since the definition does not include a constructive prescription of Newton derivatives, the question remains how to obtain a candidate for which the approximation condition can be verified. For two classes of functions, such an explicit construction is known.

\subsection*{Locally Lipschitz continuous functions on \texorpdfstring{$\scriptstyle\R^N$}{ℝⁿ}}

If $F:\R^N\to \R$ is locally Lipschitz continuous, candidates can be taken from the Clarke subdifferential, which has an explicit characterization by \cref{thm:clarke:gradient}. Under some additional assumptions, each candidate is indeed a Newton derivative.\footnote{This is the original derivation of semismooth Newton methods.}

A function $F:\R^N\to\R$ is called \emph{piecewise (continuously) differentiable} or \emph{PC$^1$ function}, if
\begin{enumerate}[(i)]
    \item $F$ is continuous on $\R^N$;
    \item for all $x\in\R^N$ there exists an open neighborhood $U\subset \R^N$ of $x$ and a finite set $\{F_i:U\to \R\}_{i\in I}$ of continuously differentiable functions with
        \begin{equation*}
            F(\tilde x) \in \{F_i(\tilde x)\}_{i\in I}\qquad\text{for all }\tilde x\in U.
        \end{equation*}
\end{enumerate}
In this case, we call $F$ a \emph{measurable selection} of the $F_i$ in $U$. The set
\begin{equation*}
    I_a(x) := \setof{i\in I}{F(x) = F_i(x)}
\end{equation*}
is called the \emph{active index set} at $x$. Since the $F_i$ are continuous, we have that $F(\tilde x) \neq F_j(\tilde x)$ for all $j\notin I_a(x)$ and $\tilde x$ sufficiently close to $x$. Hence, indices that are only active on sets of zero measure do not have to be considered in the following. We thus define the \emph{essentially active index set}
\begin{equation*}
    I_e(x) := \setof{i\in I}{x\in \mathrm{cl}\left(\setof{\tilde x\in U}{F(\tilde x) = F_i(\tilde x)}^o\right)}\subset I_a(x).
\end{equation*}
An example of an active but not essentially active index set is the following.
\begin{example}
    Consider the function $f:\R\to\R$, $t\mapsto \max\{0,t,t/2\}$, i.e., $f_1(t)=0$, $f_2(t)=t$ and $f_3(t)=t/2$. Then $I_a(0)=\{1,2,3\}$ but $I_e(0)=\{1,2\}$, since $f_3$ is active only in $t=0$ and hence $\setof{t\in\R}{f(t)=f_3(t)}^o=\emptyset=\mathrm{cl}\,\emptyset$.
\end{example}

Since any $C^1$ function $F_i:U_x\to\R$ is Lipschitz continuous with Lipschitz constant $L_i:= \sup_{\tilde x\in U_x}|\nabla F(\tilde x)|$, PC$^1$ functions are always locally Lipschitz continuous; see \cite[Corollary 4.1.1]{Scholtes:2012}.
\begin{theorem}
    Let $F:\R^N\to\R$ be piecewise differentiable. Then $F$ is locally Lipschitz continuous in all $x\in\R^N$ with local constant $L(x)=\max_{i\in I_a(x)} L_i$.
\end{theorem}
This yields the following explicit characterization of the Clarke subdifferential of a PC$^1$ function.
\begin{theorem}\label{thm:newton:clarke}
    Let $F:\R^N\to\R$ be piecewise differentiable and $x\in\R^N$. Then
    \begin{equation*}
        \partial_C F(x) = \co\setof{\nabla F_i(x)}{i\in I_e(x)}.
    \end{equation*}
\end{theorem}
\begin{proof}
    Let $x\in\R^N$ be arbitrary. By \cref{thm:clarke:gradient} it suffices to show that
    \begin{equation*}
        \setof{\lim_{n\to\infty} \nabla F(x_n)}{x_n\to x,\ x_n\notin E_F} = \setof{\nabla F_i(x)}{i\in I_e(x)}.
    \end{equation*}
    For this, let $\{x_n\}_{n\in\N}\subset\R^N$ be a sequence with $x_n\to x$ such that $F$ is differentiable in $x_n$ for all $n\in \N$, and $\nabla F(x_n)\to x^* \in \R^N$. Since $F$ is differentiable in $x_n$, it must hold that $F(\tilde x) = F_{i_n}(\tilde x)$ for some $i_n\in I$ and all $\tilde x$ sufficiently close to $x_n$, which implies that $\nabla F(x_n) = \nabla F_{i_n} (x_n)$.
    For sufficiently large $n\in \N$, we can further assume that $i_n\in I_e(x)$ (if necessary, by adding $x_n$ with $i_n\notin I_e(x)$ to $E_F$, which does not increase its Lebesgue measure). If we now consider subsequences $\{x_{n_k}\}_{k\in\N}$ with constant index $i_{n_k}=:i\in I_e(x)$ (which exist since $I_e(x)$ is finite), we obtain using the continuity of $\nabla F_i$ that
    \begin{equation*}
        x^*=\lim_{k\to\infty} \nabla F(x_{n_k}) =\lim_{k\to\infty} \nabla F_i(x_{n_k})  \in \setof{\nabla F_i(x)}{i\in I_e(x)}.
    \end{equation*}

    Conversely, for every $\nabla F_i(x)$ with $i\in I_e(x)$ there exists by definition of the essentially active indices a sequence $\{x_n\}_{n\in\N}$ with $x_n\to x$ and $F=F_i$ in a sufficiently small neighborhood of each $x_n$ for $n$ large enough. The continuous differentiability of the $F_i$ thus implies that $\nabla F(x_n) = \nabla F_i(x_n)$ for all $n\in\N$ large enough and hence that
    \begin{equation*}
        \nabla F_i(x) = \lim_{n\to \infty} \nabla F_i(x_n) = \lim_{n\to \infty} \nabla F(x_n).
        \qedhere
    \end{equation*}
\end{proof}
From this, we obtain the Newton differentiability of PC$^1$ functions.
\begin{theorem}\label{thm:newton:clarke_ndiff}
    Let $F:\R^N\to\R$ be piecewise differentiable. Then $F$ is Newton differentiable for all $x\in\R^N$, and every $D_NF(x)\in \partial_C F(x)$ is a Newton derivative.
\end{theorem}
\begin{proof}
    Let $x\in \R^N$ be arbitrary and $h\in X$ with $x+h\in U$. By \cref{thm:newton:clarke}, every $D_N F(x+h)\in \partial_C F(x+h)$ is of the form
    \begin{equation*}
        D_N F(x+h) = \sum_{i\in I_e(x+h)} \lambda_i \nabla F_i(x+h)\qquad\text{for }\sum_{i\in I_e(x+h)} \lambda_i = 1, \lambda_i\geq 0.
    \end{equation*}
    Since $F$ is continuous, we have for all $h\in\R^N$ sufficiently small that $I_e(x+h)\subset I_a(x+h)\subset I_a(x)$, where the second inclusion follows from the fact that by continuity, $F(x)\neq F_i(x)$ implies that $F(x+h) \neq F_i(x+h)$. Hence, $F(x+h) = F_i(x+h)$ and $F(x) = F_i(x)$ for all $i\in I_e(x+h)$.
    \Cref{thm:newton:frechet} then yields that
    \begin{equation*}
        |F(x+h)-F(x)-D_N F(x+h)h| \leq \sum_{i\in I_e(x+h)} \lambda_i |F_i(x+h)-F_i(x) - \nabla F_i(x+h)h| = o(\norm{h}),
    \end{equation*}
    since all $F_i$ are continuously differentiable by assumption.
\end{proof}

A natural application of the above are proximal point reformulations of optimality conditions for convex optimization problems.
\begin{example}\label{ex:newton:rn}
    We consider the minimization of $F+G$ for a twice continuously differentiable functional $F:\R^N\to\R$ and $G = \norm{\cdot}_1$. Proceeding as in the derivation of the forward--backward splitting \eqref{eq:splitting:fb}, we can use the regularity of $F$ and $G$ to write the necessary optimality condition $0\in\partial_C (F+G)(\bar x)$ equivalently as
    \begin{equation*}
        \bar x - \prox_{\gamma G}(\bar x - \gamma \nabla F(\bar x))=0
    \end{equation*}
    for any $\gamma >0$.
    By \cref{ex:proximal:rn}\,(ii), the proximal point mapping for $G$ is given componentwise as
    \begin{equation*}
        [\prox_{\gamma G}(x)]_i =
        \begin{cases}
            x_i-\gamma & \text{if }x_i>\gamma,\\
            0 & \text{if }x_i \in[-\gamma,\gamma],\\
            x_i+\gamma & \text{if }x_i < -\gamma,
        \end{cases}
    \end{equation*}
    which is clearly piecewise differentiable.
    \Cref{thm:newton:clarke} thus yields (also componentwise) that
    \begin{equation*}
        [\partial_C (\prox_{\gamma G})(x)]_i =
        \begin{cases}
            \{1\} & \text{if }|x_i|>\gamma,\\
            \{0\} & \text{if }|x_i|<\gamma,\\
            [0,1] & \text{if }|x_i|=\gamma.
        \end{cases}
    \end{equation*}
    By \cref{thm:newton:clarke_ndiff,thm:newton:vector}, a possible Newton derivative is therefore given by
    \begin{equation*}
        [D_N \prox_{\gamma G}(x) h]_i = [\1_{\{|x|\geq \gamma\}}h]_i :=
        \begin{cases}
            h_i & \text{if }|x_i|\geq \gamma, \\
            0 & \text{if }|x_i|<\gamma.
        \end{cases}
    \end{equation*}
    (The choice which case to include the equality in is arbitrary here.)
    Now, $D_N \prox_{\gamma G}(x)$ and $D_N (\nabla F)(x) = \nabla^2 F(x)$ are locally uniformly bounded (obviously from the characterization and the continuous differentiability, respectively), and using the chain rule from \cref{thm:newton:chain} and rearranging yields the semismooth Newton step
    \begin{equation*}
        \left(\1_{\calI_k} + \gamma  \1_{\calA_k}\nabla^2F(x^k)\right)s^k = - x^k + \prox_{\gamma G}(x^k-\gamma \nabla F(x^k)),
    \end{equation*}
    where we have defined the \emph{active} and \emph{inactive sets}, respectively, as
    \begin{equation*}
        \calA_k := \setof{i\in\{1,\dots,N\}}{|x^k_i - \gamma [\nabla F(x^k)]_i|\geq \gamma},\qquad \calI_k := \{1,\dots,N\}\setminus \calA_k.
    \end{equation*}
    If we now also partition $s^k$ as well as the right-hand side in active and inactive components using the case distinction in the characterization of $\prox_{\gamma G}$ (which follows the same partition), we can rearrange this linear system into blocks corresponding to active and inactive components to observe that the Newton step coincides with an \emph{active set strategy}  similar to those used for solving quadratic subproblems in sequential programming methods with inequality constraints; cf.~\cite[Chapter 8.4]{Kunisch:2008a}.
\end{example}

\subsection*{Superposition operators on \texorpdfstring{$\scriptstyle L^p(\Omega)$}{Lᵖ(Ω)}}

Rademacher's Theorem does not hold in infinite-dimensional function spaces, and hence the Clarke subdifferential no longer yields an algorithmically useful candidate for a Newton derivative in general. One exception is the class of superposition operators defined by scalar Newton differentiable functions, for which the Newton derivative can be evaluated pointwise as well.

We thus again consider for an open and bounded domain $\Omega\subset \R^N$, a Carathéodory function $f:\Omega\times \R\to\R$ (i.e., $f$ is measurable in $x$ and continuous in $z$), and $1\leq p,q\leq \infty$ the corresponding superposition operator
\begin{equation*}
    F:L^p(\Omega)\to L^q(\Omega),\qquad [F(u)](x) = f(x,u(x))\quad\text{for almost every }x\in\Omega.
\end{equation*}
The goal is now to similarly obtain a Newton derivative $D_N F$ for $F$ as a superposition operator defined by the Newton derivative $D_N f(x,z)$ of $z\mapsto f(x,z)$. Here, the assumption that $D_N f$ is also a Carathéodory function is too restrictive, since we want to allow discontinuous derivatives as well (see \cref{ex:newton:rn}). Luckily, for our purpose, a weaker property is sufficient: A function is called \emph{Baire--Carathéodory function} if it can be written as a pointwise limit of Carathéodory functions, i.e., if
\begin{equation*}
    f(x,z) = \lim_{n\to \infty} f_n(x,z) \qquad\text{for almost every }x\in \Omega\text{ and all }z\in \R,
\end{equation*}
where $f_n$ is a Carathéodory function for all $n\in\N$; see \cite[Lemma 1.4]{Appell:1990}.

Under certain growth conditions on $f$ and $D_N f$,\footnote{which can be significantly relaxed; see \cite[Proposition \textsc{a}.1]{Schiela:2008a}} we can transfer the Newton differentiability of $f$ to $F$, but we again have to take a two norm discrepancy into account.
\begin{theorem}\label{thm:newton:super}
    Let $f:\Omega\times\R\to\R$ be a Carathéodory function. Furthermore, assume that
    \begin{enumerate}[(i)]
        \item $z\mapsto f(x,z)$ is uniformly Lipschitz continuous for almost every $x\in \Omega$ and $f(x,0)$ is bounded;
        \item $z\mapsto f(x,z)$ is Newton differentiable with Newton derivative $z\mapsto D_N f(x,z)$ for almost every $x\in \Omega$;
        \item $D_N f$ is a Baire--Carathéodory function and uniformly bounded.
    \end{enumerate}
    Then for any $1\leq q<p< \infty$, the corresponding superposition operator $F:L^p(\Omega)\to L^q(\Omega)$ is Newton differentiable with Newton derivative
    \begin{equation*}
        D_N F:L^p(\Omega)\to L(L^p(\Omega),L^q(\Omega)), \qquad
        [D_N F(u)h](x) = D_N f(x,u(x))h(x)
    \end{equation*}
    for almost every $x\in \Omega$ and all $h\in L^p(\Omega)$.
\end{theorem}
\begin{proof}
    First, the uniform Lipschitz continuity together with the reverse triangle inequality yields that
    \begin{equation*}
        |f(x,z)| \leq |f(x,0)|+L|z| \leq C +L|z|^{q/q}\quad\text{for almost every }x\in\Omega\text{ and all } z\in\R,
    \end{equation*}
    and hence the growth condition \eqref{eq:superpos:growth} is satisfied for all $1\leq q<\infty$. Due to the continuous embedding $L^p(\Omega)\hookrightarrow L^q(\Omega)$ for all $1\leq q<p<\infty$, the superposition operator $F:L^p(\Omega)\to L^q(\Omega)$ is therefore well-defined and continuous by \cref{thm:superpos:continuous}.

    For any measurable $u:\Omega\to\R$, we have that $x\mapsto D_N f(x,u(x))$ is by assumption (iii) the pointwise limit of measurable functions and hence itself measurable. Furthermore, its uniform boundedness in particular implies the growth condition \eqref{eq:superpos:growth} for $p':=p$ and $q':=p-q>0$. As in the proof of \cref{thm:superpos:differentiable}, we deduce that the corresponding superposition operator $D_N F:L^p(\Omega)\to L^s(\Omega)$ is well-defined and continuous for $s:=\frac{pq}{p-q}$, and that for any $u\in L^p(\Omega)$, the mapping $h\mapsto D_NF(u)h$ defines a bounded linear operator $D_NF(u):L^p(\Omega)\to L^q(\Omega)$.
    (This time, we do not distinguish in notation between the linear operator and the function defining this operator by pointwise multiplication.)

    To show that $D_N F(u)$ is a Newton derivative for $F$ in $u\in L^p(\Omega)$, we consider the pointwise residual
    \begin{equation*}
        r:\Omega\times\R\to\R,\qquad r(x,z) :=
        \begin{cases}
            \frac{|f(x,z)-f(x,u(x))-D_N f(x,z)(z-u(x))|}{|z-u(x)|} & \text{if }z\neq u(x),\\
            0 & \text{if }z=u(x).
        \end{cases}
    \end{equation*}
    Since $f$ is a Carathéodory function and $D_N f$ is a Baire--Carathéodory function, the function $x\mapsto r(x,\tilde u(x))=:R(\tilde u)$ is measurable for any measurable $\tilde u:\Omega\to\R$ (since sums, products, and quotients of measurable functions are again measurable).
    Furthermore, for $\tilde u\in L^p(\Omega)$, the uniform Lipschitz continuity of $f$ and the uniform boundedness of $D_N f$ imply that
    \begin{equation}\label{eq:newton:superpos1}
        |[R(\tilde u)](x)| = \frac{|f(x,\tilde u(x))-f(x,u(x))-D_N f(x,\tilde u(x))(\tilde u(x)-u(x))|}{|\tilde u(x)-u(x)|}
        \leq L + C
    \end{equation}
    and thus that $R(\tilde u)\in L^\infty(\Omega)$. Hence, the superposition operator $R:L^p(\Omega)\to L^{s}(\Omega)$ is well-defined.

    Let now $\{u_n\}_{n\in\N}\subset L^p(\Omega)$ be a sequence with $u_n\to u\in L^p(\Omega)$. Then there exists a subsequence, again denoted by $\{u_n\}_{n\in\N}$, with $u_n(x)\to u(x)$ for almost every $x\in \Omega$. Since $z\mapsto f(x,z)$ is Newton differentiable almost everywhere, we have by definition that $r(x,u_n(x))\to 0$ for almost every $x\in\Omega$.
    Together with the boundedness from \eqref{eq:newton:superpos1}, Lebesgue's dominated convergence theorem therefore yields that $R(u_n)\to 0$ in $L^{s}(\Omega)$ (and hence along the full sequence since the limit is unique).\footnote{This step fails for $F:L^\infty(\Omega)\to L^\infty(\Omega)$ since pointwise convergence and boundedness together do not imply uniform convergence almost everywhere.}
    For any $\tilde u\in L^p(\Omega)$, the Hölder inequality with $\frac1p+\frac1{s}=\frac1q$ thus yields that
    \begin{equation*}
        \begin{aligned}
            \norm{F(\tilde u)-F(u)-D_NF(\tilde u)(\tilde u-u)}_{L^q} =
            \norm{R(\tilde u)(\tilde u-u)}_{L^q}
            \leq \norm{R(\tilde u)}_{L^{s}}\norm{\tilde u-u}_{L^p}.
        \end{aligned}
    \end{equation*}
    If we now set $\tilde u := u+h$ for $h\in L^p(\Omega)$ with $\norm{h}_{L^p}\to 0$, we have that $\norm{R(u+h)}_{L^{s}}\to 0$ and hence by definition the Newton differentiability of $F$ in $u$ with Newton derivative $h\mapsto D_N F(u)h$ as claimed.
\end{proof}
For $p=q\in[1,\infty]$, however, the claim is false in general, as can be shown by counterexamples.
\begin{example}
    We take
    \begin{equation*}
        f:\R\to\R,\qquad f(z) = \max\{0,z\} :=
        \begin{cases}
            0 &\text{if }z\leq 0, \\
            z &\text{if }z\geq 0.
        \end{cases}
    \end{equation*}
    This is a piecewise differentiable function, and hence by \cref{thm:newton:clarke_ndiff} we can for any $\delta\in[0,1]$ take as Newton derivative
    \begin{equation*}
        D_N f(z)h =
        \begin{cases}
            0 & \text{if }z<0,\\
            \delta h&\text{if } z=0,\\
            h & \text{if }z>0.
        \end{cases}
    \end{equation*}
    We now consider the corresponding superposition operators $F:L^p(\Omega)\to L^p(\Omega)$ and $D_N F(u)\in L(L^p(\Omega);L^p(\Omega))$ for any $p\in [1,\infty)$ and show that the approximation condition \eqref{eq:newton:semismooth} is violated for $\Omega=(-1,1)$, $u(x) = -|x|$, and
    \begin{equation*}
        h_n(x) =
        \begin{cases}
            \frac1n & \text{if } |x| < \frac1n,\\
            0 & \text{if } |x|\geq \frac1n.
        \end{cases}
    \end{equation*}
    First, it is straightforward to compute $\norm{h_n}_{L^p}^p = \frac2{n^{p+1}}$. Then since  $[F(u)](x) = \max\{0,-|x|\} = 0$ almost everywhere, we have that
    \begin{equation*}
        [F(u+h_n) - F(u) - D_N F(u+h_n)h_n](x) =
        \begin{cases}
            - |x| & \text{if }|x| < \frac1n,\\
            0    & \text{if }|x| > \frac1n,\\
            -\frac\delta{n} & \text{if } |x| = \frac1n,
        \end{cases}
    \end{equation*}
    and thus
    \begin{equation*}
        \norm{F(u+h_n) - F(u) - D_N F(u+h_n)h_n}^p_{L^p} = \int_{-\frac1n}^{\frac1n} |x|^p\,dx = \frac2{p+1}\left(\frac1n\right)^{p+1}.
    \end{equation*}
    This implies that
    \begin{equation*}
        \lim_{n\to\infty} \frac{\norm{F(u+h_n) - F(u) - D_N F(u+h_n)h_n}_{L^p}}{\norm{h_n}_{L^p}} = \left(\frac1{p+1}\right)^{\frac1p} \neq 0
    \end{equation*}
    and hence that $F$ is not Newton differentiable from $L^p(\Omega)$ to $L^p(\Omega)$ for any $p<\infty$.

    \bigskip

    For the case $p=q=\infty$, we take $\Omega=(0,1)$, $u(x) = x$, and
    \begin{equation*}
        h_n(x) =
        \begin{cases}
            nx-1 & \text{if }x\leq \frac1n,\\
            0 & \text{if }x\geq \frac1n,
        \end{cases}
    \end{equation*}
    such that $\norm{h_n}_{L^\infty}=1$ for all $n\in\N$. We also have that $x + h_n = (1+n)x-1\leq 0$ for $x\leq \frac{1}{n+1}\leq \frac1n$ and hence that
    \begin{equation*}
        [F(u+h_n) - F(u) - D_N F(u+h_n)h_n](x) =
        \begin{cases}
            (1+n)x-1 & \text{if }x \leq \frac1{n+1},\\
            0    & \text{if }x \geq \frac1{n+1}
        \end{cases}
    \end{equation*}
    since either $h_n=0$ or $F(u+h_n)=F(u)+D_NF(u)h_n$ in the second case. Now,
    \begin{equation*}
        \sup_{x\in(0,\frac{1}{n+1}]}|(1+n)x-1| = 1 \qquad\text{for all }n\in\N,
    \end{equation*}
    which implies that
    \begin{equation*}
        \lim_{n\to\infty} \frac{\norm{F(u+h_n) - F(u) - D_N F(u+h_n)h_n}_{L^p}}{\norm{h_n}_{L^p}} = 1 \neq 0
    \end{equation*}
    and hence that $F$ is not Newton differentiable from $L^\infty(\Omega)$ to $L^\infty(\Omega)$ either.
\end{example}

Due to the two norm discrepancy, we can no longer apply the semismooth Newton method directly to proximal point reformulations in function spaces. We therefore have to fall back on the Moreau--Yosida regularization.
\begin{example}
    We consider as in \cref{ex:newton:rn} the minimization of $F+G$ for a twice continuously differentiable functional $F:L^2(\Omega)\to \R$ and $G=\norm{\cdot}_{L^1}$. The proximal point reformulation of $0\in \partial (F+G)(\bar u)$,
    \begin{equation*}
        \bar u - \prox_{\gamma G}(\bar u - \gamma \nabla F(\bar u))=0,
    \end{equation*}
    now has to be considered as an equation in $L^2(\Omega)$; however, $\prox_{\gamma G}$ is \emph{not} Newton differentiable from $L^2(\Omega)$ to $L^2(\Omega)$. We therefore replace in the original optimality conditions
    \begin{equation*}
        \left\{\begin{aligned}
                -\bar p &= \nabla F(\bar u),\\
                \bar u &\in \partial G^*(\bar p),
        \end{aligned}\right.
    \end{equation*}
    the subdifferential of $G^*$ with its Moreau--Yosida regularization $H_\gamma:=(\partial G^*)_\gamma$, which by \cref{lem:lebesgue:proximal} and \cref{ex:moreau} is given pointwise as  $[H_\gamma(p)](x) = h_\gamma(p(x))$ for
    \begin{equation*}
        h_\gamma:\R\to\R,\qquad t\mapsto
        \begin{cases}
            \frac1\gamma (t - 1) & \text{if }t>1,\\
            0                & \text{if }t \in[-1,1],\\
            \frac1\gamma (t + 1) & \text{if }t<-1.
        \end{cases}
    \end{equation*}
    This function is clearly piecewise differentiable, and \cref{thm:newton:clarke} yields that
    \begin{equation*}
        \partial_C h_\gamma(t) =
        \begin{cases}
            \left\{\tfrac1\gamma\right\} & \text{if }|t|>1,\\
            \{0\} & \text{if }|t|<1,\\
            \left[0,\tfrac1\gamma\right] & \text{if }|t|=1.
        \end{cases}
    \end{equation*}
    By \cref{thm:newton:clarke_ndiff,thm:newton:vector}, a possible Newton derivative is therefore given by
    \begin{equation*}
        D_N h_\gamma(t)h  = \tfrac1\gamma \1_{\{|t|\geq 1\}}h :=
        \begin{cases}
            \frac1\gamma h &\text{if } |t| \geq 1,\\
            0 &\text{if } |t| < 1.
        \end{cases}
    \end{equation*}
    The function $D_N h_\gamma$ is now uniformly bounded (by $\frac1\gamma$) and can be approximated by the obvious pointwise limit of continuous functions. By \cref{thm:newton:super}, the superposition operator $H_\gamma:L^p(\Omega)\to L^2(\Omega)$ is therefore Newton differentiable for all $p>2$, and a possible Newton derivative is given by
    \begin{equation*}
        [D_NH_\gamma(p)h](x) = \tfrac1\gamma \1_{\{|p(x)|\geq 1\}}h(x),
    \end{equation*}
    Assume now that $F$ is such that $\bar p = -\nabla F(\bar u) \in L^p(\Omega)$ for some $p>2$. (This is the case, e.g., if $F$ involves the solution operator to a partial differential equation.) Then the reduced regularized optimality condition
    \begin{equation*}
        u_\gamma - H_\gamma(-\nabla F(u_\gamma))=0
    \end{equation*}
    is Newton differentiable by \cref{thm:newton:frechet,thm:newton:chain},
    and we arrive at the semismooth Newton step
    \begin{equation*}
        \left(\Id + \tfrac1\gamma \1_{\{|\nabla F(u^k)|\geq 1\}}\nabla^2 F(u^k)\right)s^k  = -u^k + H_\gamma(-\nabla F(u^k)),
    \end{equation*}
    where in a slight abuse of notation, $\1_{\{|p|\geq 1\}}$ denotes the function $x\mapsto \1_{\{|p(x)|\geq 1\}}$.

    In practice, the radius of convergence for semismooth Newtons applied to such a Moreau--Yosida regularization shrinks with $\gamma\to0$. A possible way of dealing with this is the following \emph{continuation strategy}: Starting with a sufficiently large value of $\gamma$, solve a sequence of problems with decreasing $\gamma$ (e.g., $\gamma^k = \gamma^0/2^k$), taking the solution of the previous problem as the starting point for the next (for which it hopefully close enough to the solution to lie within the convergence region; otherwise the continuation has to be terminated or the reduction strategy for $\gamma$ adapted).
\end{example}

\chapter{Limiting subdifferentials}
\label{chap:limiting}

While the Clarke subdifferential is a suitable concept for nonsmooth but convex or nonconvex but smooth functionals, it has severe drawbacks for nonsmooth \emph{and} nonconvex functionals: As shown in \cref{lem:clarke:fermat2}, its Fermat principle cannot distinguish minimizers from maximizers. The reason is that the Clarke subdifferential is always convex, which is a direct consequence of its construction \eqref{eq:clarke:def} via polarity with respect to (generalized) directional derivatives. To obtain sharper results for such functionals, it is therefore necessary to construct \emph{nonconvex} subdifferentials directly via a \emph{dual} limiting process. On the other hand, deriving calculus rules for the previous subdifferentials crucially exploited their convexity by applying Hahn--Banach separation theorems, and calculus rules for nonconvex subdifferentials are thus significantly more difficult to obtain.
As in \cref{chap:clarke}, we will assume throughout this chapter that $X$ is a Banach space unless stated otherwise.

\section{Bouligand subdifferentials}

The first definition is motivated by \cref{thm:clarke:gradient}: We \emph{define} a subdifferential as a suitable limit of classical derivatives (without convexification). For $F:X\to \Rbar$, we first define the \emph{set of Gâteaux points}
\begin{equation*}
    G_F := \setof{x\in X}{F \text{ is Gâteaux differentiable at }x} \subset \dom F
\end{equation*}
and then the \emph{Bouligand subdifferential} of $F$ at $x$ as
\begin{equation}\label{eq:limiting:bouligand}
    \partial_B F(x) :=\setof{x^*\in X^*}{DF(x_n)\weakto^* x^* \text{ for some } G_F\ni x_n\to x}.
\end{equation}
For $F:\R^N\to \R$ locally Lipschitz, it follows from \cref{thm:clarke:gradient} that $\partial_C F(x) = \co \partial_B F(x)$.
However, unless $X$ is finite-dimensional, it is not clear a priori that the Bouligand subdifferential is nonempty even for $x\in \dom F$.\footnote{Although in special cases it is possible to give a full characterization in Hilbert spaces; see, e.g., \cite{CCMW:2017}.} Furthermore, the subdifferential does not admit a satisfactory calculus; not even a Fermat principle holds.

\begin{example}\label{ex:limiting:bouligand}
    Let $F:\R\to\R$, $F(x) := |x|$. Then $F$ is differentiable at every $x\neq 0$ with $F'(x) = \sign(x)$. Correspondingly,
    \begin{equation*}
        0\notin \{-1,1\} = \partial_B F(0).
    \end{equation*}
\end{example}

To make this approach work therefore requires a more delicate limiting process. The remainder of this chapter is devoted to one such approach, where we only give an overview and state important results following \cite{Mordukhovich:2006}. For an alternative, more axiomatic, approach to generalized derivatives of nonconvex functionals, we refer to \cite{Penot:2013,ioffe2017variational}.

\section{Fréchet subdifferentials}
\label{sec:limiting:frechet}

We begin with the following limiting construction, which combines the characterizations of both the Fréchet derivative and the convex subdifferential. Let $X$ be a Banach space and $F:X\to\Rbar$. The \emph{Fréchet subdifferential} (or \emph{regular subdifferential} or \emph{presubdifferential}) of $F$ at $x$ is then defined as\footnote{The equivalence of \eqref{eq:limiting:frechet} with the usual definition based on corresponding normal cones follows from, e.g., \cite[Theorem 1.86]{Mordukhovich:2006}.}
\begin{equation}\label{eq:limiting:frechet}
    \partial_F F(x) := \setof{x^*\in X^*}{\liminf_{y\to x} \frac{F(y)-F(x)-\dual{x^*,y-x}_X}{\norm{y-x}_X}\geq 0}.
\end{equation}
Note how this \enquote{localizes} the definition of the convex subdifferential around the point of interest: the numerator does not need to be nonnegative for all $y$; it suffices if this holds for any $y$ sufficiently close to $x$. By a similar argument as for \cref{thm:convex:fermat}, we thus obtain a Fermat principle for \emph{local} minimizers.
\begin{theorem}\label{thm:limiting:frechet:fermat}
    Let $F:X\to\Rbar$ be proper and $\bar x \in \dom F$ be a local minimizer. Then $0\in \partial_F F(\bar x)$.
\end{theorem}
\begin{proof}
    Let $\bar x\in \dom F$ be a local minimizer. Then there exists an $\eps>0$ such that $F(\bar x)\leq F(y)$ for all $y\in O_\eps(\bar x)$, which is equivalent to
    \begin{equation*}
        \frac{F(y)-F(\bar x)-\dual{0,y-\bar x}_X}{\norm{y-\bar x}_X}\geq 0 \quad\text{for all }y\in O_\eps(\bar x)\setminus\{\bar x\}.
    \end{equation*}
    Now for any strongly convergent sequence $y_n\to \bar x$, we have that $y_n\in O_\eps(\bar x)$ for $n$ large enough. Taking the $\liminf$ in the above inequality thus yields $0\in \partial_F(\bar x)$.
\end{proof}

For convex functionals, of course, the numerator is always nonnegative by definition, and the Fréchet subdifferential reduces to the convex subdifferential.
\begin{theorem}\label{thm:limiting:frechet:convex}
    Let $F:X\to\Rbar$ be proper, convex, and lower semicontinuous and $x\in \dom F$. Then $\partial_F F(x) = \partial F(x)$.
\end{theorem}
\begin{proof}
    By definition of the convex subdifferential, any $x^*\in\partial F(x)$ satisfies
    \begin{equation*}
        F(y)-F(x)-\dual{x^*,y-x}_X \geq 0 \quad\text{for all } y\in X.
    \end{equation*}
    Dividing by $\norm{x-y}_X>0$ for $y\neq x$ and taking the $\liminf$ as $y\to x$ thus yields $x^*\in \partial_F F(x)$.

    Conversely, let $x^*\in \partial_F F(x)$ and $h\in X\setminus\{0\}$ be arbitrary. Then for any $\delta >0$, there exists an $\eps >0$ such that
    \begin{equation*}
        \frac{F(x+th)-F(x)-\dual{x^*,th}_X}{t\norm{h}_X} \geq -\delta \quad\text{for all }t\in (0,\eps).
    \end{equation*}
    Multiplying by $\norm{h}_X>0$ and letting $t\to 0$, we obtain from \cref{lem:convex:direct} that
    \begin{equation}
        \dual{x^*,h}_X \leq \frac{F(x+th)-F(x)}{t} + \delta  \to F'(x; h) + \delta.
    \end{equation}
    Since $\delta>0$ was arbitrary, this implies by \cref{lem:convex:equiv} that $x^*\in \partial F(x)$.
\end{proof}

Similarly, for Fréchet differentiable functionals, the limit in \eqref{eq:limiting:frechet} is zero for all sequences.
\begin{theorem}\label{lem:limiting:frechet:frechet}
    Let $F:X\to\R$ be Fréchet differentiable at $x\in X$. Then $\partial_F F(x) = \{F'(x)\}$.
\end{theorem}
\begin{proof}
    The definition of the Fréchet derivative immediately yields
    \begin{equation*}
        \lim_{y\to x} \frac{F(y)-F(x)-\dual{F'(x),y-x}_X}{\norm{x-y}_X} = \lim_{\norm{h}_X\to 0} \frac{F(x+h)-F(x)-F'(x)h}{\norm{h}_X} = 0
    \end{equation*}
    and hence $F'(x)\in \partial_F F(x)$.

    Conversely, let $x^*\in \partial_F F(x)$ and let again $h\in X\setminus\{0\}$ be arbitrary. As in the proof of \cref{thm:limiting:frechet:convex}, we then obtain that
    \begin{equation}
        \dual{x^*,h}_X \leq F'(x; h) = \dual{F'(x),h}_X.
    \end{equation}
    Applying the same argument to $-h$ then yields $\dual{x^*,h}_X = \dual{F'(x),h}_X$ for all $h\in X$, i.e., $x^* = F'(x)$.
\end{proof}

For nonsmooth and nonconvex functionals, the Fréchet subdifferential can be strictly smaller than the Clarke subdifferential.
\begin{example}\label{ex:limiting:frechet}
    Consider $F:\R\to\R$, $F(x) := -|x|$. For any $x\neq 0$, it follows from \cref{lem:limiting:frechet:frechet} that $\partial_F F(x) = \{-\sign x\}$. But for $x=0$ and arbitrary $x^*\in \R$, we have that
    \begin{equation*}
        \liminf_{y\to 0} \frac{F(y)-F(0)-\dual{x^*,y-0}}{|y-0|} = \liminf_{y\to 0} (-1-x^*\cdot \sign(y)) = -1-|x^*| < 0
    \end{equation*}
    and hence that
    \begin{equation*}
        \partial_F F(0) = \emptyset \subsetneq [-1,1] = \partial_C F(0).
    \end{equation*}
\end{example}
Note that $0\in \dom F$ in this example. Although the Fréchet subdifferential does not pick up a maximizer in contrast to the Clarke subdifferential, the fact that $\partial_F F(x)$ can be empty even for $x\in \dom F$ is a problem when trying to derive calculus rules that hold with equality. In fact, as \cref{ex:limiting:frechet} shows, the set-valued mapping $x\mapsto \partial_F F(x)$ fails to be closed, which is also not desirable. This leads to the next and final definition.

\section{Mordukhovich subdifferentials}
\label{sec:limiting:mordukhovich}

Let $X$ be a reflexive Banach space and $F:X\to \Rbar$. The \emph{Mordukhovich subdifferential} (or \emph{basic subdifferential} or \emph{limiting subdifferential}) of $F$ at $x\in \dom F$ is then defined as the strong-to-weak$^*$ closure of $\partial_F F(x)$, i.e.,%
\footnote{The equivalence of this definition with the original geometric definition -- which holds in \emph{reflexive} Banach spaces -- follows from \cite[Theorem 2.34]{Mordukhovich:2006}.}
\begin{equation}\label{eq:limiting:mordukhovich}
    \begin{aligned}[t]
        \partial_M F(x) &:= \text{w-$*$-}\limsup_{y\to x}\partial_F F(y) \\
        &= \setof{x^*\in X^*}{x_n^* \weakto^* x^* \text{ for } x_n^*\in \partial_F F(x_n) \text{ with }x_n\to x},
    \end{aligned}
\end{equation}
which can be seen as a generalization of the definition \eqref{eq:limiting:bouligand} of the Bouligand subdifferential. Note that in contrast to \eqref{eq:limiting:bouligand}, this definition includes the constant sequence $x_n^* \equiv x^*$ even at nondifferentiable points, which makes this a more useful concept in general. This also implies that $\partial_F F(x) \subset \partial_M F(x)$ for any $F$, and \cref{thm:limiting:frechet:fermat} immediately yields a Fermat principle.
\begin{cor}\label{thm:limiting:mordukhovich}
    Let $F:X\to\Rbar$ be proper and $\bar x \in \dom F$ be a local minimizer. Then $0\in \partial_M F(\bar x)$.
\end{cor}
As for the Fréchet subdifferential, maximizers do not satisfy the Fermat principle.
\begin{example}
    Consider again $F:\R\to\R$, $F(x) := -|x|$. Using \cref{ex:limiting:frechet}, we directly obtain from \eqref{eq:limiting:mordukhovich} that $\partial_M F(0) = \{-1,1\} = \partial_B F(0)$.
\end{example}

Since the convex subdifferential is strong-to-weak$^*$ closed, the Mordukhovich subdifferential reduces to the convex subdifferential as well.
\begin{theorem}
    Let $X$ be a reflexive Banach space, $F:X\to\Rbar$ be proper, convex, and lower semicontinuous, and $x\in \dom F$. Then $\partial_M F(x) = \partial F(x)$.
\end{theorem}
\begin{proof}
    From \cref{thm:limiting:frechet:convex}, it follows that $\partial F(x) = \partial_F F(x) \subset \partial_M F(x)$. Let therefore $x^*\in \partial_M F(x)$ be arbitrary. Then by definition there exists a sequence $\{x_n^*\}_{n\in\N}\subset X^*$ with $x_n^*\weakto^* x^*$ and $x_n^*\in \partial_F F(x_n)=\partial F(x_n)$ for $x_n\to x$. As in the proof of \cref{cor:monoton:subdiff}, it then follows that $x^*\in \partial F(x)$ as well.
\end{proof}
A similar result holds for \emph{continuously} differentiable functionals.
\begin{theorem}
    Let $X$ be a reflexive Banach space and $F:X\to\Rbar$ be continuously differentiable at  $x\in X$. Then $\partial_M F(x) = \{F'(x)\}$.
\end{theorem}
\begin{proof}
    From \cref{thm:limiting:frechet:convex}, it follows that $\{F'(x)\} = \partial_F F(x) \subset \partial_M F(x)$. Let therefore $x^*\in \partial_M F(x)$ be arbitrary. Then by definition there exists a sequence $\{x_n^*\}_{n\in\N}\subset X^*$ with $x_n^*\weakto^* x^*$ and $x_n^*\in \partial_F F(x_n) =\{F'(x_n)\}$ for $x_n\to x$. The continuity of $F'$ then immediately implies that $F'(x_n)\to F(x)$, and since strong limits are also weak-$*$ limits, we obtain $x^*=F'(x)$.
\end{proof}

We also have the following relation to Clarke subdifferentials, which should be compared to \cref{thm:clarke:gradient}.

\begin{theorem}[\protect{\cite[Theorem 3.57]{Mordukhovich:2006}}]\label{thm:limiting:clarke}
    Let $X$ be a reflexive Banach space and $F:X\to\R$ be locally Lipschitz continuous around $x\in X$. Then $\partial_C F(x) = \mathrm{cl}^* \co \partial_M F(x)$, where $\mathrm{cl}^*A $ stands for the weak-$*$ closure of the set $A\subset X^*$.%
    \footnote{Of course, in reflexive Banach spaces the weak-$*$ closure coincides with the weak closure. The statement holds more general in so-called \emph{Asplund spaces} which include some nonreflexive Banach spaces.}
\end{theorem}
The following example illustrates that the Mordukhovich subdifferential can be nonconvex.
\begin{example}
    Let $F:\R^2\to\R$, $F(x_1,x_2) = |x_1| - |x_2|$. Since $F$ is continuously differentiable for any $(x_1,x_2)$ where $x_1,x_2\neq 0$ with
    \begin{equation*}
        \nabla F(x_1,x_2)\in \{(1,1),(-1,1),(1,-1),(-1,-1)\},
    \end{equation*}
    we obtain from \eqref{eq:limiting:frechet} that
    \begin{equation*}
        \partial_F F(x_1,x_2) =
        \begin{cases}
            \{(1,-1)\} & \text{if } x_1>0,x_2>0,\\
            \{(-1,-1)\} & \text{if } x_1<0,x_2>0,\\
            \{(-1,1)\} & \text{if } x_1<0,x_2<0,\\
            \{(1,1)\} & \text{if } x_1>0,x_2<0,\\
            \setof{(t,-1)}{t\in [-1,1]} & \text{if } x_1=0,x_2>0,\\
            \setof{(t,1)}{t\in [-1,1]} & \text{if } x_1=0,x_2<0,\\
            \emptyset & \text{if } x_2=0.
        \end{cases}
    \end{equation*}
    In particular, $\partial_F F(0,0) = \emptyset$. However, from \eqref{eq:limiting:mordukhovich} it follows that
    \begin{equation*}
        \partial_M F(0,0) = \setof{(t,-1)}{t\in [-1,1]}\cup \setof{(t,1)}{t\in [-1,1]}.
    \end{equation*}
    In particular, $0\notin \partial_M F(0,0)$.
    On the other hand, \cref{thm:limiting:clarke} then yields that
    \begin{equation}
        \partial_C F(0,0) = \setof{(t,s)}{t,s\in[-1,1]} = [-1,1]^2
    \end{equation}
    and hence $0\in \partial_C F(0,0)$.
    (Note that $F$ attains neither a minimum nor a maximum on $\R^2$, while $(0,0)$ is a nonsmooth saddle-point.)
\end{example}

In contrast to the Bouligand subdifferential, the Mordukhovich subdifferential admits a satisfying calculus, although the assumptions are understandably more restrictive than in the convex setting. The first rule follows as always straight from the definition.

\begin{theorem}\label{thm:limiting:scalar}
    Let $X$ be a reflexive Banach space and $F:X\to\Rbar$. Then for any $\lambda\geq 0$ and $x\in X$,
    \begin{equation*}
        \partial_M (\lambda F)(x) = \lambda \partial_M F(x).
    \end{equation*}
\end{theorem}

Full calculus in infinite-dimensional spaces holds only for a rather small class of mappings.

\begin{theorem}[{\protect\cite[Proposition 1.107]{Mordukhovich:2006}}]\label{thm:limiting:sum}
    Let $X$ be a reflexive Banach space, $F:X\to\R$ be continuously differentiable, and $G:X\to\Rbar$ be arbitrary. Then for any $x\in \dom G$,
    \begin{equation*}
        \partial_M (F+G)(x) = \{F'(x)\} + \partial_M G(x).
    \end{equation*}
\end{theorem}
While the previous two theorems also hold for the Fréchet subdifferential (the latter even for merely Fréchet differentiable $F$), the following chain rule is only valid for the Mordukhovich subdifferential. Compared to \cref{thm:clarke:chain}, it also allows for the outer functional to be extended-real valued.
\begin{theorem}[{\protect\cite[Proposition 1.112]{Mordukhovich:2006}}]\label{thm:limiting:chain}
    Let $X$ be a reflexive Banach space, $F:X\to Y$ be continuously differentiable, and $G:Y\to \Rbar$ be arbitrary. Then for any $x\in X$ with $F(x)\in \dom G$ and $F'(x):X\to Y$ surjective,
    \begin{equation*}
        \partial_M (G\circ F)(x) = F'(x)^*\partial_M G(F(x)).
    \end{equation*}
\end{theorem}
More general calculus rules require $X$ to be a reflexive Banach\footnote{or Asplund} space as well as additional, nontrivial, assumptions on $F$ and $G$; see, e.g., \cite[Theorem 3.36 and Theorem 3.41]{Mordukhovich:2006}.

\backmatter

\printbibliography

\end{document}